\documentclass[12pt, oneside]{gsm-l}
\usepackage[scaled=.98,space]{erewhon}
\usepackage[erewhon, smallerops, vvarbb]{newtxmath}

\usepackage{chifoot}

\usepackage{nicefrac}
\usepackage{mathtools}
\usepackage{geometry}
\expandafter\def\expandafter\normalsize\expandafter{%
    \normalsize
    \setlength\abovedisplayskip{4pt}
    \setlength\belowdisplayskip{4pt}
    \setlength\abovedisplayshortskip{4pt}
    \setlength\belowdisplayshortskip{4pt}
}
\usepackage{microtype}
\usepackage{bookmark}
\usepackage[capitalize,nameinlink,noabbrev]{cleveref}
\hypersetup{hidelinks}
\usepackage{enumitem}
\newcommand{\Onebb}{\mathbb 1}
\newcommand{\Cbb}{\mathbb C}
\newcommand{\Fbb}{\mathbb F}
\newcommand{\Nbb}{\mathbb N}
\newcommand{\Qbb}{\mathbb Q}
\newcommand{\Rbb}{\mathbb R}
\newcommand{\Sbb}{\mathbb S}
\newcommand{\Tbb}{\mathbb T}
\newcommand{\Zbb}{\mathbb Z}
\newcommand{\Acal}{\mathcal A}
\newcommand{\Bcal}{\mathcal B}
\newcommand{\Fcal}{\mathcal F}
\newcommand{\Kcal}{\mathcal K}
\newcommand{\Mcal}{\mathcal M}
\newcommand{\Ncal}{\mathcal N}
\newcommand{\Pcal}{\mathcal P}
\newcommand{\Qcal}{\mathcal Q}
\newcommand{\Scal}{\mathcal S}
\newcommand{\Tcal}{\mathcal T}
\newcommand{\Ucal}{\mathcal U}
\newcommand{\Cscr}{\mathscr C}
\newcommand{\Pscr}{\mathscr P}
\newcommand{\cfrak}{{\mathfrak c}}

\renewcommand{\to}{\rightarrow}

\newcommand{\into}{\hookrightarrow}

\newcommand{\normal}{\vartriangleleft}
\newcommand{\dd}{\,\text{d}}
\newcommand{\ddd}{\text{d}}
\newcommand{\Linfty}{L_\infty}
\newcommand{\LinftyS}{L_\infty^*}
\newcommand{\linfty}{\ell_\infty}
\newcommand{\linftyS}{\ell_\infty^*}
\newcommand{\Lone}{L_1}
\newcommand{\Ltwo}{L_2}
\newcommand{\Cp}{\Cscr^+}
\newcommand{\Means}{M}
\newcommand{\LIM}{\text{LIM}}
\newcommand{\TLIM}{\text{TLIM}}
\newcommand{\TRIM}{\text{TRIM}}
\newcommand{\TIM}{\text{TIM}}
\newcommand{\LUC}{\textit{LUC}}
\newcommand{\FC}{\text{FC}}
\newcommand{\VN}{VN(G)}
\newcommand{\AG}{A(G)}
\newcommand{\Ghat}{\widehat{G}}
\newcommand{\cl}{\text{c}\ell}
\DeclareMathOperator{\supp}{supp}

\DeclareMathOperator{\conv}{conv}
\DeclareMathOperator{\Sdiff}{\triangle}
\DeclareMathOperator*{\plim}{\mathit{p}-lim}
\DeclareMathOperator*{\qlim}{\mathit{q}-lim}
\def\mmath#1{\text{\scalebox{1.1}{$#1$}}}
\def\sub#1{_\mmath{#1}}
\newcommand{\ie}{\textit{i.e.}}
\newcommand{\pp}{pp.\ }

\DeclarePairedDelimiter{\card}{\#(}{)}
\DeclarePairedDelimiter{\Haar}{|}{|}
\DeclarePairedDelimiter{\paren}{(}{)}
\DeclarePairedDelimiter{\set}{\{}{\}}

\DeclarePairedDelimiter{\norm}{\|}{\|}
\DeclarePairedDelimiter{\net}{\{}{\}}
\DeclarePairedDelimiter{\seq}{\{}{\}_{n=1}^{\infty}}
\DeclarePairedDelimiter{\abs}{|}{|}
\DeclarePairedDelimiter{\inner}{\langle}{\rangle}

\theoremstyle{definition}
\newtheorem{debug}{NOT TO BE USED}[chapter]
\newtheorem{thm}[debug]{Theorem}
\newtheorem{lem}[debug]{Lemma}
\newtheorem{cor}[debug]{Corollary}

\newtheorem{prop}[debug]{Proposition}
\newtheorem{rem}[debug]{Remark}
\newtheorem{que}[debug]{Question}
\newtheorem{defn}[debug]{Definition}
\newtheorem{point}[debug]{}

\makeatletter
\renewcommand{\thechapter}{\Roman{chapter}}
\makeatother

\makeatletter
\def\@tocline#1#2#3#4#5#6#7{\relax
  \ifnum #1>\c@tocdepth 
  \else
    \par \addpenalty\@secpenalty\addvspace{#2}%
        \begingroup \hyphenpenalty\@M
    \@ifempty{#4}{%
      \@tempdima\csname r@tocindent\number#1\endcsname\relax
    }{%
      \@tempdima#4\relax
    }%
    \parindent\z@ \leftskip#3\relax \advance\leftskip\@tempdima\relax
    \rightskip\@pnumwidth plus4em \parfillskip-\@pnumwidth
    #5\leavevmode\hskip-\@tempdima #6\nobreak\relax
    \ifnum #1=0
    \hfil\hbox to\@pnumwidth{}
    \else
    \hfil\hbox to\@pnumwidth{\@tocpagenum{#7}}\fi
    \par
    \nobreak
    \endgroup
  \fi}
\makeatother

\begin{document}
\frontmatter
\thispagestyle{empty}
\begin{center}
\textbf{\huge Topological Invariant Means on Locally Compact Groups}
\end{center}

\vspace*{2pc}

\begin{center}
{\large \textbf{John Hopfensperger}}
\\ {\large May 2021}
\end{center}

\vspace*{2pc}
\begin{center}
{\large Ching Chou}
\\{\large Adviser}
\end{center}

\vspace*{8pc}

\begin{center}
{\large\noindent Dissertation submitted to the Faculty of the Graduate School

of the University at Buffalo, State University of New York,

in partial fulfillment of the requirements for the degree of

Doctor of Philosophy,

Department of Mathematics}
\end{center}

\clearpage
\thispagestyle{empty}
\vspace*{40pc}
\begin{center}
{\footnotesize \textcopyright\ John Hopfensperger 2021}
\end{center}

\clearpage
\thispagestyle{empty}
\vspace*{7pc}
\begin{center}
{\large For my sister.}
\end{center}

\clearpage
\thispagestyle{empty}

\chapter*{Abstract}
\thispagestyle{empty}
Suppose $G$ is an amenable locally compact group.
If $\{F_\gamma\} = \{F_\gamma\}_{\gamma\in\Gamma}$ is a F\o{}lner net for $G$,
associate it with the net $\{\Onebb_{F_\gamma} / |F_\gamma|\} \subset \Lone(G) \subset \LinftyS(G)$.
Thus, every accumulation point of $\{F_\gamma\}$ is a topological left-invariant mean on $G$.
The following are examples of results proved in the present thesis:
\begin{enumerate}
	\item There exists a F\o{}lner net which has as its accumulation points
	a set of $2^{2^{\kappa}}$ distinct topological left-invariant means on $G$,
	where $\kappa$ is the smallest cardinality of a covering of $G$ by compact subsets.

	\item If $G$ is unimodular and $\mu$ is a topological left-invariant mean on $G$,
	there exists a F\o{}lner net which has $\mu$ as its unique accumulation point.

	\item Suppose $L \subset G$ is a lattice subgroup.
	There is a natural bijection of the left-invariant means on $L$
	with the topological left-invariant means on $G$ if and only if $G/L$ is compact.

	\item Every topological left-invariant mean on $G$ is also topological right-invariant if and only if
	$G$ has precompact conjugacy classes.
\end{enumerate}
These results lie at the intersection of functional analysis with general topology.
Problems in this area can often be solved with standard tools when $G$ is $\sigma$-compact or metrizable,
but require more interesting arguments in the general case.
\clearpage

\tableofcontents

\mainmatter
\setcounter{page}{0}
\chapter*{Notation}\label{Notation}
\vspace*{-0.5pc}
{\large \textbf{Spaces}}

\medskip
\begin{center}
$\begin{array}{cl}
\Nbb                    &   \{0, 1, 2, \ldots\}\\
\Qbb                    &   \text{rational numbers}\\
\Cbb					&	\text{complex numbers}\\
\beta X 				&	\text{Stone-\v{C}ech compactification of } X\\
\Pscr(X)                &   \text{power set of } X\\
\linfty(X) 				&	\text{bounded functions on } X\\
C(X) 					&	\text{continuous bounded functions on } X\\
G 						& \text{locally compact Hausdorff group}\\
\LUC(G)					&	\text{left-uniformly continuous bounded functions on } G\\
\Lone(G)				&	\text{Haar integrable functions on } G\\
\Linfty(G)				&	\text{dual of } \Lone(G)\\
P_1(G) 					&   \{f\in \Lone(G) : f\geq 0 \text{ and } \|f\|_1 = 1\}\\
\Means(G) 				&   \{\mu\in\LinftyS(G) : \mu\geq 0 \text{ and } \mu(1) = 1\}\\
\end{array}$
\end{center}

\medskip

{\large \textbf{Subsets of a group}}

\medskip

\begin{center}
$\begin{array}{cl}
\varnothing 				&   \text{empty set}\\
A \setminus B 			& \text{difference of sets}\\
A \Sdiff B				&	\text{symmetric difference of sets}\\
A \Subset B   			&	A \text{ is a finite nonempty subset of } B\\
A^{-1} 					& \{a^{-1} : a\in A\}\\
AB 						& \{ab : a\in A \text{ and } b\in B\}\\
\card{A}				&	\text{cardinality of } A\\
\Haar{A} 				&	\text{Haar measure of } A\\
\Onebb_{A}				&	\text{characteristic function of } A\\
\mu_A = \Onebb_A / \Haar{A}	&   \text{Normalized characteristic function of }A\\
\cl(A) 					&	\text{closure of } A
\end{array}$
\end{center}
\pagebreak
{\large \textbf{Translations}}

\medskip
\begin{center}
$\begin{array}{cl}
l_xf(y) = f(x^{-1}y)	&	[f \text{ defined on a group}]\\
L_xf(y) = f(xy)			&	[f \text{ defined on any semigroup or group}]\\
r_xf(y) = f(yx)			&	[f \text{ defined on any semigroup or group}]\\
\Delta  				&	\text{modular function}\\
f^\dag(x) = f(x^{-1}) \Delta(x^{-1})        &   \text{left-right involution } [f\in\Lone(G)]\\
R_t f = (L_t f^\dag)^\dag		&	[f\in\Lone(G)]\\
\end{array}$
\end{center}



\chapter{The roots of amenability}
In this chapter, we answer some interesting questions about invariant means on $\Nbb$ and $\Tbb$,
which will be generalized in later chapters.

\section*{Invariant means}
\begin{point}
Let $S$ be either a discrete semigroup or a locally compact group.
Given a function $f: S \to \Cbb$, the left translate $L_x f$ is defined by $L_x f(y) = f(xy)$.\footnote{
    This is the contravariant left-translation, satisfying $L_{x} L_{y} = L_{yx}$.
    Once we specialize from semigroups to groups, we will prefer the covariant left-translation
    defined by $l_x f(y) = f(x^{-1}y)$.
}
Suppose $\Bcal$ is a space of bounded functions on $S$,
closed under left translation and containing the constant functions.
A \textit{mean} on $\Bcal$ is a linear functional $\mu$ satisfying any two (and hence all three)
of the following properties:
\begin{enumerate}
	\item $\mu$ is unital, that is $\mu(1) = 1$.
	\item $\mu$ is positive, that is $\mu(f)\geq 0$ for each $f\geq 0$.
	\item $\mu$ has norm one.
\end{enumerate}
If $\mu$ satisfies the following property, it said to be \textit{left-invariant}:
\begin{enumerate}
	\setcounter{enumi}{3}
	\item $\mu(f - L_xf) = 0$ for all $f\in \Bcal$ and $x\in S$.
\end{enumerate}
\end{point}

\begin{point}\label{defn amenable}
A discrete semigroup $S$ is said to be \textit{amenable} if $\linfty(S)$ admits a left-invariant mean.
A locally compact group $G$ is said to be amenable if $\Linfty(G)$ admits a left-invariant mean.
A locally compact group $G$ is said to be \textit{amenable-as-discrete} if $\linfty(G)$ admits a left-invariant mean.
\end{point}

\begin{point}
Any compact group $G$ is amenable, as the (normalized) Haar integral is an invariant mean on $\Linfty(G)$.
For finite groups this is trivial;
mathematicians have known for thousands of years that the arithmetic mean of finitely many values is invariant under permutations.
The first infinite group shown to be amenable was the circle group $\Tbb$,
whose Haar integral was constructed by Lebesgue \cite{Lebesgue1904}.
\end{point}

\begin{prop}
There is no countably additive left-invariant mean on $\linfty(\Tbb)$.\footnote{
	Originally due to Vitali \cite{Vitali1905}, this theorem remains true
	when $\Tbb$ is replaced by any nondiscrete locally compact group, see \cite[16.13]{HR1}.
}
\end{prop}
\begin{proof}
Let $H = e^{2\pi i \Qbb}$ be the subgroup of rational points in $\Tbb$.
Let $E\subset \Tbb$ be a transversal for $\Tbb / H$, \ie\ a selection of one point from each coset.
If there exists a countably additive left-invariant mean $\mu\in\linftyS(\Tbb)$, we obtain the following equality,
\[1 = \mu(1)
= \mu\Big(\sum_{h\in H} l_h \Onebb_{E}\Big)
= \sum_{h\in H} l_h \mu(\Onebb_E) = \sum_{n=1}^\infty \mu(\Onebb_E),\]
which cannot be satisfied by any value of $\mu(\Onebb_E)$.
\end{proof}

\begin{point}
We have enough background to ask a few interesting questions.
\begin{enumerate}
	\item Can $\Linfty(G)$ admit an invariant mean if $G$ is not compact?
	\item Does $\linfty(\Tbb)$ admit \textit{any} invariant mean?
	\\ In other words, is $\Tbb$ amenable-as-discrete?
	\item Is the Haar integral the unique invariant mean on $\Linfty(\Tbb)$?
\end{enumerate}
We could dispatch all three questions with nothing more than the Hahn-Banach theorem.
However, we will acquaint ourselves with the use of F\o{}lner sequences and ultrafilters in the following introduction.
\end{point}

\section*{\texorpdfstring
	{The Stone-\v{C}ech compactification and ultrafilters}
	{The Stone-Cech compactification and ultrafilters}
}\label{Ultrafilter section}

\begin{point}
Let $A$ be a commutative Banach algebra with unit $1$.
The \textit{Gelfand spectrum} of $A$ is the set $\Acal$ of nonzero multiplicative functionals,
\[\Acal = \{h\in A^* : h(1) = 1, \text{ and } h(xy) = h(x)h(y) \text{ for all }x,y\in A\},\]
which is $w^*$-closed, hence compact by the Banach-Alaoglu theorem.

The \textit{Gelfand transform} of $A$ is the homomorphism
\[\widehat{\ }: A \to C(\Acal) \text{ defined by } x\mapsto \widehat{x},
	\text{ where } \widehat{x}(h) = h(x),\]
which trivially satisfies $\|\widehat{x}\| \leq \|x\|$.
\end{point}

\begin{prop}[{\cite[Theorem 1.20]{Folland}}]\label{Gelfand transform for Cstar}
When $A$ is a commutative $C^*$-algebra, the Gelfand transform satisfies $\widehat{x^*} = \overline{\widehat{x}}$.
The range $\widehat{A}$ is therefore dense by the Stone-Weierstrass theorem.
Furthermore, the spectral radius coincides with the norm for $C^*$-algebras,
so the Gelfand transform is an isometric $*$-isomorphism.
\end{prop}

\begin{point}
If $X$ is a completely regular space, let $\beta X$ be the Gelfand spectrum of $C(X)$.
$X$ maps homeomorphically onto a dense subset of $\beta X$ via $x\mapsto \delta_x$.
The Gelfand transform $C(X) \to C(\beta X)$ is an isomorphism by \cref{Gelfand transform for Cstar},
and it satisfies $\widehat{f}(\delta_x) = f(x)$.
We associate $X$ with its image in $\beta X$,
so $\widehat{f}$ becomes the unique continuous extension of $f$ from $X$ to $\beta X$.

The fact that each $f\in C(X)$ has a unique continuous extension to $\beta X$
is enough to conclude that any continuous map of $X$ into a compact Hausdorff space
has a unique continuous extension to $\beta X$, see \cite[Theorem 10.7]{GillmanJerison}.
The name for $\beta X$ is the \textit{Stone-\v{C}ech compactification of $X$.}
\end{point}

\begin{point}
Suppose $D$ is a discrete topological space, so $C(D) = \linfty(D)$.
Consider the map from $\beta D$ to $\Pscr(\Pscr(D))$ given by
\[h \mapsto p_h = \{E \subset D : h(\Onebb_E) = 1\}.\]
For each $h\in\beta D$ and $E\subset D$, we have $h(\Onebb_E) = h(\Onebb_E^2) = h(\Onebb_E)^2$,
hence $h(\Onebb_E)$ is either $0$ or $1$.
Since each functional on $\linfty(D)$ is determined by its restriction to the characteristic functions,
we conclude $h\mapsto p_h$ is one-to-one.

An \textit{ultrafilter on D} is a maximal family $p\in\Pscr(\Pscr(D))$
closed under finite intersections but not containing $\varnothing$.
It's not hard to prove $h\mapsto p_h$ is a bijection from $\beta D$ onto the set of ultrafilters on $D$.
We associate $\beta D$ with the set of ultrafilters on $D$ via this bijection.
\end{point}

\begin{point}\label{defn of tilde f}
Suppose $\{x_\gamma\}_{\gamma\in\Gamma}$ is an indexed subset of a compact Hausdorff space $X$.
Regarding $\Gamma$ as a discrete space,
let $f: \Gamma \to X$ be the continuous function $\gamma \mapsto x_\gamma$
with continuous extension $\tilde{f}: \beta \Gamma \to X$.
Given $p\in\beta\Gamma$, $\tilde{f}(p)$ is the unique point in the intersection $\bigcap_{E\in p} \cl(f[E])$.
The standard notation for this point is $\plim_\gamma x_\gamma$.
\end{point}

\begin{point}
A \textit{directed set} is a partially ordered set in which any two elements have a common upper bound.
A \textit{net} is a family $\{x_\gamma\} = \{x_\gamma\}_{\gamma\in \Gamma}$ indexed by a directed set.
A \textit{tail of $\{x_\gamma\}$} is a set of the form $\{x_\beta\}_{\beta \succ \alpha}$ where $\alpha\in\Gamma$.
Suppose $\{x_\gamma\} \subset X$ is a net in a topological space.
A point $y\in X$ is called a \textit{accumulation point of $\{x_\gamma\}$}
if every neighborhood of $y$ intersects every tail of $\{x_\gamma\}$.

In other words, the set of accumulation points of $\{x_\gamma\}$ is
\[\textstyle A = \bigcap_{\alpha \in \Gamma} \cl(\{ x_\gamma : \gamma \succ \alpha \}).\]
If this intersection is the singleton $\{y\}$, we say $\{x_\gamma\}$ converges and write $x_\gamma \to y$.
If $\{x_\gamma\}$ lies in a compact space $X$,
then $A$ is a nested intersection of nonempty compact sets, hence $A\neq\varnothing$.
\end{point}

\begin{point}
For example, let $\Rbb^{\Rbb}$ be
the space of all real functions with the topology of pointwise convergence.
Let $\Pscr_{f}(\Rbb)$ be the finite subsets of $\Rbb$, ordered by inclusion.
Then $\Onebb_{\Rbb}$ is the unique accumulation point of the net $\{\Onebb_{F}\}_{F \in \Pscr_{f}(\Rbb)}$.
\end{point}

\begin{point}
Suppose $\{x_\gamma\}_{\gamma\in\Gamma}$ is a net in a compact Hausdorff space $X$.
Let $\Gamma^* \subset \beta \Gamma$ be the set of ultrafilters
that contain $T_\alpha = \{\beta \in \Gamma : \beta \succ \alpha\}$ for each $\alpha\in\Gamma$.
Then $p\mapsto \plim_\gamma x_\gamma$ maps $\Gamma^*$ onto the accumulation points of $\{x_\gamma\}$.

For example, $\Nbb^*$ is precisely $\beta \Nbb \setminus \Nbb$.
If $\{x_n\}$ is a bounded real sequence, then $p\mapsto \plim_n x_n$ maps $\Nbb^*$
onto the subsequential limit points of $\{x_n\}$.
\end{point}

\begin{prop} \label{Lemma Cardinality of Lambda*}
Suppose $\Gamma$ is a directed set of cardinality $\kappa$,
such that every tail $T_\alpha = \{\beta \in\Gamma : \beta \succ \alpha\}$ has cardinality $\kappa$.
Then $\Gamma^*$ has cardinality $2^{2^\kappa}$.
\end{prop}
\begin{proof}
This follows immediately from \cite[Theorem 3.62]{HindmanStraussBook}, which actually says something a bit stronger:
Let $\Gamma$ be an infinite set with cardinality $\kappa$, 
and let $\Scal$ be a collection of at most $\kappa$ subsets of $\Gamma$
such that $\card{\bigcap F} = \kappa$ for any finite $F \subset \Scal$.
Then there are $2^{2^\kappa}$ $\kappa$-uniform ultrafilters containing $\Scal$.\footnote{
	An ultrafilter $p$ is said to be $\kappa$-uniform if $\card{E} \geq \kappa$ for each $E\in p$.
}
\end{proof}

\section*{Generalized limits and Banach limits}

\begin{point}
Let $\linfty$ be the space of bounded sequences.
A \textit{generalized limit} is a linear functional $\mu$ such that:
\begin{enumerate}
	\item If $\lim_n x_n$ exists, then $\mu(x) = \lim_n x_n$.
	\item For all $x\in\linfty$, $\liminf_n x_n \leq \mu(x) \leq \limsup_n x_n$.
\end{enumerate}
It follows that each generalized limit lies in $M(\Nbb)$, the set of means on $\linfty$.
Notice that $M(\Nbb)$ is a closed subset of the the unit shell of $\linftyS$,
which is compact by the Banach-Alaoglu theorem.
\end{point}

\begin{prop}
The set of convergent sequences is the largest set on which all generalized limits agree.
\end{prop}
\begin{proof}
For each $n\in\Nbb$, define $\widehat{n} \in \linftyS$ by $\widehat{n}(x) = x_n$.
Notice $\{\widehat{n} : n\in \Nbb\}$ lies in $M(\Nbb)$.
If $x\in\linfty$ is not convergent, 
suppose $a = \lim_k x_{m_k}$ and $b= \lim_k x_{n_k}$ are distinct subsequential limits.
Let $\mu$ be an accumulation point of $\{\widehat{m_k}\}_{k\in\Nbb}$,
and $\nu$ an accumulation point of $\{\widehat{n_k}\}_{k\in\Nbb}$.
Now $\mu,\nu$ are evidently both generalized limits, but $\mu(x) = a \neq b = \nu(x)$.
\end{proof}

\begin{point}
Regarding $\beta \Nbb$ as means on $\linfty$, we see the set of nonprincipal ultrafilters
$\Nbb^*\subset \beta \Nbb$ consists of generalized limits.
By \cref{Lemma Cardinality of Lambda*}, $\card{\Nbb^*} = 2^{\cfrak}$.
For comparison, $\card{\linfty} = \Nbb^{\cfrak} = 2^\cfrak$,
hence $\card{\linftyS} = (2^\cfrak)^\cfrak = 2^\cfrak$.
In other words, there are as many generalized limits on $\linfty$ as there are functionals!
\end{point}

\begin{point}
Consider the periodic sequence $(0,1,0,1,\ldots)$.
We would intuitively like the limit of this sequence to be
$\nicefrac{1}{2}$, whereas
the limit along any ultrafilter will be either $0$ or $1$.
For this reason alone, ultrafilters are not the most satisfying generalization of the limit concept.

On the other hand, suppose $\mu\in\linftyS$ is an invariant mean and $x$ is a periodic sequence.
Then $\mu(x)$ is simply the average of $x$ over its period.
Perhaps this felicity is why Banach chose to exhibit invariant means
as his examples of generalized limits in \cite[II \S 3]{Banach32}.
Because of their appearance in this seminal text,
invariant means on $\linfty$ are also known as \textit{Banach limits}.
\end{point}

\begin{point}\label{H construction of banach limit}
\textbf{(First construction of a Banach limit.)}

Let $H\subset \linfty$ be the closed subspace generated by
\[ \{x - L_m x : x\in\linfty,\ m\in\Nbb \}. \]
A mean $\mu\in\linftyS$ is left-invariant if and only if it vanishes on $H$.
Assuming $1$ lies outside $H$,
we can define a functional $p$ on $H + \Cbb$ by
\[p(h+\lambda) = \lambda, \text{ where } h\in H \text{ and } \lambda\in\Cbb.\]
Any norm-one extension of $p$ to the domain $\linfty$ is an invariant mean,
and there exist such extensions by the Hahn-Banach theorem.

To prove $1$ lies outside $H$, it suffices to show
\[ \|1 - (x - L_m x) \|_\infty \geq 1 \text{ for each }x\in\linfty \text{ and }m\in\Nbb.\]
If this is not the case, then the subsequence $(x_0, x_{m}, x_{2m}, \ldots)$ increases without bound, a contradiction.
\qed
\end{point}

\begin{point}\label{Folner construction of Banach limit}
\textbf{(Second construction of a Banach limit.)}

For each $n\geq 0$, define $F_n = \{0, \ldots, n\}$.
The sequence $\{F_n\}$ is called a \textit{F\o{}lner sequence} for $\Nbb$, since it satisfies
\[\textstyle\lim_n\ {\card{F_n \Sdiff m+ F_n}}/{\card{F_n}} = 0 \text{ for each } m\in\Nbb.\]
To each set $F_n$ corresponds $\mu_n\in\linftyS$ defined by
\[ \mu_n(x) = \frac{1}{\card{F_n}} \sum_{k\in F_n} x_k
= \text{ the arithmetic mean of } x \text{ over }F_n.\]
Given $x\in\linfty$ and $m\in \Nbb$, we have
\[ \mu_n(x - L_m x)
    = \frac{1}{\card{F_n}} \sum_{k\in F_n} x_k
    - \frac{1}{\card{F_n}} \sum_{k\in m + F_n} x_k\]
which yields
\begin{equation}\label{Folner equation for N}
|\mu_n(x - L_m x)| \leq \|x\|_\infty \cdot \card{F_n \Sdiff m+F_n} / \card{F_n}.
\end{equation}
Suppose $\mu\in\linftyS$ is any accumulation point of $\{\mu_n : n\in\Nbb\}$.
Given $(x - L_mx)$, there exists an increasing sequence $\{n_k\}$ such that
\[\textstyle \mu(x - L_m x) = \lim_k \mu_{n_k}(x - L_m x) = 0.\]
Therefore, $\mu$ is an invariant mean.
\qed
\end{point}

\begin{point}
We know there are $2^\cfrak$ generalized limits.
Could there possibly be so many Banach limits?
This question was first answered affirmatively by Chou \cite{Chou69},
although his proof was rather different than the one that follows.
\end{point}

\begin{prop}
There are $2^\cfrak$ invariant means on $\linfty$.
\end{prop}
\begin{proof}
Let $C_n =\sum_{k=0}^n k$.
Define a F\o{}lner sequence $\{F_n\}$ by
\[F_n = \{C_n, C_n + 1, \ldots, C_n + n\} \text{ for } n\geq 0.\]
As in \cref{Folner construction of Banach limit}, define $\mu_n\in\linftyS$
to be the arithmetic mean of $x$ over $F_n$.
By \cref{Folner equation for N}, the map $\beta \Nbb \to \linftyS$ given by
$p \mapsto \plim_n \mu_n$ sends nonprincipal ultrafilters to invariant means.
It suffices to show this map is injective.

Suppose $p$ and $q$ are distinct ultrafilters on $\Nbb$, say $E\in p$ but $E^c \in q$.
Notice $A = \bigcup_{n\in E} F_n$ is disjoint from $B = \bigcup_{n\in E^c} F_n$.
Since $\plim_n \mu_n(\Onebb_A) = 1$, we conclude $\plim_n \mu_n(\Onebb_B) = 0$,
whereas $\qlim_n \mu_n(\Onebb_B) = 1.$
\end{proof}

\begin{point}
Since the previous argument was essentially about sets,
it would have been natural to write $\mu(A)$ in place of $\mu(\Onebb_A)$.
Rather than regard this as a mere shorthand,
we associate means and measures as follows.
\begin{itemize}[left=0pt .. \parindent]
	\item A mean $\mu\in\linftyS(G)$ induces a finitely additive probability measure $m$ on $G$,
	via $m(E) = \mu(\Onebb_E)$.
	\item The subspace $H\subset \linfty(G)$ of simple functions is dense.
	A finitely additive probability measure $m$ induces a mean $\mu$ on $H$ via
	\[\mu(\lambda_1 \Onebb_{E_1} + \cdots + \lambda_n \Onebb_{E_n}) = \lambda_1 m(E_1) + \cdots \lambda_n m(E_n),\]
	which extends uniquely to a mean on $\linfty(G)$.
\end{itemize}
\end{point}

\section*{Abelian groups, free groups, and (discrete) F\o{}lner nets}
A pair $(K,\epsilon)$ always denotes $K\Subset G$ and $\epsilon > 0$,
where $\Subset$ is the notation for ``finite nonempty subset.''

\begin{point}
A set $F\Subset G$ is said to be \textit{$(K, \epsilon)$-invariant} if
\[\card{F \Sdiff kF} / \card{F} < \epsilon \text{ for each }k\in K.\]
$G$ is said to satisfy the \textit{discrete F\o{}lner condition} if
there exists a $(K,\epsilon)$-invariant set $F$ for each pair $(K,\epsilon)$.
A net $\{F_\lambda\}$ of nonempty finite subsets of $G$ is called a \textit{discrete F\o{}lner net} if
\[\textstyle \lim_\lambda\ \card{F_\lambda \Sdiff x F_\lambda} / \card{F_\lambda} = 0 \text{ for each }x \in G.\]
\end{point}

\begin{prop}
$(\text{1}) \iff (\text{2}) \Rightarrow (\text{3})$ for the following statements.
\begin{enumerate}[left=0pt .. \parindent]
	\item $G$ satisfies the discrete F\o{}lner condition.
	\item $G$ admits a discrete F\o{}lner net $\{F_\lambda\}$.
	\item $G$ is amenable-as-discrete.
\end{enumerate}
\end{prop}
\begin{proof}
$(1) \Rightarrow (2)$.
Let $\Gamma$ be the set of all $(K,\epsilon)$ pairs, ordered by
\[(K, \delta) \preceq (J, \epsilon) \iff K\subseteq J \text{ and } \delta \geq \epsilon.\]
For each $\gamma\in\Gamma$, choose a $\gamma$-invariant set $F_\gamma.$
$\{F_\gamma\}_{\gamma \in \Gamma}$ is a F\o{}lner net.

$(2) \Rightarrow (1).$
We are given a discrete F\o{}lner net $\{F_\gamma\}_{\gamma \in \Gamma}$ and a pair $(K, \epsilon)$.
For each $k\in K$, choose $\gamma_k \in \Gamma$ such that
\[\card{F_{\gamma} \Sdiff k F_{\gamma}} / \card{F_{\gamma}} < \epsilon \text{ for each } \gamma \succeq \gamma_k.\]
If $\beta$ is a common upper bound for $\{\gamma_k : k\in K\}$,
then $F_\beta$ is $(K, \epsilon)$-invariant.

$(2) \Rightarrow (3).$
Given a discrete F\o{}lner net $\{F_\gamma\}$, define a net $\{\mu_\gamma\} \subset \linftyS(G)$ by
\[\mu_\gamma(f) = \text{the arithmetic mean of }f \text{ over }F_\gamma, \text{ where } f\in\linfty(G).\]
As in \cref{Folner construction of Banach limit}, any accumulation point of this net is a left-invariant mean.
\end{proof}

\begin{point}
Suppose $H$ is a group, and $F\Subset H$ with $H = \bigcup_{n=1}^\infty F^n$.
In this case, $F$ is called a generating set.
Since $\#(F^{m+n}) \leq \#(F^m) \#(F^n)$, we see $\lim_n \sqrt[n]{\#(F^n)}$ exists.
If this limit is $1$, we say $H$ has \textit{subexponential growth},
otherwise we say it has \textit{exponential growth}.
By \cite[Proposition 2.5]{mann2011}, the type of growth of $H$
(subexponential or exponential) does not depend on the choice of generating set.
\end{point}

\begin{prop}\label{not exp therefore amenable}
If every finitely-generated subgroup $H < G$ has subexponential growth,
then $G$ satisfies the discrete F\o{}lner condition.\footnote{
	The converse is not true, and there exist relatively simple counterexamples such as the lamplighter group.
}
\end{prop}
\begin{proof}
Suppose $G$ admits no $(K, \epsilon)$-invariant set.
Then $G$ also admits no\linebreak $(K\cup\{e\}, \epsilon)$-invariant set.
Replacing $K$ by $K\cup\{e\}$, we have
\[(\forall F \Subset G)\ \card{KF} / \card{F} \geq 1+\epsilon.\]
In particular, $\#(K^n) \geq (1+\epsilon)^{n-1} \#(K)$.
Let $H$ be the subgroup generated by $J= K \cup K^{-1}$.
We see $\lim_n \sqrt[n]{\#(J^n)} \geq (1+\epsilon)$, so $H$ has exponential growth.
\end{proof}

\begin{point}\label{abelian groups are amenable-as-discrete}
For example, suppose $G$ is abelian and $K\Subset G$.
$\card{K^n}$ is at most the number of $\card{K}$-tuples of non-negative numbers whose sum is $n$.
This number grows subexponentially, on the order of $n^{\card{K}-1}$.
Hence $G$ is amenable-as-discrete.
\end{point}

\begin{prop}
The free group $\Fbb_2 = \langle a,b \rangle$ is not amenable.
\end{prop}
\begin{proof}
Suppose $\mu\in\linftyS(\Fbb_2)$ is a left-invariant mean.
Define
\begin{align*}
& A = \big\{\text{reduced words beginning in } a \text{ or } a^{-1}\big\} \subset \Fbb_2 \\
& B = \big\{\text{reduced words beginning in } b \text{ or } b^{-1}\big\} \subset \Fbb_2.
\end{align*}
Notice that $aA = bB = \Fbb_2$. We have the following contradiction:
\[2 = \mu(aA) + \mu(bB) = \mu(A) + \mu(B) = \mu(A \sqcup B) \leq 1.\qedhere\]
\end{proof}

\begin{prop}\label{discrete Folner condition}
If $G$ is amenable-as-discrete, it satisfies the discrete F\o{}lner condition.
\end{prop}
\begin{proof}
The following proof appears in greater detail in \cite[Theorem 4.4]{Kerr-Li}.
It is considerably simpler than the original proof due to F\o{}lner \cite{F_lner1955}.

Suppose $G$ does not satisfy F\o{}lner's condition.
By \cref{not exp therefore amenable}, there exists $K\Subset G$
such that $\card{KF} \geq (1+\epsilon)\card{F}$ for each $F\Subset G$.
Let $n$ be large enough that $(1+\epsilon)^n \geq 2$. Replacing $K$ by $K^n$, we have
\[\card{KF} / \card{F} \geq 2 \text{ for each }F\Subset G.\]
Let $\Omega$ be the set of all maps $f: G\times \{1,2\} \to \Pscr(G)$ such that
\begin{enumerate}[label={(\alph*)}]
	\item $f(x,i) \subseteq Kx \text{ for each }(x,i) \in G\times \{1,2\}$, and
	\item $\big|\bigcup_{(x,i) \in S} f(x,i)\big| \geq |S| \text{ for each } S\Subset G\times\{1,2\}$.
\end{enumerate}
$\Omega$ is nonempty, since it contains the map $(x,i) \mapsto Kx$.
Order $\Omega$ by
\begin{itemize}
	\item $f \preceq g \iff f(x,i) \subseteq g(x,i) \text{ for each }(x,i) \in G\times \{1,2\}$.
\end{itemize}
By Zorn's lemma, $\Omega$ contains a minimal element $m$.
It's an exercise in combinatorics to prove each $m(x,i)$ is a singleton.
This being done, we may regard $m$ as a map $G\times \{1,2\} \to G$.
Since $m\in \Omega$, we have
\begin{enumerate}[label={(\alph*$'$)}]
	\item $m(x, i) \in Kx$ for each $(x,i) \in G\times\{1,2\}$, and
	\item $m$ is injective.
\end{enumerate}
For each $k\in K$, define the sets
\[A_k = \{x\in G : m(x,1) = kx\} \text{ and } B_k = \{x \in G : m(x,2) = kx\}.\]
Now each of $\{A_k : k\in K\}$ and $\{B_k : k\in K\}$ are partitions of $G$.
On the other hand, $\{k A_k : k\in K\} \cup \{k B_k : k\in K\}$ is a disjoint family.
Assuming there exists a left-invariant mean $\mu\in\linftyS(G)$, we obtain the following contradiction:
\[2 = \mu\Big(\bigcup_{k\in K} A_k\Big) + \mu\Big(\bigcup_{k\in K} B_k\Big)
	= \mu\Big(\bigcup_{k\in K} kA_k \cup \bigcup_{k\in K} kB_k\Big) \leq 1. \qedhere\]
\end{proof}

\section*{Lebesgue's integral}
\begin{point}
Lebesgue \cite{Lebesgue1904} gave the following axiomatic definition of the Lebesgue integral,
where $a,b,c \in \Rbb$ and $f, g\in B(\Rbb)$, the bounded Borel functions on $\Rbb$ with sup-norm.
\begin{enumerate}
	\item $\int_a^b f(x) \dd{x} = \int_{a+c}^{b+c} f(x-c)\dd{x}$.
	\item $\int_a^b f(x) \dd{x} + \int_b^c f(x) \dd{x} + \int_c^a f(x) \dd{x} = 0 $.
	\item $\int_a^b [f(x) + g(x)] \dd{x} = \int_a^b f(x) \dd{x} + \int_a^b g(x) \dd{x} $.
	\item If $f\geq 0$ on $(a,b)$, then $\int_a^b f(x) \dd{x} \geq 0 $.
	\item $\int_0^1 1 \dd{x} = 1 $.
	\item If $f_n(x) \uparrow f(x)$, then $\int_a^b f_n(x) \dd{x} \uparrow \int_a^b f(x) \dd{x}$.
\end{enumerate}
Lebesgue asked whether property (6) is redundant.
\end{point}

\begin{point}\label{Lebesgue axioms}
To reformulate Lebesgue's axioms in modern terminology,
notice that he only considers the integral over a finite interval $(a, b)$.
By properties (1) and (3),
\[\int_a^b f(x) \dd{x} = \sum_{n= \lfloor a\rfloor}^{\lceil b \rceil} \int_0^1 g(x+n) \dd{x},
\text{ where } g(x) =
\begin{Bmatrix*}[l]
f(x)	& x \in (a,b)\\
0		& \text{else}
\end{Bmatrix*}.\]
Thus the Lebesgue integral is determined by its restriction to $B([0,1])$.\footnote{
	Define $\iota: ([0,1]) \to B(\Rbb)$ by $\iota(f)(x) = f(x)$ for $x\in[0,1]$, and $\iota(f)(x)=0$ otherwise.
	We associate $B([0,1])$ with its image under $\iota$, which is an isometric embedding.}
Since Lebesgue measure is non-atomic, we may replace $[0,1]$ with the circle group $\Tbb$,
and reformulate Lebesgue's definition as follows:
\begin{enumerate}
	\item $I: B(\Tbb) \to \Cbb$ is a positive linear functional.
	\item $I(1) = 1$.
	\item $I(l_x f) = I(f)$ for each $x\in G$ and $f\in B(\Tbb)$.
	\item $I$ is countably additive.
\end{enumerate}
In these terms, Lebesgue's question is whether property (4) is redundant.
\end{point}

\begin{prop}
Property (4) is redundant if $B(\Tbb)$ is replaced by $C(\Tbb)$,
the algebra of continuous bounded functions on $\Tbb$.
\end{prop}
\begin{proof}
Since $I\geq 0$ and $I(1) = 1$, we see $\|I\| = 1$.
The Riesz-Kakutani representation theorem (\cref{Riesz Kakutani})
states that any bounded linear functional on $C(\Tbb)$
may be written as the integral over a regular Borel measure on $\Tbb$, which is countably additive.
\end{proof}

\begin{defn}
Given $E\Subset G$, and a function $f: G\to \Rbb$, define
\[A_E(f) = \frac{1}{\card{E} } \sum_{x\in E} l_x f,\ \ \text{equivalently}\ \ 
A_E(f) = \frac{\Onebb_E}{\card{E}}*f.\]
\end{defn}

\begin{prop}\label{H = S if amenable-as-discrete}
Consider the following subspaces of $B(G)$.
\begin{enumerate}[left=0pt .. \parindent]
	\item $H$, the closed subspace generated by $\{f - l_x f: f\in B(G),\ x\in G\}$
	\item $S = \{f\in B(G) : (\forall \epsilon > 0)\ (\exists F\Subset G)\ 
	\|A_F(f)\| < \epsilon \}$
\end{enumerate}
If $G$ is amenable-as-discrete, then $S = H$.
\end{prop}
\begin{proof}
A mean $\mu\in B^*(G)$ is left-invariant if and only if it vanishes on $H$.
Every left-invariant mean vanishes on $S$, so $S \subseteq H$.
Conversely, suppose $G$ is amenable-as-discrete and choose $(f - l_x f)\in H$.
By \cref{discrete Folner condition}, there exists a F\o{}lner net $\{F_\gamma\}$ for $G$.
We see $(f - l_x f)\in S$, since
\[\|A_{F_\gamma}(f - l_x f)\| \leq \|A_{(F_\gamma \Sdiff xF_\gamma)}(f)\| \to 0.\]
Since $S$ is a closed subspace containing $(f - l_x f)$, we conclude $H\subseteq S$.
\end{proof}

\begin{prop}[{\cite[Theorem 19]{Banach23}}]\label{Banachs theorem}
Let $\Tbb$ be the circle group, and $B(\Tbb)$ the set of bounded Borel functions on $\Tbb$.
There exists an invariant mean $\mu\in B^*(\Tbb)$, distinct from the Lebesgue integral.
This proves property (4) of \cref{Lebesgue axioms} is not redundant.
\end{prop}
\begin{proof}
Define the spaces $S, H \subset B(\Tbb)$ as in \cref{H = S if amenable-as-discrete}.
The abelian group $\Tbb$ is amenable-as-discrete by \cref{abelian groups are amenable-as-discrete}, so $S = H$.
Clearly $1$ lies outside $S$.

Let $E\subset \Tbb$ be a set of Lebesgue measure zero with meager complement.
We will show that $\Onebb_E + \lambda$ lies outside $H$ for any $\lambda\in\Cbb$.
To this end, fix $F\Subset \Tbb$ and define
$g = A_F(\Onebb_E - \lambda) = A_F(\Onebb_E) - \lambda$.
There exists $y \in \bigcap_{x\in F} xE$ by the Baire category theorem,
and there exists $z\not\in \bigcup_{x\in F} x E$ because $E$ has measure zero.
We see $\| g\| \geq \tfrac12 |g(y) - g(z)| = \tfrac12$.
Hence $\Onebb_E - \lambda$ lies outside $H$.

Now, as in \cref{H construction of banach limit},
define a mean $p$ on $H + \Cbb + \Cbb \Onebb_E$ by
\[p(h + \lambda_1 + \lambda_2 \Onebb_{E}) = \lambda_1 + \lambda_2,
\text{ where } h\in H \text{ and } \lambda_1, \lambda_2 \in \Cbb.\]
By the Hahn-Banach theorem, extend $p$ to a norm-one functional $\mu\in B^*(\Tbb)$,
which is evidently a left-invariant mean.
The Lebesgue integral of $\Onebb_E$ is zero, whereas $\mu(\Onebb_E) = 1$.
\end{proof}

\begin{point}
Modern literature on locally compact amenable groups
is mainly concerned with invariant means on $\Linfty(G)$, rather than $B(G)$.
The argument of \cref{Banachs theorem} fails to apply to $\Linfty(\Tbb)$, since $\Onebb_E = 0$ almost everywhere.
However, a rather obvious modification of the proof suggests itself:
\end{point}

\begin{prop}\label{Banach Rudin Granirer}
Let $G$ be any infinite compact group that is amenable-as-discrete.
There exists a left-invariant mean $\mu\in\LinftyS(G)$ distinct from the Haar integral.
\end{prop}
\begin{proof}
Define the following subspaces of $\Linfty(G)$.
\begin{enumerate}[left=0pt .. \parindent]
	\item $H$, the closed subspace generated by $\{f - l_x f: f\in \Linfty(G),\ x\in G\}$
	\item $S = \{f\in \Linfty(G) : (\forall \epsilon > 0)\ (\exists F\Subset G)\ 
	\| A_F(f)\|_\infty < \epsilon \}$
\end{enumerate}
The same argument as \cref{H = S if amenable-as-discrete} shows $S = H$ when $G$ is amenable-as-discrete.

Let $\epsilon = \tfrac13$.
By \cref{small open dense set}, there exists an open dense set $E\subset G$ with $\Haar{E} \leq \epsilon$.
I claim $\Onebb_E + \lambda$ lies outside $H$ for each $\lambda\in \Cbb$.
If not, there exists $F\Subset \Tbb$ such that $\|A_F(\Onebb_E - \lambda)\|_\infty
= \|A_F(\Onebb_E) - \lambda\|_\infty < \epsilon$.
In this case,
$$\textstyle \big| \Haar{E} - \lambda\big|
= \big|\int \Onebb_E - \lambda\big|
= \big|\int A_F(\Onebb_E) - \lambda\big|
\leq \int |A_F(\Onebb_E) - \lambda| < \epsilon,$$
hence $\lambda < \epsilon + \Haar{E} \leq 2\epsilon$.
But $A_F(\Onebb_E) = 1$ on the set of positive measure $\bigcap_{x\in F} xE$.
Therefore $\|A_F(\Onebb_E) - \lambda\|_\infty > 1-2\epsilon = \epsilon$, a contradiction.

From here, the proof proceeds as in \cref{Banachs theorem}.
\end{proof}

\begin{point}
The preceding proposition remained unproved for almost 50 years after Banach proved \cref{Banachs theorem}.
Then, in 1972, independent proofs were given by Granirer \cite{GranirerDiscrete} and Rudin \cite{Rudin72}.
Granirer uses the assumption that $G$ is amenable-as-discrete in essentially the same way we do here,
to prove that $\Linfty(G) / H$ has dimension greater than one.\footnote{
	In fact, he shows that $\Linfty(G) / H$ is non-separable in norm.
}
Rudin considers the natural action of $G$ on $\Acal$, the Gelfand spectrum of $\Linfty(G)$,
which is non-measurable when $G$ is nondiscrete.\footnote{
	For $f\in\Linfty(G)$ and $h\in \Acal$, the function $x \mapsto h(l_x f)$ need not be measurable.
}

This will be discussed further in \cref{Chapter Open Problems}.
\end{point}
\chapter{Technical Background}
In the following, $G$ always denotes an infinite locally compact Hausdorff group.

\section*{Haar measure}
\begin{prop}[Riesz-Kakutani representation theorem, {\cite[11.37]{HR1}}]\label{Riesz Kakutani}
Let $X$ be a locally compact Hausdorff space and $C_c(X)$ the continuous compactly supported functions on $X$.
If $I\in C_c^*(X)$ is a positive linear functional, there exists a regular Borel measure $\mu$ on $X$ such that
\[\textstyle I(f) = \int_X f \dd\mu \text{ for each }f\in C_c(X).\]
This equation uniquely determines $\mu$ on the Borel sets.

Suppose $\mu$ is defined on the $\sigma$-algebra $\Mcal_\mu$.
Given any $A\in \Mcal_\mu$, there exists a Borel set $B$ with $\mu(A\Sdiff B) = 0$.
When we define $L_p(X,\mu)$, we mod out by equivalence on sets of measure zero.
Hence, for our purposes in this thesis, it suffices to assume $\Mcal_\mu$ is simply the Borel sets.
\end{prop}

\begin{prop}[Haar integral, {\cite[15.5]{HR1}}]
Let $C_c(G)$ be the continuous, compactly supported functions on $G$.
There is a nontrivial, positive, left-invariant functional $I$ on $C_c(G)$, which is unique up to scaling.
\end{prop}

The (left) Haar measure on $G$ is the measure constructed from the Haar integral via the Riesz representation theorem.
The Haar measure of $E\subset G$ is denoted $|E|$.
The Haar integral of $f: G \to \Cbb$ is denoted $\int_G f(x) \dd{x}$.
It is quite easy to prove that $|G|$ is finite if and only if $G$ is compact,
in which case we normalize Haar measure so that $|G| = 1$.

\begin{prop}[{\cite[15.11]{HR1}}]
The modular function $\Delta: G \to \Rbb^{\times}$ defined by $|Ex| = |E|\Delta(x)$ is a continuous homomorphism
from $G$ to the multiplicative group of positive real numbers.
\end{prop}

\begin{prop}[{\cite[15.14]{HR1}}]\label{f to fdag}
For $f\in C_c(G)$, let $f^\dag(x) = f(x^{-1}) \Delta(x^{-1})$.
Then the Haar integral of $f^\dag$ equals the Haar integral of $f$.
\end{prop}

As usual, $\Lone(G)$ denotes the set of absolutely integrable functions on $G$, modulo equivalence on null sets.
It is easy to check that \cref{f to fdag} remains true when $C_c(G)$ is replaced by $\Lone(G)$.

\section*{\texorpdfstring
{The function spaces $\Linfty(G)$ and $\LUC(G)$}
{The function spaces Linfty(G) and LUC(G)}}


\begin{point}
We want to define $\Linfty(G)$ so that it is the dual space of $\Lone(G)$.
This involves a small subtlety when $G$ is larger than $\sigma$-compact.
For example, suppose $\phi: G \to \Cbb$ is not measurable, but only \textit{locally measurable}.
In other words, $\phi \Onebb_K$ is measurable for each compact $K\subset G$.
If $f\in\Lone(G)$, its support $\{x\in G : f(x) \neq 0\}$ is $\sigma$-compact, hence $\phi f$ is measurable
and $f \mapsto \int_G \phi(x) f(x) \dd{x}$ defines a linear functional on $\Lone(G)$.
We see $\Linfty(G)$ must be allowed to contain certain nonmeasurable functions.
\end{point}

\begin{point}\label{Defn Linfty}
$E\subset G$ is called \textit{locally null} if $|E \cap K| = 0$ for each compact $K\subset G$.
Define $\Linfty(G)$ to be the algebra of bounded locally measurable functions (with pointwise operations),
modulo equivalence on locally null sets.
For $\phi\in\Linfty(G)$, define $\|\phi\|_\infty$ to be the inf of all $c\geq 0$ such that
$\{x\in G : |\phi(x)| > c\}$ is locally null.
Thus defined, $\Linfty(G)$ is (isomorphic to) the dual space of $\Lone(G)$.
See \cite[Section 2.3]{Folland} for details.
\end{point}

\begin{point}
Let $C(G)$ be the Banach algebra of continuous bounded functions on $G$ with supnorm,
and consider the map $\iota: C(G) \to \Linfty(G)$ which sends $\phi$ to its equivalence class modulo locally null sets.
It's not hard to show $\iota$ is an isometry, so we may regard $C(G)$ as a subalgebra of $\Linfty(G)$.
In this context, we define the Banach algebra of left-uniformly continuous functions on $G$ as follows:
\[\LUC(G)
	= \big\{\phi \in \Linfty(G) : \lim_{x\to e}\|\phi - l_x \phi\|_\infty = 0 \big\}.\]
\end{point}

\begin{prop}\label{convolution smooths}
If $f\in \Lone(G)$ and $\phi\in \Linfty(G)$, then $f*\phi \in \LUC(G)$.
\end{prop}
\begin{proof}
By the regularity of Haar measure,
$\lim_{x\to e} |K \Sdiff xK| = 0$ for any compact $K\subset G$.
Compactly supported simple functions are dense in $\Lone(G)$,
so it follows that $\lim_{x\to e} \|f - l_x f\|_1 = 0$.
Now apply the inequality
\[\|f*\phi - l_x(f*\phi)\|_\infty
= \|f*\phi - (l_x f)*\phi\|_\infty
\leq \|f - l_x f\|_1 \cdot \|\phi\|_\infty.
\qedhere\]
\end{proof}

Let $P_1(G) = \{f\in\Lone(G) : f\geq 0 \text{ and } \int_G f = 1\}$.

\begin{prop}\label{amenable-as-discrete implies amenable}
If $G$ is amenable-as-discrete, then it is amenable.
\end{prop}
\begin{proof}
Suppose $\mu\in\linftyS(G)$ is a left-invariant mean.
Choose $f\in P_1(G)$ and define $F: \Linfty(G) \to \LUC(G)$ by $\phi \mapsto f*\phi$.
It is easy to verify the following properties:
\begin{enumerate}
	\item $F$ is unital, that is $F(1) = 1$.
	\item $F$ is positive, that is $F(\phi) \geq 0$ if $\phi \geq 0$.
	\item $F$ is $G$-equivariant, that is $F(l_x \phi) = l_x F(\phi)$.
\end{enumerate}
It follows that $\mu \circ F$ is a left-invariant mean on $\Linfty(G)$.
\end{proof}

\section*{\texorpdfstring
{The sets $\LIM(G)$ and $\TLIM(G)$}
{The sets LIM(G) and TLIM(G)}}
\begin{point}
Let $M(G)$ denote the set of means on $\Linfty(G)$, and $\LIM(G)$ the set of left-invariant means on $\Linfty(G)$.
Recall that
\[\LIM(G) = \big\{\mu\in M(G) : \big(\forall \phi\in\Linfty(G)\big)\ 
\big(\forall x\in G\big)\ \langle\mu,\, \phi - l_x\phi\rangle = 0 \big\}.\]
The set of \textit{topological left-invariant} means on $\Linfty(G)$ is
\[\TLIM(G) = \big\{\mu\in M(G) : \big(\forall \phi\in\Linfty(G)\big)\ 
\big(\forall f\in P_1(G)\big)\ \langle\mu,\, \phi - f*\phi\rangle = 0 \big\}.\]
For $f\in P_1(G)$, think of $f*\phi$ as the weighted average of left-translates $\int f(x)\, l_x \phi\dd{x}$.
\end{point}

\begin{prop}\label{Prop When G is compact}
When $G$ is compact, the the unique topological left-invariant mean on $G$ is the Haar integral $\mu$.
\end{prop}
\begin{proof}
Pick $\nu\in\TLIM(G)$.
Since $G$ is compact, $1\in C_c(G)$ and \(\nu(1) = 1\).
Thus $\nu|_{C_c(G)}$ is nonzero, and it induces a nonzero left-invariant measure on $G$
via the Riesz-Kakutani representation theorem.
By the uniqueness of Haar measure, we see $\mu|_{C_c(G)} = \nu|_{C_c(G)}$.
Pick $\phi\in\Linfty$ and $f\in P_1(G)$.
Then $f*\phi\in C_c(G)$, so $\mu(\phi) = \mu(f*\phi) = \nu(f*\phi) = \nu(\phi)$.
\end{proof}

\begin{prop}\label{Tlim subset Lim}
$\TLIM(G) \subset \LIM(G)$.
\end{prop}
\begin{proof}
Choose any $f\in P_1(G)$, $\phi\in\Linfty(G)$, $x\in G$, $\mu\in\TLIM(G)$. We have
\[f*(l_x \phi) = (R_x f)*\phi\]
where $R_xf \in P_1(G)$ is defined by $R_xf(y) = f(yx^{-1}) \Delta(x^{-1})$.
Hence
\[\langle \mu,\, l_x \phi\rangle = 
\langle \mu,\, f*(l_x \phi)\rangle = 
\langle \mu,\, (R_x f)*\phi\rangle = 
\langle \mu,\, \phi\rangle.
\qedhere\]
\end{proof}

\begin{prop}
If $G$ is discrete, $\TLIM(G) = \LIM(G)$.
\end{prop}
\begin{proof}
Left-translation is equivalent to convolution by a point-mass.
Convex combinations of point-masses are dense in $P_1(G)$, so $\LIM(G) \subset \TLIM(G)$.
\end{proof}

\begin{prop}
Suppose $\mu$ is a left-invariant mean on $\LUC(G)$.
Then $\mu$ is topological left-invariant on $\LUC(G)$.
\end{prop}
\begin{proof}
Choose $\phi\in\LUC(G)$, $f\in P_1(G)$, and $\epsilon > 0$.
Since compactly-supported functions are dense in $\Lone(G)$, we may suppose $f$ vanishes off the compact set $K$.
Let $K = K_1 \sqcup \ldots \sqcup K_n$ be a partition of $K$ into sets $K_j$
such that $\sup_{x,y\in K_j} \|l_x\phi - l_y \phi\|_\infty < \epsilon$,
and choose a point $k_j \in K_j$ for each $j$.
Define
\[\psi = \sum_{j=1}^n \alpha_j (l_{k_j} \phi),
\text{ where } \alpha_j = \int_{Kj} f(x) \dd{x}.\]
Notice that $\langle \mu,\ \phi - \psi\rangle = 0$, since $\mu$ is left-invariant.
On the other hand, we have
\[f*\phi - \psi
= \int_K f(x)\, (l_x\phi - \psi) \dd{x}
= \sum_{j=1}^n \int_{K_j} f(x)\, (l_x\phi- l_{k_j} \phi) \dd{x},\]
which implies $\|f*\phi - \psi\|_\infty < \epsilon$,
hence $|\langle \mu,\, f*\phi - \psi\rangle| = |\langle \mu,\, f*\phi - \phi\rangle| < \epsilon$.
Since $\epsilon$ was arbitrary, we conclude $\langle \mu,\, f*\phi - \phi\rangle = 0$.
\end{proof}

\begin{prop}\label{TlimG is LimLucG}
The restriction map $r: \TLIM(G) \to \LIM(\LUC(G))$ is a surjective isometry.
\end{prop}
\begin{proof}
Choose $f\in P_1(G)$, $\mu\in\TLIM(G)$, and $m\in\LIM(\LUC(G))$.
To see $r$ is an isometry, notice $\mu(\phi) = \mu(f*\phi)$,
where $f*\phi\in\LUC(G)$ and $\|f*\phi\|_\infty \leq \|\phi\|_\infty$.
To see $r$ is surjective, define $\tilde{m}(\phi) = m(f*\phi)$ and notice $r \tilde{m} = m$.
\end{proof}

\begin{cor}\label{TLIMG is nonempty}
If $G$ is amenable, then $\TLIM(G)$ is nonempty.
\end{cor}
\begin{proof}
If $\mu\in\LIM(G)$, then $\mu|_{\LUC(G)}$ is in $\LIM(\LUC(G))$.
\end{proof}

\section*{F\o{}lner nets}

\begin{point}
Let $K\subset G$ and $\epsilon > 0$.
A function $f\in P_1(G)$ is said to be $(K,\epsilon)$-invariant if $\|l_x f - f\|_1 < \epsilon$ for each $x\in K$.
A net $\{f_\gamma\} \subset P_1(G)$ is said to \textit{converge to invariance in norm} if
$f_\gamma$ is eventually $(K,\epsilon)$-invariant for each $K\Subset G$ and $\epsilon > 0$.
Equivalently, $\| f_\gamma - l_x f_\gamma\|_1 \to 0$ for each $x\in G$.
\end{point}

\begin{prop}\label{limit of conv to invar is LIM}
If $P_1(G)$ contains a net $\{f_\gamma\}$ converging to invariance in norm,
all its accumulation points are left-invariant means, hence $G$ is amenable.
\end{prop}
\begin{proof}
Associate $P_1(G)$ with its image in $M(G)$ under the canonical embedding.
$M(G)$ is compact Hausdorff, hence $\{f_\gamma\}$ has an accumulation point $\mu\in M(G)$.
We must show $\mu$ is left-invariant.

Choose $x\in G$, $\phi\in\Linfty(G)$, and $\epsilon > 0$.
Define
\[U = \{\nu\in M(G) : |\langle \nu-\mu,\, \phi\rangle| < \epsilon \text{ and }
|\langle \nu - \mu,\, l_x \phi\rangle| < \epsilon\},\]
which is a $w^*$-neighborhood of $\mu$.
Let $\beta\in\Gamma$ be large enough that $f_\gamma$
is $(\{x^{-1}\}, \nicefrac{\epsilon}{\|\phi\|_\infty})$-invariant for each $\gamma \succ \beta$, hence
\[|\langle f_\gamma,\, \phi - l_x \phi\rangle|
= |\langle f_\gamma - L_x f_\gamma,\, \phi\rangle|
\leq \|f_\gamma - L_x f_\gamma\|_1\cdot \|\phi\|_\infty < \epsilon.\]
By definition of accumulation point, there exists $f_\gamma \in U$ with $\gamma \succ \beta$.
We have
\[|\langle \mu,\, \phi - l_x \phi \rangle|
\leq |\langle \mu - f_\gamma,\, \phi\rangle|
+ |\langle f_\gamma,\, \phi - l_x \phi\rangle|
+ |\langle f_\gamma - \mu,\, l_x\phi\rangle| < 3\epsilon.\]
Since $\epsilon$ was arbitrary, we conclude $\langle \mu,\, \phi - l_x\phi\rangle = 0$.
\end{proof}

\begin{point}
Conversely, Day \cite[Theorem 1]{Day_OG} showed that every amenable group $G$ has a net $\{f_\gamma\}\subset P_1(G)$
converging to invariance in norm.
Implicit in his proof is a stronger result:
Every $\mu\in\LIM(G)$ is the limit of such a net.
This is our \cref{Days theorem}.
\end{point}

\begin{lem}[{\cite[Theorems 3.4 and 3.10]{rudinBook}}]\label{Lemma Hahn Banach}
Let $X$ be a Banach space, and suppose $A, B\subset X^*$ are disjoint convex sets, compact in the $w^*$-topology.
There exists $T\in X$ that separates $A$ and $B$, in the sense
	$\inf_{a\in A}\Re\langle a, T\rangle > \sup_{b\in B}\Re\langle b, T\rangle$.
\end{lem}

\begin{lem}\label{P1(G) is dense in M(G)}
$P_1(G)$ is dense in $M(G)$.
\end{lem}
\begin{proof}
Suppose $\mu\in M(G) \setminus \cl(P_1(G))$.
By \cref{Lemma Hahn Banach}, there exists $\phi\in\Linfty(G)$ separating the sets $\{\mu\}$ and $\cl(P_1(G))$.
Suppose $\phi$ decomposes into real and imaginary parts as $\phi_1 + i \phi_2$.
Let $\psi = \phi_1 + \|\phi_1\|_\infty$, so that $\psi \geq 0$ satisfies
\[\textstyle \langle \mu,\, \psi\rangle > \sup_{f\in P_1(G)} \langle f,\, \psi\rangle.\]
This contradicts
\[\textstyle \langle \mu,\, \psi\rangle \leq \|\mu\| \cdot \|\psi\|_\infty
= \|\psi\|_\infty = \sup_{f\in P_1(G)} \langle f,\, \psi\rangle.
\qedhere\]
\end{proof}

\begin{lem}[{\cite[Theorem 3.12]{rudinBook}}]\label{norm and weak closures coincide}
Let $E\subset X$ be a convex subset of a Banach space.
Then the weak and norm closures of $E$ coincide.
\end{lem}

\begin{lem}[{\cite[17.13]{kelleyNamioka}}]\label{product of weak topologies}
Let $X, Y$ be Banach spaces. Then the weak topology on $X\times Y$ is the product of the weak topologies.
\end{lem}

\begin{point}
For $\mu\in M(G)$, define $l_x \mu$ by
\[\langle l_x\mu,\, \phi\rangle = \langle \mu,\, L_x \phi\rangle \text{ for }\phi\in\Linfty(G).\]
Notice that $\mu \mapsto l_x \mu$ is continuous in the $w^*$-topology.
Let $\tau: P_1(G) \to M(G)$ denote the canonical embedding.
It is easy to check that
\[\tau(l_x f) = l_x (\tau f) \text{ for }x\in G \text{ and } f\in P_1(G).\]
Usually we associate $P_1(G)$ with its image under $\tau$, but we maintain a distinction in the following proof.
\end{point}

\begin{prop}[Day]\label{Days theorem}
Suppose $G$ is amenable.
Let $U$ be a $w^*$-neighborhood of $\mu \in \LIM(G)$, $K\Subset G$, and $\epsilon > 0$.
There exists a $(K,\epsilon)$-invariant $f\in \tau^{-1}(U)$.
\end{prop}
\begin{proof}
By \cref{P1(G) is dense in M(G)}, there exists a net $\{f_\alpha\}\subset P_1(G)$
such that $\tau f_\alpha \overset{w^*}{\to} \mu$.
For any $x\in G$, we have $(\tau f_\alpha - l_x \tau f_\alpha) \overset{w^*}{\to} 0$,
hence $(f_\alpha - l_x f_\alpha) \to 0$ weakly.
By \cref{product of weak topologies}, we see
$(f - l_x f)_{x\in K} \to 0$ weakly in $\prod_{x \in K} \Lone(G)$.
This shows $0$ is in the weak-closure of the convex set
$\big\{(f - l_x f)_{x\in K} : f\in \tau^{-1}(U)\big\}$,
which coincides with the norm-closure by \cref{norm and weak closures coincide}.
Hence there exists $f\in \tau^{-1}(U)$ with $\|f - l_x f\|_1 < \epsilon$ for each $x\in K$.
\end{proof}

\begin{point}
If $F\subset G$ has finite positive measure, define $\mu_F = \Onebb_F / |F|$.
A set $F\subset G$ is said to be $(K,\epsilon)$-invariant if $\mu_F$ is.
A net $\{F_\gamma\}$ of subsets of $G$ is called a \textit{F\o{}lner net}
if it is eventually $(K,\epsilon)$-invariant for each $K\Subset G$ and $\epsilon > 0$.
The following theorem shows that every amenable group $G$ admits a F\o{}lner net.
This generalizes the discrete result, \cref{discrete Folner condition}.
\end{point}

\begin{prop}[{\cite[Theorem 3.5]{NamiokaFolner}}]\label{exists pointwise folner net}
Suppose $G$ is amenable.
For each $K\Subset G$ and $\epsilon > 0$, there exists a $(K,\epsilon)$-invariant set $F \subset G$.
\end{prop}
\begin{proof}
Let $\delta = \epsilon / \#(K)$.
By \cref{Days theorem}, there exists $f\in P_1(G)$ which is $(K, \nicefrac{\delta}{3})$-invariant.
Let $g\in P_1(G)$ be a simple function with $\|g - f\|_1 < \nicefrac{\delta}{3}$.
It follows that $g$ is $(K,\delta)$-invariant.
Express $g$ in ``layer cake'' form as
\[g = \sum_{i=1}^n c_i \mu_{A_i}
\text{ where }
c_i > 0 \text{ and }
A_1 \subset A_2 \subset \ldots \subset A_n.\]
At any given point, the functions $\{\Onebb_{A_i} - \Onebb_{x A_i}\}_{i=1}^n$ all have the same sign, hence
\[|g - l_x g|
= \sum_{i=1}^n c_i |\Onebb_{A_i} - \Onebb_{x A_i}| / |A_i|
= \sum_{i=1}^n c_i \Onebb_{(A_i \Sdiff x A_i)} / |A_i|\ \ \text{for each }x\in G.\]
In particular,
\[\tfrac{\epsilon}{\#(K)} = \delta
> \|g - l_x g\|_1 = \sum_{i=1}^n c_i |A_i \Sdiff x A_i| / |A_i|\ \ \text{for each } x \in K.\]
Summing over $x\in K$, we have
\[\epsilon > \sum_{n=1}^n c_i \sum_{x\in K} |A_i \Sdiff xA_i| / |A_i|.\]
Since $\sum_{i=1}^n c_i = 1$, there exists some $i$ with
\[\epsilon > \sum_{x\in K} |A_i \Sdiff xA_i| / |A_i|.\]
In particular, $A_i$ is $(K,\epsilon)$-invariant.
\end{proof}

\section*{Topological F\o{}lner nets}

\begin{defn}\label{weaker TI net defn}
A net $\{f_\gamma\} \subset P_1(G)$ is called a \textit{(left) TI-net}
it is eventually $(K,\epsilon)$-invariant for each compact $K\subset G$ and $\epsilon > 0$.
Equivalently, $\|f_\gamma - l_x f_\gamma\|_1 \to 0$ uniformly on compact sets.
Hulanicki \cite{Hulanicki1966}
showed that $G$ admits a TI-net net if it is amenable.\footnote{
	The existence of a TI-net on $G$ is known as ``Reiter's condition.''
	As far back as \cite[\pp 407]{Reiter52},
	Reiter showed how to construct a TI-net when $G$ is locally compact abelian,
	and he discussed the applications of TI-nets in various other papers.
}
Implicit in his proof is our \cref{every tlim is limit of top net}:
Every $\mu\in\TLIM(G)$ is the limit of a TI-net.

A net $\{F_\gamma\}$ of subsets of $G$ is called a \textit{topological F\o{}lner net}
if it is eventually $(K,\epsilon)$-invariant for each compact $K\subset G$ and $\epsilon > 0$.
\cref{exists topological folner net} shows every amenable $G$ admits a topological F\o{}lner net.
\end{defn}

\begin{prop}\label{limit of TI-net}
Every accumulation point of a TI-net is a topological left-invariant mean.
\end{prop}
\begin{proof}
Same as \cref{limit of conv to invar is LIM}, but with $p\in P_1(G)$ in place of $x\in G$.
\end{proof}

For $f\in \Lone(G)$, let $f^{\dag}(x) = f(x^{-1}) \Delta(x^{-1})$.
Notice that $f\mapsto f^{\dag}$ is an involution of $\Lone(G)$ that sends $P_1(G)$ to itself.
For $\mu\in M(G)$ and $p\in P_1(G)$ define $p* \mu$ by
\[\langle p*\mu,\, \phi\rangle = \langle \mu,\, p^{\dag}* \phi\rangle \text{ for }\phi\in\Linfty(G).\]
Notice that $\mu \mapsto p* \mu$ is continuous in the $w^*$-topology.
Let $\tau: P_1(G) \to M(G)$ denote the canonical embedding.
It is easy to check that
\[\tau(p* f) = p* (\tau f) \text{ for }f, p\in P_1(G).\]

\begin{prop}\label{lemma P1 invariant}
Suppose $G$ is amenable.
Let $U$ be a $w^*$-neighborhood of $\mu \in \TLIM(G)$, $K \Subset P_1(G)$, and $\epsilon > 0$.
There exists $f\in \tau^{-1}(U)$ with $\|f - p*f\|_1 < \epsilon$ for each $p \in K$.
\end{prop}
\begin{proof}
Exactly the same as \cref{Days theorem}, but with $p*f$ in place of $l_x f$.
\end{proof}

\begin{prop}\label{every tlim is limit of top net}
Suppose $G$ is amenable.
Let $U$ be a $w^*$-neighborhood of $\mu \in \TLIM(G)$, $K \subset G$ be compact, and $\epsilon > 0$.
There exists a $(K,\epsilon)$-invariant $h\in \tau^{-1}(U)$.
\end{prop}
\begin{proof}
Choose $g\in P_1(G)$ and let $\delta = \epsilon / 5$.
Let $V$ be a small enough $w^*$-neighborhood of $\mu$ that
$f*V = \{f*\nu : \nu\in V\} \subset U.$
Let $E\subset G$ be a small enough neighborhood of $e$ that
$\| \mu_E * f - f\|_1 < \delta.$
By compactness, we have $K \subset \{x_1, \ldots, x_n\}E$ for some $\{x_1, \ldots, x_n\} \Subset G$.
By \cref{lemma P1 invariant}, there exists $f\in \tau^{-1}(V)$
with $\|f - p*f\|_1 < \delta$ for each $p \in \{g, \mu_{x_1 E}, \ldots, \mu_{x_n E}\}$.
Now the chain of inequalities given in \cite[Proposition 4.1]{Hulanicki1966}
shows $h = g*f$ is $(K,\epsilon)$-invariant.
We have $\tau(h)\in f*V \subset U$, hence $h\in \tau^{-1}(U)$.
\end{proof}

\begin{prop}\label{almost topological folner}
Let $G$ be amenable, $K\subset G$ compact, and $\epsilon > 0$.
There exists $F\subset G$ and $K_0 \subset K$ such that
$F$ is $(K_0,\, \epsilon)$-invariant and $|K\setminus K_0| < \epsilon$.
\end{prop}
\begin{proof}
Assume $0 < |K|$, otherwise there is nothing to prove.
By \cref{every tlim is limit of top net}, there exists $f\in P_1(G)$ which is $(K,\, \epsilon^2/|K|)$-invariant.
As in \cref{exists pointwise folner net}, take $f$ to be a simple function, expressed in layer-cake form as
\[f = \sum_{i=1}^n c_i \mu_{A_i} \text{ where }
c_i > 0 \text{ and }
A_1 \subset \ldots \subset A_n.\]
As in the same proposition, we have
\[\epsilon^2 / |K| > \|f - l_x f\|_1 = \sum_{i=1}^n c_i |A_i \Sdiff xA_i| / |A_i| \text{ for each } x\in K.\]
Integrating over $x\in K$, we have
\[\epsilon^2 > \sum_{i=1}^n c_i \int_K |A_i \Sdiff xA_i| / |A_i|\dd{x}.\]
Since $\sum_{i=1}^n c_i = 1$, there is at least one $i$ with
\[\epsilon^2 > \int_K |A_i \Sdiff xA_i| / |A_i| \dd{x}.\]
If $N = \{x\in K : |A_i \Sdiff xA_i| / |A_i| \geq \epsilon\}$, clearly $|N| < \epsilon$.
Let $F = A_i$ and $K_0 = K \setminus N$.
\end{proof}

\begin{prop}\label{exists topological folner net}
Let $G$ be amenable, $K\subset G$ compact with $|K| > 0$, and $0 < \epsilon \leq |K|/2$.
There exists $F\subset G$ which is $(K,\epsilon)$-invariant.
\end{prop}
\begin{proof}

Let $A = K \cup K^{-1}K$, so that $K \subset xA$ for each $x\in K$.
If $A_0\subset A$ is any subset with $|A\setminus A_0| < \epsilon$, I claim $K \subset A_0A_0^{-1}$.
Indeed, for each $x\in K$,
\[2\epsilon \leq |K| \leq |x A \cap A|
\leq |xA_0 \cap A_0| + |x(A\setminus A_0)| + |A\setminus A_0|
< |xA_0 \cap A_0| + 2\epsilon.\]
Thus $|xA_0\cap A_0| > 0$, which shows $x \in A_0 A_0^{-1}$.

By \cref{almost topological folner}, there exists $F\subset G$ which is $(A_0,\, \nicefrac\epsilon2)$-invariant,
where $A_0 \subset A$ with $|A\setminus A_0| < \epsilon$.
If $x,y \in A_0$, then
\[|xy^{-1} F \Sdiff F |
\leq |y^{-1}F \Sdiff F | + |F \Sdiff x^{-1}F|
= |F \Sdiff yF| + |xF \Sdiff F| < \epsilon|F|.\]
Thus $F$ is $(A_0 A_0^{-1}, \epsilon)$-invariant.
\end{proof}


\begin{defn}
$F\subset G$ is called \textit{$(K,\epsilon)$-absorbent} if $|KF \Sdiff F| / |F| < \epsilon$.
$\{F_\gamma\}$ is called a \textit{topological absorbent F\o{}lner net} if if is eventually
$(K,\epsilon)$-absorbent for each nonempty compact $K\subset G$ and $\epsilon > 0$.

The following lemma shows there is no loss of generality in quantifying over pairs $(K,\epsilon)$ where $e\in K$.
This being the case, $|KF\Sdiff F| = |KF\setminus F|$,
hence the condition of being $(K,\epsilon)$-absorbent gets stronger as $K$ gets larger.
\end{defn}

\begin{lem}
Let $J\subset G$ be a nonempty set and $K = J\cup\{e\}$.
If $F\subset G$ is $(K,\epsilon)$-absorbent, then it is $(J,2\epsilon)$-absorbent.
\end{lem}
\begin{proof}
$
|JF \Sdiff F|
\ \ \ =\ \ \ |JF \setminus F| + |F \setminus JF|
\ \ \ = \ \ \ |JF \setminus F| + |F| - |F \cap JF|
\\\leq |JF \setminus F| + |JF| - |F \cap JF|
\ \ \ =\ \ \ 2|JF \setminus F|
\ \ \ \leq\ \ \ 2|KF \setminus F|
\ \ \ <\ \ \ 2\epsilon|F|.$
\end{proof}

As the reader might expect, every amenable group $G$ admits a topological absorbent F\o{}lner net.
However, in \cref{tia paper}, we need a more specific fact than the \textit{existence} of such a net:
Given $(K,\epsilon)$, any sufficiently $(K',\epsilon')$-invariant set $F'\subset G$
can be \textit{approximated in measure} by a $(K,\epsilon)$-absorbent set $F$.

\begin{prop}[{\cite[Proposition 3.1.2]{EmersonGreenleaf67}}]\label{Emerson Greenleaf Lemma}
Choose $K\subset G$ which is compact, has nonempty interior, and contains $e$.
Let $K' = K^2 K^{-1}$.
Then there exist $\epsilon > 0$ and $M > 0$ satisfying the following:
For any $0 < \epsilon' < \epsilon$, and any $F'\subset G$ which is $(K',\epsilon')$-invariant,
there exists a subset $F \subset F'$ such that
\begin{enumerate}
\item $F$ is $\big(K,\, M\sqrt{\epsilon'}\big)$-absorbent, and
\item $|F' \setminus F| < \sqrt{\epsilon'}|F|$.
\end{enumerate}
In particular, when $\epsilon'$ is small enough, $F$ is $(K,\epsilon)$-absorbent.
\end{prop}

\section*{Ordinals, cardinals, and cardinal invariants}
An ordinal is a set $\alpha$ such that $\bigcup \alpha \subset \alpha$ and $(\alpha, \in)$ is well-ordered.

\begin{prop}[{\cite[Theorem 6.3.1]{JechSetTheory} }]
Given a well ordered set $(S, <)$, there is an ordinal $\alpha$
such that $(S, <)$ is order-isomorphic to $(\alpha, \in)$.
\end{prop}

Given ordinals $\alpha$ and $\beta$, we will write $\alpha < \beta$ to mean $\alpha \in \beta$.
This allows us to write ``$\alpha \leq \beta$,'' as opposed to ``$\alpha \in \beta$ or $\alpha = \beta$.''

\begin{prop}[Transfinite Induction, {\cite[Theorem 6.4.1]{JechSetTheory} }]
Let $P(\alpha)$ be a property of ordinals.
Assume that for all ordinals $\beta$, if $P(\gamma)$ holds for all $\gamma < \beta$, then $P(\beta)$ holds.
Then we have $P(\alpha)$ for all ordinals $\alpha$.
\end{prop}

It follows that every ordinal is the well-ordered chain of all previous ordinals.
For example, if we define $0 = \varnothing$, then the first infinite ordinal is $\Nbb$.
Any set can be well-ordered with the axiom of choice, see \cite[Theorem 8.1.13]{JechSetTheory},
hence there is an ordinal of every cardinality.
Therefore, for any set $S$, we may define the cardinal number $\card{S}$
to be the first ordinal $\alpha$ such that there exists a bijection from $\alpha$ to $S$.

If $\alpha$ and $\beta$ are ordinals, we write $\alpha\cdot\beta$ for $\#(\alpha \times \beta)$,
and $\alpha^\beta$ for $\#(\alpha^\beta) = \#(\{\text{all maps from }\alpha \text{ to } \beta\})$.

\begin{defn}\label{Defn Cardinal Invariants}
The following are called \textit{cardinal invariants} of a topological group $G$,
since they are preserved by homeomorphisms.
\begin{enumerate}
	\item $\mu = \mu(G)$ is the smallest cardinality of a basis of open sets at $e\in G$.
	\item $\kappa = \kappa(G)$ is the smallest cardinality of a covering of $G$ by compact subsets.
\end{enumerate}
\end{defn}

\begin{prop}[{\cite[Theorem 8.3]{HR1}}]
A Hausdorff group $G$ is metrizable if and only if $\mu(G) \leq \Nbb$.
\end{prop}

\begin{prop}[Kakutani-Kodaira, {\cite[Theorem 8.7]{HR1}}]
If $\kappa(G) \leq \Nbb$, there is a closed normal subgroup $N \normal G$
such that $|N| = 0$ and $G/N$ is metrizable.
\end{prop}

\begin{cor}\label{small open dense set}
If $\kappa(G) \leq \Nbb$ and $\epsilon > 0$, there is a dense open set $E\subset G$ with $|E| \leq \epsilon$.
\end{cor}
\begin{proof}
Let $N \normal G$ be a closed normal subgroup such that $|N| = 0$ and $G/N$ is metrizable.
Since $\kappa(G/N) \leq \Nbb$, we see $G/N$ contains a countable dense subset $D$.
Therefore $DN$ is dense in $G$ and has measure zero.
By the regularity of Haar measure, $DN \subset E$ for some open $E\subset G$ with $|E| \leq \epsilon$.
\end{proof}


In the remainder of this chapter,
let $\{U_\alpha\}_{\alpha < \mu}$ be a neighborhood basis of distinct symmetric sets at the origin.
Since $G$ is locally compact, we furthermore suppose that each $U_\alpha$ is \textit{regular open},
in other words $U_\alpha$ is the interior of its closure.\footnote{
	If $\{K_\alpha\}$ is a neighborhood basis of closed sets,
	then $\{K_\alpha^{\circ}\}$ is a neighborhood basis of regular open sets.
}
If $U,V$ are distinct regular open sets,
notice $U\Sdiff V$ has nonempty interior

\begin{lem}\label{dense subset of haar sets}
Suppose $\mu$ is infinite and $K\subset G$ is nonempty compact.
There exists a family $D$ of regular open sets such that
\\(1) $\#(D) = \mu$, and
\\(2) Given a measurable set $A \subset K$ and $\epsilon > 0$, there is $B \in D$ with $|A\Sdiff B| < \epsilon$.
\end{lem}
\begin{proof}
For $\alpha < \mu$, choose $D_\alpha \Subset K$ with $K \subset U_\alpha D_\alpha$.
Let $D_{00} = \bigcup_{\alpha < \mu} D_\alpha$.
Obviously $D_{00}\subset K$ is dense, and $\#(D_{00}) \leq \mu$.
Define
\[D_0 = \{U_\alpha d : \alpha < \mu,\ d\in D_{00}\}
\ \ \text{and}\ \  D = \{\text{finite unions in }D_0\}.\]
Of course $\#(D) = \#(D_0)$, and $\#(D_0) = \#(D_{00}) \cdot \mu = \mu$.

Since Haar measure is regular, there exist a compact set $C$ and an open set $V$
with $C \subset A \subset V$ and $|V \setminus C| < \epsilon$.
About each $x\in C$, choose a neighborhood of the form $U_{\alpha_x} d_x \subset V$.
Let $V_0 \in D$ be a finite union of these neighborhoods covering $C$.
Since $C \subset V_0 \subset V$, we have $|A \Sdiff V_0| < \epsilon$.
\end{proof}

Let $B = B(G) \subset \Linfty(G)$ be the Boolean algebra of characteristic functions.

\begin{prop}\label{Cardinality of B in general}
If $G$ is nondiscrete locally compact, then $\#(B) = \mu^{\kappa\cdot \Nbb}$.
\\(When $G$ is discrete, of course $\#(B) = 2^{\#(G)} = 2^{\kappa}$.)
\end{prop}
\begin{proof}
Let $\{K_\gamma\}_{\gamma < \kappa}$ be a family of compact subsets that cover $G$.
For each $\gamma$, apply \cref{dense subset of haar sets} to $K_\gamma$ to obtain $D_\gamma$.
For each $\Onebb_E\in B$ and $\gamma < \kappa$, define $E^{\gamma} = E \cap K_\gamma$,
and choose a sequence $\{E^\gamma_n\}_{n\in\Nbb} \subset D_\gamma$ with $|E^\gamma\Sdiff E^\gamma_n| \to 0$.
Evidently
\[\Onebb_E \mapsto \big( (\gamma,n) \mapsto E^{\gamma}_n \big)\]
is an injection of $B$ into $(D_\gamma)^{\kappa \times \Nbb}$,
which shows
$\#(B) \leq \#(\mu)^{\kappa \cdot \Nbb}.$

To get the opposite inequality, suppose all the elements of our neighborhood basis $\{U_\alpha\}_{\alpha < \mu}$
are contained in a single compact set $K$.
Let $\{Kx_{\gamma}\}_{\gamma < \kappa}$ be disjoint translates of $K$.
Evidently the map from $\mu^{\kappa}$ to $B$ defined by
\[f\mapsto \bigcup_{\gamma < \kappa} U_{f(\gamma)} x_{\gamma}\]
is an injection, which shows $\mu^{\kappa \cdot \Nbb} \leq \#(B)^{\Nbb}$.
Since $B$ is a complete Boolean algebra, \cite{comfortHager} yields $\#(B) = \#(B)^{\Nbb}$.
\end{proof}

Now let $\Acal$ denote the spectrum of $\Linfty(G)$.
Since each $h\in\Acal$ is determined by its restriction to the set $B$ of characteristic functions,
the map
\[h \mapsto \{\Onebb_E \in B : h(\Onebb_E) = 1\}\]
is a bijection of $\Acal$ with the set of ultrafilters on $B$.

\begin{prop}\!\!\footnote{
	Thanks to Alex Chirvasitu for showing me the papers \cite{comfortHager} and \cite{balcarFranek}.
	When $G$ is compact, he also pointed out that $\#(\Bcal) = \#(RO)$,
	where $RO$ is the Boolean algebra of regular open sets in $G$.
}\ \ 
If $G$ is locally compact, then $\#(\Acal) = 2^{\#(B)}$.
\end{prop}
\begin{proof}
By \cite[Corollary 1]{balcarFranek}, the number of ultrafilters on the complete Boolean algebra $B$ is $2^{\#(B)}$.
\end{proof}

\chapter{\texorpdfstring
{Counting topological invariant means on $\Linfty(G)$ and $VN(G)$ with ultrafilters}
{Counting topological invariant means with ultrafilters}
}\label{Chapter Enumerating TIM}
The following chapter is based on my paper \cite{HopfenspergerCounting}.

\medskip

\noindent\textbf{\large Abstract}

\medskip\noindent
In 1970, Chou showed there are $\card{\Nbb^*} = 2^{2^{\Nbb}}$ topological invariant means on $L_\infty(G)$
for any noncompact, $\sigma$-compact amenable group.
Over the following 25 years, the sizes of the sets of topological invariant means on $L_\infty(G)$ and $VN(G)$
were determined for any locally compact group.
Each paper on a new case reached the same conclusion -- ``the cardinality is as large as possible'' --
but a unified proof never emerged.
In this paper, I show $L_1(G)$ and $\AG$ always contain orthogonal nets converging to invariance.
An orthogonal net indexed by $\Gamma$ has $\card{\Gamma^*}$ accumulation points,
where $\card{\Gamma^*}$ is determined by ultrafilter theory.

Among a smattering of other results, I prove Paterson's conjecture
that left and right topological invariant means on $L_\infty(G)$ coincide
if and only if $G$ has precompact conjugacy classes.
Finally, I discuss some open problems arising from the study
of the sizes of sets of invariant means on groups and semigroups.

\section*{History}
\begin{point}\label{First point}
F\o{}lner's condition for amenable discrete groups says, for all finite $K\subset G$ and $\epsilon > 0$,
there exists a finite $(K,\epsilon)$-invariant set $F_{(K,\epsilon)}\subset G$.
The set $\Gamma$ of all ordered pairs $\gamma = (K, \epsilon)$ is a directed set, ordered by increasing $K$ and decreasing $\epsilon$.
Define $m_\gamma\in\ell_\infty^*(G)$ by \[m_\gamma(f) = \frac{1}{\card{F_\gamma}} \sum_{x\in F_\gamma} f(x).\]
Then the net $\{m_\gamma\}_{\gamma\in\Gamma}$ converges to invariance, and any accumulation point is a left-invariant mean on $\ell_\infty(G)$.

The analogue of F\o{}lner's condition for locally compact amenable groups says,
for all compact $K\subset G$ and $\epsilon > 0$, there exists compact $F\subset G$ which is $(K,\epsilon)$-invariant.
The analogous definition of $m_\gamma\in\LinftyS(G)$ is
\[m_\gamma(f) = \frac{1}{\Haar{F_\gamma}} \int_{F_\gamma} f(x) \dd{x}.\]
If $m$ is any accumulation point of $\{m_\gamma\}_{\gamma\in\Gamma}$,
it is not only left-invariant but topological left-invariant.
That is, $m(\phi) = m(f*\phi)$ for any $\phi\in\Linfty(G)$ and $f\in \Lone(G)$ with $\|f\|_1 = \int_G f = 1$.
\end{point}

\begin{point}
Let $\TLIM(G)$ denote the topological left-invariant means on $\Linfty(G)$, and $\TIM(G)$ the (two-sided) topological invariant means.
In \cite{patersonFCBar}, Paterson conjectured that $\TLIM(G) = \TIM(G)$ if and only if $G$ has precompact conjugacy classes, and proved it assuming $G$ is compactly generated.
The short and insightful paper \cite{milnes} proved Paterson's conjecture assuming $G$ is $\sigma$-compact, and gave me the ideas to prove it in full generality.
\end{point}

\begin{point}\label{Intro_Number_of_Tims}
Let $\kappa = \kappa(G)$ be the first ordinal such that
there is a family $\Kcal$ of compact subsets of $G$ with $\card{\Kcal} = \kappa$ and $G = \bigcup \Kcal$.
It's not hard to prove $\card{\TIM(G)} \leq \card{\TLIM(G)} \leq 2^{2^\kappa}$.
Of course when $\kappa = 1$, the unique topological invariant mean is Haar measure.
But when $\kappa \geq \Nbb$, $\card{\TIM(G)}$ actually equals $2^{2^\kappa}$.
Here is an abbreviated history of this surprising result, which took almost 20 years to establish:

When $\kappa=\Nbb$, Chou \cite{Chou70} defined $\pi: \Linfty(G)\to\ell_\infty(\Nbb)$ by
\[\pi(f)(n) = \frac{1}{\Haar{U_n}}\int_{U_n} f,\]
where $\{U_n\}$ is a F\o{}lner sequence of mutually disjoint sets.
Thus $\pi^* : \ell_\infty^*(\Nbb) \to \LinftyS(G)$ is an embedding.
Let $c_0 = \{x\in \ell_\infty(\Nbb) : \lim_n x_n = 0\}$ and \linebreak
$\Fcal = \{m\in\ell_\infty^*(\Nbb) : m(1)=1,\, m \geq 0,\, m|_{c_0} \equiv 0\},$
so that $\pi^*[\Fcal]\subset\TLIM(G)$.
Regarding the nonprincipal ultrafilters on $\Nbb$ as elements of $\Fcal$,
we see $\card{\TLIM(G)} \geq \card{\beta\Nbb - \Nbb} = 2^{2^{\Nbb}}$.

When $G$ is discrete and $\kappa \geq \Nbb$, \cite{ChouExactCard} proved $\card{\TIM(G)} \geq 2^{2^\kappa}$.
When $G$ is nondiscrete and $\kappa \geq \Nbb$, \cite{lau-paterson} proved $\card{\TLIM(G)} \geq 2^{2^\kappa}$.
These papers take a different approach, constructing a large number of disjoint translation-invariant subsets of the LUC-compactification $G^{LUC}$.
Each such subset supports at least one invariant mean by Day's fixed point theorem.

The full result $\card{\TIM(G)} \geq 2^{2^\kappa}$ was finally proved by Yang in \cite{Yang}.
Yang realized that the trick to generalizing Chou's embedding argument
is to replace $\Nbb$ by the indexing set of a F\o{}lner net.
\end{point}

\begin{point}
When $G$ is abelian, let $\Ghat$ denote its dual group.
If $G$ is discrete abelian, it is well-known that $\kappa(\Ghat) = 1$, hence Haar measure is the unique topological invariant mean on $L_\infty(\Ghat)$.
More generally, let $\mu = \mu(G)$ be the first ordinal
such that $G$ has a neighborhood basis $\Ucal$ at the origin with $\card{\Ucal} = \mu$.
In these terms, $\kappa(\Ghat) = \mu(G)$ by \cite[24.48]{HR1}.
Thus when $G$ is nondiscrete abelian, $\kappa(\Ghat) \geq \Nbb$ and $\card{\TIM(\Ghat)} = 2^{2^\mu}$.

When $G$ is non-abelian, the group von Neumann algebra $\VN$ is the natural analogue of $L_\infty(\Ghat)$.
If $\TIM(\Ghat)$ denotes the set of topological invariant means on $\VN$, then the analogous results hold:
When $\mu = 1$, $\TIM(\Ghat)$ is the singleton comprising the point-measure $\delta_e$, as proved in \cite{renaud}.
When $\mu \geq \Nbb$, $\card{\TIM(\Ghat)} = 2^{2^\mu}$.
This is proved in \cite{Chou82} when $\mu = \Nbb$, using an embedding $\pi^*: \ell_\infty^*(\Nbb) \to \VN^*$,
and in \cite{Hu95} when $\mu > \Nbb$,
using a family $\{\pi_\gamma^*: \ell_\infty^*(\mu) \to \VN^*\}_{\gamma<\mu}$ of embeddings!
\end{point}
\section*{Projections and Means}

\begin{point}\label{Definition of Support}
A positive unital functional on a von Neumann algebra $X$ is called a state.
A state $u$ is called normal if $\langle u, \sup_\alpha P_\alpha\rangle = \sup_\alpha \langle u, P_\alpha\rangle$
for any family $\{P_\alpha\}\subset X$ of (orthogonal) projections.
Let $\Pcal_1$ denote the set of normal states on $X$.
Sakai \cite{SakaiPredual} famously proved that the linear span of $\Pcal_1$ forms the predual of $X$,
that is, $\mathrm{span}(\Pcal_1)^* = X$.
For each $u\in \Pcal_1$, we can define a projection $S(u)$ called the support of $u$,
which is the inf of all projections $P$ such that $\langle P, u\rangle = \langle I, u\rangle = 1$.
Now $u,v\in \Pcal_1$ are called orthogonal if $S(u) S(v) = 0$.
\end{point}

\begin{lem}\label{Lemma Orthogonal Implies Injection}
Suppose $\{u_\gamma\}_{\gamma\in\Gamma} \subset \Pcal_1$ are mutually orthogonal.
In other words, $S(u_\gamma) S(u_\beta) = 0$ when $\gamma\neq\beta$.
Then the map $\beta\Gamma \to X^*$ given by $p\mapsto \plim_\gamma u_\gamma$ is one-to-one.
\end{lem}
\begin{proof}
Suppose $p,q$ are distinct ultrafilters, say $E\in p$ and $E^C \in q$.
Let $P = \sup_{\gamma\in E} S(u_\gamma)$.
Clearly $\langle \plim_\gamma u_\gamma, P\rangle = 1$.
For any $\beta\in E^C$, $P \leq 1 - S(u_\beta)$, hence $\langle u_\beta, P\rangle = 0$, hence $\langle\qlim_\gamma u_\gamma, P\rangle = 0$.
\end{proof}

\begin{point}\label{Linfty is operators}
For example, $\Linfty(G)$ is a von Neumann algebra of multipliers on $\Ltwo(G)$, and its predual is $\Lone(G)$.
Let $ P_1(G)$ denote the normal states on $\Linfty(G)$.
By Sakai's result, the linear span of $ P_1(G)$ is $\Lone(G)$, 
so we conclude
\[P_1(G) = \{f\in \Lone(G) : f\geq 0, \|f\|_1 = 1\}.\]
Given $f\in P_1(G)$, let $\supp(f) = \{x\in G: f(x)\neq 0\}$.
Clearly $S(f)$ is the indicator function $\Onebb_{\supp(f)}$.
From this, we conclude $f,g \in  P_1(G)$ are orthogonal when $\supp(f)\cap \supp(g) = \varnothing$.
\end{point}

\begin{point}
Let $X$ be any von Neumann algebra.
Endow $X^*$ with the $w^*$-topology,
and let $\Mcal\subset X^*$ be the set of all states.
In the context of amenability, $\Mcal$ is traditionally called the set of means on $X$.
Notice $\|m\|=1$ for each $m\in\Mcal$, since $T \leq \|T\| I$ for any self-adjoint $T$.
Notice $\Pcal_1$ is convex and $\Mcal$ is compact convex.
\end{point}


\begin{lem}\label{Lemma P is Dense in M}
$\Pcal_1$ is dense in $\Mcal$.
\end{lem}
\begin{proof}
Suppose to the contrary that a mean $m$ lies outside $\cl(\Pcal_1)$.
Applying \cref{Lemma Hahn Banach} to the sets $\{m\}$ and $\cl(\Pcal_1)$,
obtain $T\in X$ such that
$\Re\langle m, T\rangle > \sup_{u\in\Pcal_1} \Re \langle u, T\rangle$.
Decompose $T$ into its self-adjoint parts as $T_1 + i T_2$,
so that $\langle m, T_1\rangle > \sup_{u\in\Pcal_1} \langle u, T_1\rangle$.
Letting $S = T_1 + \|T_1\| I \geq 0$, we have $\langle m, S\rangle > \sup_{u\in\Pcal_1} \langle u, S\rangle = \|S\|$,
contradicting $\|m\| = 1$.
\end{proof}
\section*{TI-nets in \texorpdfstring{$P_1(G)$}{P1(G)}}\label{TLIM(G)}
\begin{point}
Recall that $\Haar{E}$ denotes the left Haar measure of a measurable set $E\subset G$.
If $0 < |E| < \infty$, we define $\mu_E = \Onebb_E / |E| \in P_1(G)$.
Recall that the modular function $\Delta: G\to \Rbb^\times$ is a continuous homomorphism
defined by $|Ut| = |U| \Delta(t)$.
The map $f\mapsto f^\dag$ defined by $f^\dag(x) = f(x^{-1}) \Delta(x^{-1})$
is an involution of $\Lone(G)$ that sends $ P_1(G)$ to itself.
Left and right translation are defined by $l_t f(x) = f(t^{-1} x)$
and $r_t f(x) = f(xt)$, so that $l_{xy} = l_x l_y$ and $r_{xy} = r_x r_y$.
Additionally, we define $R_t f(x) = f(xt^{-1}) \Delta(t^{-1})$,
so that $R_t \mu_U = \mu_{Ut}$.
Notice $R_{xy} = R_y R_x$.
\end{point}

\begin{point}\label{Definition_fT}
For $T\in\Linfty(G)$ and $f,g\in\Lone(G)$, define $fT$ by $\langle fT, g\rangle = \langle T, f*g\rangle$,
and $T f$ by $\langle T f, g\rangle = \langle T, g*f\rangle$.
Equivalently, $fT(t) = f^\dag * T(t) =  \langle R_t f, T\rangle$, and $Tf(t) = \langle l_t f, T\rangle$.
A mean $m$ is said to be topological left invariant if $\langle m, fT - T\rangle=0$ for all $f\in P_1(G)$ and $T\in\Linfty(G)$.
Likewise, $m$ is topological right invariant if $\langle m,Tf - T\rangle = 0$, and topological two-sided invariant if it is both of the above.
The sets of topological left/right/two-sided invariant means on $\Linfty(G)$ are denoted $\TLIM(G)/ \TRIM(G)/ \TIM(G)$.
\end{point}


\begin{point}
A net $\{f_\gamma\}\subset P_1(G)$ is called a left TI-net
if $\|g*f_\gamma - f_\gamma\|_1 \to 0$ for all $g\in P_1(G)$.\footnote{
	The arguments of \cref{almost topological folner} and \cref{exists topological folner net}
	show this definition of left TI-net implies \cref{weaker TI net defn}.
	The converse is \cref{Lemma lambda_gamma is left TI-net}.
}
The analogous definitions for right/two-sided TI-nets are obvious.
\end{point}

\begin{point}
Suppose $\Kcal$ is a collection of compact sets covering $G$
with $\card{\Kcal} = \kappa$.
Let $\Kcal'$ be the set of all finite unions in $\Kcal$.
Pick any compact $U\subset G$ with nonempty interior.
Then $\Kcal'' = \{UK : K\in \Kcal'\}$ is a collection of compact sets satisfying:
(1) $\card{\Kcal''} = \kappa$.
(2) $\Kcal''$ is closed under finite unions.
(3) $\bigcup_{K\in\Kcal''}K^\circ = G$.
Hence there is no loss of generality in supposing $\Kcal$ itself satisfies (1)-(3).
It follows that every compact $C\subset G$ is contained in some $K\in\Kcal$.

Henceforth, let \(\Kcal\) satisfying (1)-(3) be fixed.
\end{point}

\begin{point}\label{Definition_Folner_Net}
Suppose $G$ is noncompact amenable.
Let $\Gamma = \{(K, n) : K\in\Kcal,\ n\in\Nbb\}$, which is a directed set with the following partial order:
$[(K,m) \preceq (J,n)] \iff [K\subseteq J$ and $m\leq n]$.
By condition (2) above, notice that each tail of $\Gamma$ has cardinality
$\kappa \cdot \Nbb = \kappa = \card{\Gamma}$, as required by \cref{Lemma Cardinality of Lambda*}
to guarantee $|\Gamma^*| = 2^{2^{\kappa}}$.

\cref{exists topological folner net} says that for each $\gamma = (K,n)$,
we can pick a compact $(K, \nicefrac{1}{n})$-left-invariant set $F_\gamma$.
In other words, letting $\lambda_\gamma = \mu_{F_\gamma}$, we have
	$\| l_x \lambda_\gamma - \lambda_\gamma\|_1 < \nicefrac1n$ for each $x\in K$.

Henceforth, let the nets $\{F_\gamma\}$ and $\{\lambda_\gamma\}$ be fixed.
\end{point}

\begin{lem}\label{Lemma lambda_gamma is left TI-net}
$\{\lambda_\gamma\}$ is a left TI-net.
\end{lem}
\begin{proof}
Pick $f\in P_1(G)$.
Since compactly supported functions are dense in $ P_1(G)$, we may suppose $f$ has compact support $C$.
Pick $K\in\Kcal$ containing $C$.
Notice \[\|f*\lambda_\gamma - \lambda_\gamma\|_1
	\leq \int f(t)\, \|l_t \lambda_\gamma - \lambda_\gamma\|_1 \dd{t}
	\leq \sup_{t\in C} \|l_t \lambda_\gamma - \lambda_\gamma\|_1.\]
In particular, $\|f*\lambda_\gamma - \lambda_\gamma\|_1 < 1/n$ whenever $\gamma \succeq (K, n)$.
\end{proof}

\begin{cor}
Let $\rho_\gamma = \lambda_\gamma^\dag$. Then $\{\rho_\gamma\}$ is a right TI-net.
\end{cor}
\begin{proof}
For any $f\in  P_1(G)$, $\|\rho_\gamma *f - \rho_\gamma\|_1 = \|f^\dag*\lambda_\gamma - \lambda_\gamma\|_1\to 0$.
\end{proof}

\begin{lem}\label{Lemma_Product_of_TI_Nets}
$\{\lambda_\gamma *\rho_\gamma\}$ is a two-sided TI-net.
\end{lem}
\begin{proof}
First of all, $\lambda_\gamma* \rho_\gamma \geq 0$ and $\|\lambda_\gamma* \rho_\gamma\|_1 = 1$,
so $\{\lambda_\gamma *\rho_\gamma\} \subset  P_1(G)$.
For each $h\in P_1(G)$,
$\|h*\lambda_\gamma* \rho_\gamma - \lambda_\gamma* \rho_\gamma\|_1
\leq \|h*\lambda_\gamma - \lambda_\gamma\|_1 \cdot \|\rho_\gamma\|_1 \to 0$
and $\|\lambda_\gamma* \rho_\gamma* h - \lambda_\gamma* \rho_\gamma\|_1
\leq \|\lambda_\gamma\| \cdot \|\rho_\gamma *h - \rho_\gamma\|_1\to 0$.
\end{proof}

The next lemma generalizes \cite[Theorem 3.2]{Chou70}.
Intuitively it says, ``we can prove facts about the entire set $\TLIM(G)$, simply by proving them about the right-translates of a single left TI-net.''
After that, we have generalizations to $\TRIM(G)$ and $\TIM(G)$.

\begin{lem}\label{Lemma_cl_conv_Xp}
For any $p\in\Gamma^*$,
let $X_p = \{\plim_\gamma [R_{t_\gamma} \lambda_\gamma] : \{t_\gamma\}\in G^{\Gamma}\}.$
Then $\cl(\conv(X_p)) = \TLIM(G)$.
In particular, if $X_{\Gamma^*}$ is the set of all limit points of all right-translates of $\{\lambda_\gamma\}$, then $\cl(\conv(X_{\Gamma^*})) = \TLIM(G)$.
\end{lem}
\begin{proof}
Suppose $m \in\TLIM(G)$ lies outside the closed convex hull of $X_p$.
As in the proof of \cref{Lemma P is Dense in M},
	there exists $T\in L_\infty(G,\Rbb)$ with $\langle m,T\rangle > \sup_{\nu\in X_p}\langle\nu,T\rangle$.
For each $\gamma$,
\[\textstyle \langle m,T\rangle
	= \langle m, \lambda_\gamma T\rangle
	\leq \|\lambda_\gamma T\|_\infty
	= \sup_t [\lambda_\gamma T(t)]
	= \sup_t \langle R_t \lambda_\gamma,\, T \rangle.\]
In particular, if $\gamma = (K,n)$, choose $t_\gamma$ so that
$\langle m, T\rangle < \langle R_{t_\gamma} \lambda_\gamma,\, T\rangle + 1/n$.
Now \[\textstyle \langle m,T\rangle \leq \langle\plim_\gamma [R_{t_\gamma}\lambda_\gamma],\, T\rangle,
\ \ \text{ a contradiction}.\qedhere\]
\end{proof}

\begin{lem}
For any $p\in\Gamma^*$,
let $X_p = \{\plim_\gamma [l_{t_\gamma} \rho_\gamma] : \{t_\gamma\}\in G^{\Gamma}\}.$
Then $\cl(\conv(X_p)) = \TRIM(G)$.
\end{lem}
\begin{proof}
Essentially the same as the proof of \cref{Lemma_cl_conv_Xp}, modified as follows:
For each $\gamma$,
\[\textstyle\langle m,T\rangle
	= \langle m, T\rho_\gamma\rangle
	\leq \|T \rho_\gamma\|_\infty
	= \sup_t [T\rho_\gamma(t)]
	= \sup_t \langle l_t \rho_\gamma,\ T \rangle.\qedhere\]
\end{proof}

\begin{lem}
For any $p\in\Gamma^*$, 
let $X_p = \{\plim_\gamma [\lambda_\gamma *(l_{t_\gamma} \rho_\gamma)]: \{t_\gamma\}\in G^{\Gamma}\}$.
Then $\cl(\conv(X_p)) = \TIM(G)$.
\end{lem}
\begin{proof}
Essentially the same as the proof of \cref{Lemma_cl_conv_Xp}, modified as follows:
For each $\gamma$,
	\[\textstyle\langle m,T\rangle
	= \langle m, \lambda_\gamma T \rho_\gamma\rangle
	\leq \|\lambda_\gamma T \rho_\gamma\|_\infty
	=\sup_t[\lambda_\gamma T \rho_\gamma(t)]
	=\sup_t \langle \lambda_\gamma * (l_t \rho_\gamma),\ T\rangle.\qedhere\]
\end{proof}

\begin{lem} \label{Lemma Disjoint Folner}
There exists $\{t_\gamma\}\in G^{\Gamma}$ such that $\{F_\gamma t_\gamma F_\gamma^{-1}\}$ are mutually disjoint.
Since $\supp(\lambda_\gamma *(l_t \rho_\gamma)) \subset F_\gamma t F_\gamma^{-1}$, it follows from \cref{Linfty is operators} and \cref{Lemma_Product_of_TI_Nets} that $\{\lambda_\gamma *(l_t \rho_\gamma)\}$ is an orthogonal TI-net.
\end{lem}
\begin{proof}
Since $|\Gamma| = \kappa$, let $(\Gamma, <)$ denote the well-ordering of $\Gamma$ induced by some bijection with $\kappa$.
Let $t_0 = e$.
As induction hypothesis, suppose $\{F_\gamma t_\gamma F_\gamma^{-1}\}_{\gamma < \alpha}$ are disjoint, where $\alpha\in\Gamma$.
If it is impossible to find $t_\alpha$ such that $\{F_\gamma t_\gamma F_\gamma^{-1}\}_{\gamma \leq \alpha}$ are disjoint, then
$\{F_\alpha^{-1} F_\gamma t_\gamma F_\gamma^{-1} F_\alpha\}_{\gamma < \alpha}$ is a collection of compact sets of cardinality less than $\kappa$ covering $G$, a contradiction.
\end{proof}

\begin{lem}\label{L-UpperBoundSizeOfTlim}
Let $K\subset G$ be any compact set with nonempty interior, and $\{K t_\gamma\}_{\gamma < \kappa}$ be a covering of $G$ by translates of $K$.
For each $n$, let $\lambda_n\in P_1(G)$ be $(K^{-1}, \nicefrac1n)$-left-invariant, and let $\rho_n = \lambda_n^\dag$.
Finally, let $X = \left\{ R_{t_\gamma}\rho_n : n\in\Nbb,\ \gamma < \kappa\right\}$. 
Then $\TLIM(G) \subset \cl(\conv(X))$, and $\#(\cl(\conv(X))) \leq 2^{2^\kappa}$.
\end{lem}
\begin{proof}
Suppose $m\in \TLIM(G)$ lies outside the closed convex hull of $X$.
As in \cref{Lemma P is Dense in M}, there exists $T\in L_\infty(G,\Rbb)$ and $\epsilon > 0$
	with $\langle m, T\rangle > 2\epsilon +  \langle R_{t_\gamma} \rho_n, T\rangle$ for all $n\in\Nbb$ and all $\gamma < \kappa$.
Let $n$ be large enough that $\|T\|_\infty / n < \epsilon$.
Let $s$ be chosen so that $\sup_t \langle R_t \rho_n, T\rangle < \epsilon + \langle R_{s} \rho_n, T\rangle$. Say $s\in K t_{\alpha}$.
By definition of $\rho_n$,
	$\langle R_{s} \rho_n, T\rangle <
  \langle R_{t_\alpha} \rho_n, T\rangle + \|T\|_\infty/n$.
Now \[\textstyle \langle m, T\rangle
	= \langle m, \rho_n T\rangle
	\leq \|\rho_n T\|_\infty
	= \sup_t \langle R_t \rho_n, T\rangle
	< 2\epsilon + \langle R_{t_\alpha} \rho_n, T\rangle,\] contradicting the choice of $m$.
This proves $\TLIM(G) \subset \cl(\conv(X))$.

Let $\conv_{\Qbb}(X)$ denote the set of all finite convex combinations in $X$ with rational coefficients.
Evidently $\card{\conv_{\Qbb}(X)} = \card{X} = \kappa$, and $\cl(\conv(X)) = \cl(\conv_{\Qbb}(X))$.
Since $\cl(\conv(X))$ is a regular Hausdorff space with dense subset $\conv_{\Qbb}(X)$, it has cardinality at most $2^{2^{\card{\conv_{\Qbb}(X)} }} = 2^{2^\kappa}$.
\end{proof}

\begin{thm}\label{TIM cardinality result}
Let $G$ be a locally compact amenable group, with $1 < \kappa =$
the smallest cardinality of a covering of $G$ by compact subsets.
Then $|\TIM(G)| = |\TLIM(G)| = 2^{2^\kappa}$.
\end{thm}
\begin{proof}
By \cref{Lemma Disjoint Folner}, there exists an orthogonal TI-net $\{\phi_\gamma\}$.
By \cref{Lemma Orthogonal Implies Injection},
the map $p\mapsto \plim_\gamma \phi_\gamma$ is one-to-one from $\Gamma^*$ to $\TIM(G)$,
so $\card{\TIM(G)} \geq \card{\Gamma^*}$.
By \cref{Lemma Cardinality of Lambda*}, $\card{\Gamma^*} = 2^{2^\kappa}$.
Of course, $\TIM(G) \subset \TLIM(G)$, so $2^{2^\kappa} \leq \card{\TIM(G)} \leq \card{\TLIM(G)}$.
\cref{L-UpperBoundSizeOfTlim} yields the opposite inequality.
\end{proof}
\section*{A Proof of Paterson's Conjecture} \label{S-PatersonsConjecture}
Let $t^g$ be shorthand for $g t g^{-1}$.
We write $G\in[\FC]^-$ to signify that each conjugacy class $t^G = \{t^g : g\in G\}$ is precompact.
When $G\in [FC]^-$ is furthermore discrete, each conjugacy class must be finite.
In this case, we write $G\in \FC$.

\begin{lem} \label{T-FCisAmenable}
If $G\in \FC$, then $G$ is amenable.
\end{lem}
\begin{proof}
It suffices to show that each finitely generated subgroup is amenable.
Suppose $K\subset G$ is finite, and let $\langle K \rangle$ denote the subgroup generated by $K$.
For any $x\in \langle K\rangle$, let $C(x) = \{y \in \langle K\rangle : x^y = x\}$.
Evidently $C(x)$ is a subgroup, whose right cosets correspond to the (finitely many) distinct values of $x^y$.
Therefore $[\langle K\rangle : C(x)]$ is finite.
Let $Z$ denote the center of $\langle K\rangle$.
Clearly $Z = \bigcap_{x \in K} C(x)$.
Thus $[\langle K\rangle: Z] \leq \prod_{x\in K} [\langle K\rangle : C(x)]$ is finite.
Since $\langle K\rangle$ is finite-by-abelian, it is amenable.
\end{proof}

\begin{thm} \label{T-Structure}
$G\in[\FC]^-$ if and only if $G$ is a compact extension of $\Rbb^n \times D$, for some $D\in\FC$ and $n\in\Nbb$.
\end{thm}
\begin{proof}
This is \cite[Theorem 2.2]{liukkonen}.
\end{proof}

\begin{cor}
If $G\in [\FC]^-$, then $G$ is amenable.
\end{cor}
\begin{proof}
In light of \cref{T-FCisAmenable}, any group of the form $\Rbb^n \times D$ is amenable.
Compact groups are amenable, so compact extensions of amenable groups are amenable.
Hence the result follows from \cref{T-Structure}
\end{proof}

\begin{cor} \label{C-FCisUnimodular}
If $G\in [\FC]^-$, then $G$ is unimodular.
\end{cor}
\begin{proof}
Groups of the form $\Rbb^n \times D$ are unimodular, since $D$ is discrete,
and compact extensions of unimodular groups are unimodular.
Hence the result follows from \cref{T-Structure}
\end{proof}

\begin{lem}\label{Lemma_CG_is_precompact}
If $G\in[\FC]^-$, and $C\subset G$ is compact, then $C^G = \{ c^g : c\in C, g\in G\}$ is precompact.

In \cite{milnes}, Milnes observes that this lemma would imply \cref{Theorem If FC Bar then...}.
He is unable to prove it, apparently because he is unaware of \cref{T-Structure}.
\end{lem}
\begin{proof}
By \cref{T-Structure}, let $G/K = \Rbb^n \times D$, where $K\normal G$ is a compact normal subgroup.
Let $\pi: G\to G/K$ denote the canonical epimorphism.
Pick $C\subset G$ compact.
Notice $C^G$ is precompact if $\pi(C^G)$ is, because the kernel of $\pi$ is compact.
Since $\pi(C)$ is compact, $\pi(C) \subset B \times F$ for some box $B\subset\Rbb^n$ and finite $F\subset D$.
Now $\pi(C^G) = \pi(C)^{\pi(G)} \subset (B\times F)^{\Rbb^n \times D} = B\times F^D$.
Evidently $F^D$ is finite, hence $B\times F^D$ is compact, proving $\pi(C^G)$ is precompact.
\end{proof}
\begin{point}\label{Symmetric_Folner_Net}
As in \cref{Definition_Folner_Net}, let $\{F_\gamma\}_{\gamma\in\Gamma}$ be a F\o{}lner net for $G$,
and $\lambda_\gamma = \mu_{F_\gamma}$ the corresponding left TI-net.
Assuming $G$ is unimodular, by \cite[Theorem 4.4]{Chou70} we can choose $F_\gamma$ to be symmetric.\footnote{
	When $G$ is \textit{not} unimodular, it's rather trivial that $F_\gamma$ can be chosen symmetric,
	see \cite[Theorem 2]{Emerson1974}.
	In this case, however, symmetry is not enough to ensure $F_\gamma$ is a two-sided F\o{}lner net;
	it's not even possible for a set of finite measure to be $(\{x\}, \nicefrac12)$-right-invariant
	when $\Delta(x) < \nicefrac12$!
}
Hence $\lambda_\gamma(x) = \lambda_\gamma(x^{-1}) = \lambda_\gamma^\dag(x)$,
and $\{\lambda_\gamma\}$ is a two-sided TI-net.
\end{point}

\begin{thm}[Following Milnes]\label{Theorem If FC Bar then...}
If $G\in[\FC]^-$, then $\TLIM(G) \subset \TIM(G)$.
\end{thm}
\begin{proof}
By \cref{C-FCisUnimodular}, $G$ is unimodular.
As above, take $\{\lambda_\gamma\} = \{\lambda_\gamma^\dag\}$ to be a TI-net.
Recall \cref{Lemma_cl_conv_Xp}, which says $\cl(\conv(X_p)) = \TLIM(G)$.
So it suffices to prove $X_p \subset \TIM(G)$.
To this end, we will show $\{R_{t_\gamma} \lambda_\gamma\}$ is a right TI-net for any $\{t_\gamma\}\in G^{\Gamma}$.
Let $x_\gamma$ be shorthand for $t_\gamma x t_\gamma^{-1}$.
Now \[\Haar{r_x R_{t_\gamma} \lambda_\gamma - R_{t_\gamma} \lambda_\gamma }
	= \tfrac{\Haar{F_\gamma t_\gamma x^{-1} \Sdiff F_\gamma t_\gamma}}{\Haar{F_\gamma t_\gamma}}
	= \tfrac{\Haar{F_\gamma x_\gamma^{-1} \Sdiff F_\gamma}}{\Haar{F_\gamma}}
	= \| r_{x_\gamma} \lambda_\gamma - \lambda_\gamma \|_1
	= \| l_{x_\gamma} \lambda_\gamma - \lambda_\gamma \|_1.\]
For any compact $C\subset G$, notice $\{x_\gamma : x\in C, \gamma\in\Gamma\} \subset C^G$,
which is precompact by \cref{Lemma_CG_is_precompact}.
Hence \[\sup_{x\in C} \|r_x R_{t_\gamma} \lambda_\gamma - R_{t_\gamma} \lambda_\gamma \|_1
	\leq \sup_{y\in C^G} \|l_y \lambda_\gamma - \lambda_\gamma\|\to 0.\]
By the same argument as \cref{Lemma lambda_gamma is left TI-net},
$\|(R_{t_\gamma} \lambda_\gamma)* f- R_{t_\gamma} \lambda_\gamma\|_1 \to 0$ for any $f\in P_1(G)$.
\end{proof}

\begin{lem}
If $x^G$ is not precompact, then there is a $\sigma$-compact open subgroup $H\leq G$ such that $x^H$ is not precompact.
\end{lem}
\begin{proof}
Let $U$ be any compact neighborhood of $t_0 = e$.
Inductively pick $t_{n+1}$ such that $x^{t_{n+1}}U \cap \{x^{t_0}, \hdots, x^{t_n}\}U = \varnothing$.
This is possible, otherwise $x^G$ is contained in the compact set $\{x^{t_0}, \hdots, x^{t_n}\} U U^{-1}$, a contradiction.
Now the sequence $\{x^{t_0}, x^{t_1}, \hdots\}$ does not accumulate anywhere, so it escapes any compact set.
Take $H$ to be the subgroup generated by $\{x^{t_0}, x^{t_1}, \hdots\} U$.
\end{proof}

\begin{thm}\label{Theorem Paterson Converse}
If $G$ has an element $x$ such that $x^G$ is not precompact, then $\TLIM(G) \neq \TIM(G)$.
\end{thm}
\begin{proof}
Let $H \leq G$ be any $\sigma$-compact open subgroup such that $x^H$ is not precompact.
Pick a sequence of compact sets $\{H_n\}$ such that
$H_n \subset H_{n+1}^\circ$ for each $n$,
and $\bigcup_n H_n = H$.
It follows that any compact $K\subset H$ is contained in some $H_n$.

Inductively construct a sequence $\{c_n\}\subset H$ satisfying the following:
\\(1) $c_n \in H\setminus K_n$, where $K_n = \bigcup_{m < n} H_n^{-1} H_m c_m \{x, x^{-1}\}$.
\\(2) $(c_n x c_n^{-1}) \not\in H_n^{-1} H_n$.
\\Since $K_n$ is compact, $x^{H \setminus K_n}$ is not precompact, and in particular escapes $H_n^{-1} H_n$.
Thus it is possible to satisfy (1) and (2) simultaneously.

Now $A = \bigcup_n H_n c_n$ and $B = \bigcup_n H_n c_n x$ are easily seen to be disjoint:
If they are not disjoint, then $H_n c_n x \cap H_m c_m \neq \varnothing$ for some $n,m$.
If $n > m$, then $c_n \in H_n^{-1} H_m c_m x^{-1}$, violating (1).
If $n < m$, then $c_m \in H_m^{-1} H_n c_n x$, violating (1).
If $n=m$, then $c_n x c_n^{-1} \in H_n^{-1} H_n$, violating (2).

Let $T$ be a transversal for $G/H$.
Notice $TA \cap TB = \varnothing$, since $A,B$ are disjoint subsets of $H$.
Define $\pi: TH \to H$ by $\pi(th) = h$, which is continuous since $H$ is open.
Let $\{F_\gamma\}$ be a F\o{}lner net for $G$.
Each $F_\gamma$ is compact, hence $\pi(F_\gamma) \subset H_{n(\gamma)}$ for some $n(\gamma)$.
Let $c_\gamma = c_{n(\gamma)}$.
Now $\pi(F_\gamma c_\gamma) = \pi(F_\gamma)c_\gamma \subset A$, hence $F_\gamma c_\gamma \subset \pi^{-1}(A) = TA$.
Likewise $F_\gamma c_\gamma x \subset TB$.
Since this holds for each $\gamma$, $C = \bigcup_\gamma F_\gamma t_\gamma \subset TA$
and $D = \bigcup_\gamma F_\gamma t_\gamma x \subset TB$.
Thus $C \cap D = \varnothing$.
Let $m = \plim_\gamma[R_{t_\gamma} \mu_{F_\gamma}]$ for some $p\in\Gamma^*$.
Now $m(\Onebb_C) = 1$, but $m(r_x \Onebb_C) = m(\Onebb_D) = 0$, hence $m\in \TLIM(G) \setminus \TIM(G)$.
\end{proof}
\section*{Background on Fourier Algebra}
The algebras $\AG$ and $\VN$ are studied in the influential paper \cite{Eymard}.
Chapter 2 of \cite{Lau-Kaniuth} gives a modern treatment of the same material in English.

\begin{point}
Let $\lambda$ denote the left-regular representation.
By definition, $\VN = \lambda[G]'' = \lambda[\Lone(G)]''$.
$ P_1(\Ghat)$ denotes the set of normal states on $\VN$,
and the linear span of $ P_1(\Ghat)$ is denoted by $\AG$.
It turns out that $ P_1(\Ghat)$ consists of all vector states of the form
$\omega_f(T) = \langle Tf,f\rangle$ where $\|f\|_2 = 1$.
By the polarization identity, $\AG$ consists of all functionals of the form
$\omega_{f,g}(T) = \langle Tf,g\rangle$, with $f,g\in\Ltwo(G)$.
\end{point}

\begin{point}
$\AG$ can also be characterized as the set of all functions
\[u_{f,g}(x) = \langle \lambda(x) f, g\rangle = \overline{g} * \check{f}(x),\ \ \text{where } f,g\in\Ltwo(G).\]
In this way, $\AG \subset C_0(G)$.
The duality with $\VN$ is given by $\langle T, u_{f,g}\rangle = \langle Tf,g\rangle$.
In particular, $\langle \lambda(h), u\rangle = \int h(x) u(x)\dd{x}$ for all $h\in\Lone(G)$.

$\AG$ is a commutative Banach algebra with pointwise operations
and norm $\|u\| = \sup_{\|T\| = 1} |\langle T, u\rangle|$.
Equivalently, $\|u\|$ is the sup of $\left|\int h(x) u(x)\dd{x}\right|$ over all $h\in\Lone(G)$
such that $\|\lambda(h)\|\leq 1$.
Since $\|\lambda(h)\|\leq \|h\|_1$, we conclude $\|u\| \geq \|u\|_\infty$.
\end{point}

\begin{point}
If $u\in\AG$ is positive as a functional on $\VN$, we write $u\geq 0$.
It's easy to see $u\geq 0$
iff $u = u_{f,f}$ for some $f\in\Ltwo(G)$
iff $\|u\|_\infty = u(e) = \langle I, u\rangle = \|u\|$.
In particular, $ P_1(\Ghat) = \{u\in\AG : \|u\| = u(e) = 1\}$ is closed under multiplication.
Suppose $\|f\|_2 = \|g\|_2 = 1$. Then
\begin{align*}
\|u_{f,f} - u_{g,g}\|
	&= \sup_{\|T\|=1} |\langle Tf, f\rangle - \langle Tg, g\rangle|
\\&	\leq \sup_{\|T\|=1} |\langle T(f-g), f\rangle| + |\langle Tg, (f-g)\rangle|
\\&	\leq 2\|f-g\|_2.
\end{align*}
\end{point}

\section*{TI-nets in \texorpdfstring{$P_1(\widehat{G})$}{P1(G\^{})}}\label{VN(G)}
\begin{point}
For $T\in\VN$ and $u,v\in\AG$, define $uT$ by $\langle uT, v\rangle = \langle T, uv\rangle$.
A mean $m$ is said to be topological invariant if
$\langle m, uT\rangle  = \langle m, T\rangle$ for all $u\in P_1(\Ghat)$ and $T\in\VN$.
$\TIM(\Ghat)$ denotes the set of topological invariant means on $\VN$.
There is no distinction between left/right topological invariant means, since multiplication in $ P_1(\Ghat)$ is commutative.
\end{point}

\begin{point}
A net $\{u_\gamma\}\subset P_1(\Ghat)$ is called a weak TI-net if
$\lim_\gamma \langle vu_\gamma - u_\gamma, T\rangle = 0$ for all $v\in P_1(\Ghat)$ and $T\in\VN$.
By \cref{Lemma P is Dense in M}, topological invariant means are precisely the limit points of weak TI-nets.
$\{u_\gamma\}$ is simply called a TI-net if $\lim_\gamma \|vu_\gamma - u_\gamma\| = 0$ for all $v\in P_1(\Ghat)$.
\end{point}

\begin{lem}\label{Lemma Small Supp Means TI-net}
Pick any net $\{u_\gamma\}\subset  P_1(\Ghat)$.
Suppose that $\supp(u_\gamma)$ is eventually small enough to fit in any $V\in \Cscr(G)$.
Then $\{u_\gamma\}$ is a TI-net.
\end{lem}
\begin{proof}
This is \cite[Proposition 3]{renaud}.
I'll repeat the proof for convenience.

Pick $u\in P_1(\Ghat)$.
Since compactly supported functions are dense in $\AG$, it suffices to consider the case when $u$ has compact support $C$.
$\AG$ is regular by \cite[Proposition 2.3.2]{Lau-Kaniuth}, so we can pick $v\in\AG$ with $v\equiv 1$ on $C$.
Since $u-v(e) = 0$, by \cite[Lemma 2.3.7]{Lau-Kaniuth} we can find some $w\in\AG$ with $\|(u-v) - w\| < \epsilon$ and $w\equiv 0$ on some neighborhood $W$ of $e$.
Suppose $\gamma$ is large enough that $\supp(u_\gamma) \subset W\cap C$.
Then $v u_\gamma = u_\gamma$, and $wu_\gamma = 0$,
hence \[\|u u_\gamma - u_\gamma\|
	= \|(u-v) u_\gamma\|
	\leq \|(u-v-w)u_\gamma\| + \|w u_\gamma\|
	\leq \|(u-v-w)\| < \epsilon.\qedhere\]
\end{proof}

\begin{point}
Let $\Cscr(G)$ denote the compact, symmetric neighborhoods of $e$.
As a consequence of \cref{Lemma Small Supp Means TI-net}, we see that $ P_1(\Ghat)$ has TI-nets for any $G$,
in contrast to $ P_1(G)$ which has TI-nets only when $G$ is amenable:
If $\{U_\gamma\}\subset \Cscr(G)$ is a neighborhood basis at $e$, directed by inclusion,
let $u_\gamma = (\Onebb_{U_\gamma} * \Onebb_{U_\gamma}) / \Haar{U_\gamma}$.
Since $\supp(u_\gamma) \subset U_\gamma^2$, and $\{U_\gamma^2\}$ is a neighborhood basis at $e$,
$\{u_\gamma\}$ is a TI-net.

\cref{Lemma Small Supp Means TI-net} has a sort of converse, given by the following lemma.
\end{point}

\begin{lem}\label{Lemma Every TIM(Ghat) is limit of TI-net}
Every $m\in \TIM(\Ghat)$ is the limit of a net $\{u_\gamma\}$
such that $\supp(u_\gamma)$ is eventually small enough to fit in any $V\in\Cscr(G)$.
\end{lem}
\begin{proof}
Suppose to the contrary that there exist $V\in\Cscr(G)$, $T\in\VN$, and $\epsilon > 0$
such that $|\langle m-u, T\rangle| > \epsilon$ for all $u\in P_1(\Ghat)$ with $\supp(u) \subset V$.
Fix any $u\in P_1(\Ghat)$ with $\supp(u) \subset V$.
By \cref{Lemma P is Dense in M}, pick $v\in P_1(\Ghat)$ with $|\langle m - v, T\rangle | < \epsilon$,
and $|\langle m - v, uT\rangle| < \epsilon$.
Now $vu\in \Pcal(\Ghat)$ with $\supp(vu)\subset V$.
But $|\langle m - vu, T\rangle|
	=|\langle m, T\rangle - \langle vu, T\rangle|
	= |\langle m, uT\rangle - \langle v, uT\rangle|
	= |\langle m - v, uT\rangle| < \epsilon$, a contradiction.
\end{proof}

\begin{cor}
Every $m\in\TIM(\Ghat)$ is the limit of a TI-net.
\end{cor}

\begin{cor}\label{Corollary Upper Bound on TIM Ghat}
$\card{\TIM(\Ghat)} \leq 2^{2^\mu}$.
\end{cor}
\begin{proof}
Pick any $V\in\Cscr(G)$.
By the previous lemma, $\TIM(\Ghat)$ is in the closure of $X = \{u\in  P_1(\Ghat) : \supp(u) \subset V\}$.
Therefore, it suffices to find a set $D$ with $\#(D) = \mu$ and $X \subset \cl(D)$.

Let $B = \{f\in\Ltwo(G) : \|f\|_2 = 1,\ \supp(f) \subset V\}$.
Using \cref{dense subset of haar sets}, it is easy to show $B$ has a dense subset $B'$ of cardinality $\mu$.
Now $D = \{u_{f,f} : f\in B'\}$ is the desired set.
\end{proof}

\begin{point}\label{rank one support}
Let $u_{f,f}$ be a normal state on $\VN$, and $P$ the projection onto the linear span of $f$.
Since $P$ has rank one, it is clearly the inf of all projections $Q$
such that $1 = \langle Q f, f\rangle = \langle Q, u_{f,f}\rangle$.
In the terminology of \cref{Definition of Support}, $P$ is the support of $u_{f,f}$.
Thus a family $\{u_{f_\alpha, f_\alpha}\} \subset \AG$ is orthogonal
if and only if $\{f_\alpha\} \subset \Ltwo(G)$ is orthogonal.

Apparently neither Chou nor Hu made this observation.
If it seems obvious, it is because we took care to express normal states as $u_{f,f}$.
For example, in order to construct an orthogonal TI-sequence, \cite[\pp 210-213]{Chou82} takes three pages to describe a generalized Gram-Schmidt procedure for normal states on a von Neumann algebra.
This procedure does not generalize beyond sequences, and \cite{Hu95} is unable to construct an orthogonal net when $\mu > \Nbb$.
\end{point}

\begin{lem}\label{Lemma_Orthogonal_TI_Sequence}
When $\mu = \Nbb$, $ P_1(\Ghat)$ contains an orthogonal TI-sequence.
\end{lem}
\begin{proof}
Let $\{U_n\} \subset \Cscr(G)$ be a descending neighborhood basis at $e$ such that $|U_n \setminus U_{n+1}| > 0$.
Let $V_n = U_n \setminus U_{n+1}$ and $f_n = \Onebb_{V_n} / \Haar{V_n}^{\nicefrac12}$.
In light of \cref{Lemma Small Supp Means TI-net} and \cref{rank one support},
$\{u_{f_n, f_n}\}$ is the desired sequence.
\end{proof}

\begin{lem}\label{Lemma_Minimal_Cardinality_Subbasis}
Let $\Ucal \subset \Cscr(G)$. If $\bigcap \Ucal = \{e\}$, then $\Ucal$ is a neighborhood sub-basis at $e$,
hence $\card{\Ucal} \geq \mu$.
\end{lem}
\begin{proof}
Suppose $\bigcap \Ucal = \{e\}$, and let $V$ be any open neighborhood of $e$.
Since $\varnothing = \{e\}\cap V^C = \bigcap_{U\in \Ucal} \big[U\cap V^C\big]$ is an intersection of compact sets, it follows that some finite sub-intersection is also empty.
In other words, $U_1\cap \hdots \cap U_n \subset V$ for some $U_1, \hdots, U_n \in \Ucal$.
\end{proof}

\begin{lem}\label{Lemma KK}
Let $\Ucal \subset \Cscr(G)$.
Suppose each $U \in \Ucal$ has a ``successor'' $V\in \Ucal$, such that $V^2 \subset U$.
Then $H = \bigcap \Ucal$ is a compact subgroup of $G$.
\end{lem}
\begin{proof}
$H$ is compact, because it is the intersection of compact sets in a Hausdorff space.
Suppose $x,y \in H$. Suppose $U, V \in \Ucal$ with $V^2 \subset U$.
Since $x,y \in H\subset V$, $xy \in V^2 \subset U$.
Since $U$ was arbitrary, we conclude $xy\in H$.
Likewise, $x^{-1}\in H$ because each $U$ is symmetric.
\end{proof}

\begin{point}
Suppose $H$ is a compact subgroup of $G$, with normalized Haar measure $\nu$.
Let $L^2(H \backslash G)$ be the set of functions in $L^2(G)$ that are constant on right-cosets of $H$.
Define $Pf(x) = \int f(hx) \dd\nu(h)$.
Now it is routine to check that $P$ is the orthogonal projection onto $L^2(H \backslash G)$.
The only interesting detail, needed to prove $P = P^*$, is that $\int f(hx) \dd\nu(h) = \int f(h^{-1}x) \dd\nu(h)$ because $H$ is unimodular.
\end{point}

\begin{point}\label{KK Small U and H}
Suppose $\{H_\gamma\} = \{H_\gamma\}_{\gamma < \mu}$ is a descending chain of compact subgroups.
Let $P_\gamma$ denote the orthogonal projection onto $L^2(H_\gamma \backslash G)$.
Then $\{P_\gamma\}$ is an ascending chain of projections, and $\{P_{\gamma + 1} - P_\gamma\}$ is a chain of mutually orthogonal projections.
Suppose we construct functions $\{f_\gamma\} \subset \Ltwo(G)$ with $f_\gamma \in L^2(H_{\gamma + 1} \backslash G) - L^2(H_\gamma \backslash G)$, so that $(P_{\gamma+1} - P_\gamma) f_\gamma \neq 0$.
Letting $g_\gamma = (P_{\gamma+1} - P_\gamma) f_\gamma$,
we see $\{u_\gamma\}= \bigl\{u_{g_\gamma, g_\gamma} / \|g_\gamma\|_2^2\bigr\}$ is a chain of mutually orthogonal functions in $ P_1(\Ghat)$, since $S(u_\gamma) \leq P_{\gamma+1} - P_\gamma$.

Consider the condition $f_\gamma \in L^2(H_{\gamma+1} \backslash G) - L^2(H_\gamma \backslash G)$.
How can we achieve this?
Let $\nu_\gamma$ denote the Haar measure of $H_\gamma$.
Pick $U_\gamma \in \Cscr(G)$ small enough that $\nu_\gamma(U_\gamma^4 \cap H_\gamma) < 1$.
Construct $H_{\gamma + 1}$ small enough that $H_{\gamma + 1} \subset U_\gamma$, and let $f_\gamma = \Onebb_{H_{\gamma + 1} U_\gamma}$. Obviously $f_\gamma \in L^2(H_{\gamma+1} \backslash G)$.
On the other hand,
$P_\gamma f_\gamma(x) = \nu_\gamma(\{h\in H_\alpha : hx \in H_{\gamma + 1} U_\gamma\}) < 1 = f_\gamma(x)$
for any $x\in H_{\gamma + 1} U_\gamma$.
Hence $P_\gamma f_\gamma \neq f_\gamma$ and $f_\gamma\not\in L^2(H_\gamma \backslash G)$.

We need to add one more detail to our construction, to make the chain of orthogonal functions $\{u_\gamma\}$ into a TI-net.
Namely, suppose $\{V_\gamma\}\subset\Cscr(G)$ is a neighborhood basis at $e$, and $\{U_\gamma, H_\gamma\}$ satisfy $H_\gamma U_\gamma^2 \subset V_\gamma$.
Then $\supp(g_\gamma) \subset H_\gamma H_{\gamma+1} U_\gamma \subset V_\gamma$,
hence $\supp(u_\gamma) \subset V_\gamma^2$.
Let $\Gamma$ denote the set of ordinals less than $\mu$, ordered by $\alpha\prec\gamma \iff V_\gamma \subset V_\alpha$.
By \cref{Lemma Small Supp Means TI-net}, $\{u_\gamma\}_{\gamma\in\Gamma}$ is a TI-net.
Notice that each tail of $\Gamma$ has cardinality $|\Gamma| = \mu$, as required by \cref{Lemma Cardinality of Lambda*}.
\end{point}

\begin{lem}\label{Lemma_Orthogonal_TI_Net}
If $\mu > \Nbb$, $ P_1(\Ghat)$ contains an orthogonal TI-net of cardinality $\mu$.
\end{lem}
\begin{proof}
Fix $\{V_\gamma\}_{\gamma < \mu} \subset\Cscr(G)$, a well-ordered neighborhood basis at $e$.
Of course this well-ordering has no topological meaning, but it's necessary for transfinite induction.
The purpose of our induction is to select $\{U_\gamma, H_\gamma\}_{\gamma < \mu}$
as in \cref{KK Small U and H}, from which the lemma clearly follows.

For $0 \leq \beta < \mu$, suppose we have picked $\{U_\gamma, H_\gamma\}_{\gamma < \beta}$
such that (1)-(4) hold for all $\gamma < \beta$:
(1) $H_\gamma$ is the intersection of $(\gamma+1)\cdot \Nbb$ elements of $\Cscr(G)$,
and is a subgroup of each previous $H_\alpha$.
(2) If $\gamma=\alpha + 1$ is a successor ordinal, then $H_\gamma \subset U_\alpha$.
(3) $H_\gamma U_\gamma^2 \subset V_\gamma$.
(4) $\nu_\gamma(H_\gamma \cap U_\gamma^4) < 1$.

Let $\{W_n\}_{n\in\Nbb}\subset \Cscr(G)$ be any chain with $W_0 \subset V_\beta$
and $W_{n+1}^2 \subset W_n$ for all $n$.
If $\beta = \alpha + 1$, we may suppose $W_0 \subset U_\alpha$.
Let $H_\beta = \big(\bigcap_{n\in\Nbb}W_n\big) \cap \big(\bigcap_{\gamma < \beta} H_\gamma\big)$.
Notice $H_\beta W_2^2 \subset W_0 \subset V_\beta$.
Since $H_\beta$ is the intersection of $(\beta+1) \cdot \Nbb < \mu$ elements of $\Cscr(G)$,
it follows from \cref{Lemma_Minimal_Cardinality_Subbasis} that $H_\beta \neq \{e\}$.
In particular, it is possible to pick $U\in \Cscr(G)$ with $\nu_\beta (H_\beta \cap U^4) < 1$.
Let $U_\beta = U \cap W_2$.
Now $U_\beta, H_\beta$ satisfy (1)-(4).
\end{proof}

\begin{thm}
Let $G$ be any locally compact group, and $1 < \mu$ the smallest cardinality of a neighborhood basis about $e\in G$.
Then $\card{\TIM(\Ghat)} = 2^{2^\mu}$.
\end{thm}
\begin{proof}
Let $\{u_\gamma\}_{\gamma\in\Gamma}$ be the orthogonal TI-net of \cref{Lemma_Orthogonal_TI_Sequence}
or \cref{Lemma_Orthogonal_TI_Net}, depending on $\mu$.
By \cref{Lemma Orthogonal Implies Injection}, $p\mapsto \plim_\gamma u_\gamma$
is an injection of $\Gamma^*$ into $\TIM(\Ghat)$.
Thus $\card{\TIM(\Ghat)} \geq \card{\Gamma^*} = 2^{2^\mu}$ by \cref{Lemma Cardinality of Lambda*}.
The opposite inequality is \cref{Corollary Upper Bound on TIM Ghat}.
\end{proof}
\chapter{When is an invariant mean the limit of a F\o{}lner net?}\label{tia paper}
The following chapter is based on my paper \cite{HopfenspergerTLIM0}.

\medskip

\noindent\textbf{\large Abstract}

\medskip\noindent
Let $G$ be a locally compact amenable group, $\TLIM(G)$
	the topological left-invariant means on $G$,
	and $\TLIM_0(G)$ the limit points of F\o{}lner-nets.
I show that $\TLIM_0(G) = \TLIM(G)$ unless $G$ is $\sigma$-compact non-unimodular,
	in which case $\TLIM_0(G) \neq \TLIM(G)$.
This improves a 1970 result of Chou and a 2009 result of Hindman and Strauss.
I consider the analogous problem for the non-topological left-invariant means,
	and give a short construction of a net converging to invariance ``weakly but not strongly,''
	simplifying the proof of a 2001 result of Rosenblatt and Willis.

\section*{History}

In this chapter, \(G\) is always a locally compact group.
The left Haar measure of \(E\subset G\) is denoted \(\Haar{E}\).
The set of means on \(\Linfty(G)\) is
$$\Means(G) = \{\mu\in\LinftyS(G) : \|\mu\| = 1,\ \mu\geq 0\},$$
	which is endowed with the \(w^*\)-topology to make it compact.
Regarding \(\Lone(G)\) as a subset of \(\LinftyS(G)\),
define \(P_1(G) = \set*{f\in\Lone(G) : \|f\|_1 = 1,\ f\geq 0}\).

\begin{prop}[{\cref{P1(G) is dense in M(G)}}]\label{Prop M1 Dense in M}
\(P_1(G)\) is dense in \(\Means(G)\).
\end{prop}

Define left-translation of functions by \(l_x\phi(y) = \phi(x^{-1}y)\).
The set of left-invariant means on \(\Linfty\) is
\begin{center}
\(\LIM(G) =
\set*{\mu\in\Means(G) : \paren*{\forall x\in G}\ \paren*{\forall \phi\in\Linfty}\ 
	\mu(\phi) = \mu(l_x \phi)}\).
\end{center}
\(G\) is said to be amenable if \(\LIM(G)\) is nonempty.
\cref{Prop M1 Dense in M} shows that every left-invariant mean \(\mu\) is the limit
	of a net \(\net*{f_\alpha}\) in \(P_1(G)\).
Such a net is said to converge weakly to invariance, because
the net \(\net*{f_\alpha - l_x f_\alpha}\) converges weakly to $0$ for all \(x\in G\).

A mean \(f\in P_1(G)\) is said to be \((K,\epsilon)\)-invariant if
\(\norm{f - l_x f}_1 < \epsilon\) for each \(x\in K\).
A net in \(P_1(G)\) is said to converge strongly to invariance if it is eventually
\((K,\epsilon)\)-invariant for each finite \(K\subset G\) and \(\epsilon > 0\).

\begin{prop}[{\cref{Days theorem}}]
If $G$ is amenable, \(P_1(G)\) contains a net converging strongly to invariance.
\end{prop}

Let \(\Cp\) be the compact subsets of \(G\) with positive measure.
For \(A\in\Cp\) define \(\mu\sub{A} = \Onebb\sub{A} / \Haar{A} \in P_1(G)\).
Notice \(\norm{\mu\sub{A} - l_x \mu\sub{A}}_1 = \Haar{A \Sdiff xA} / \Haar{A}\).
If \(\mu\sub{A}\) is \((K,\epsilon)\)-invariant, we may also say
	\(A\) is \((K,\epsilon)\)-invariant.
Let $\Cp(K,\epsilon)$ be the set of all $(K,\epsilon)$-invariant sets in $\Cp$.
Let \(\Means_{\Cp}(G) = \set*{\mu\sub{A} : A\in \Cp}\).

\begin{prop}[{\cref{exists pointwise folner net}}]
If $G$ is amenable, \(\Means_{\Cp}(G)\) contains a net \(\net{\mu\sub{A_\alpha}}\) converging strongly to invariance.
In this case, both \(\net{\mu\sub{A_\alpha}}\) and \(\net*{A_\alpha}\) are called F\o{}lner nets.
\end{prop}

\begin{que}
Is it possible to construct a net converging weakly but not strongly to invariance?

The main result of \cite{WeakNotStrong} is that every \(\mu\in\LIM(G)\)
is the limit of some net in \(P_1(G)\) \textit{not} converging strongly to invariance.
Theorems~\cref{Theorem weak but not strong discrete}
and \cref{Theorem weak but not strong nondiscrete}
give a shorter proof of a slightly stronger result:
Every non-atomic \(\mu\in\Means(G)\) is the limit of some net in \(\Means_{\Cp}(G)\)
not converging strongly to invariance.
\end{que}

\begin{que}\label{Question LIM0 = LIM}
Let \(\LIM_0(G)\) be the limit points of F\o{}lner nets.
Does \(\LIM_0(G) = \LIM(G)\)?

Hindman and Strauss asked this question for amenable semigroups in \cite{Hindman2009},
and answered it affirmatively for the semigroup \((\Nbb, +)\).
The following theorems extend their affirmative result to various classes of groups:
\cref{Theorem Unimodular Case} when \(G\) is discrete,
\cref{Theorem LIM when G is large} when \(G\) is larger than \(\sigma\)-compact, and
\cref{Theorem LIM0} when \(G\) is nondiscrete but amenable-as-discrete.
The question remains open when \(G\) is \(\sigma\)-compact but not amenable-as-discrete.
\end{que}

For \(f\in P_1(G)\), regard \(f*\phi\) as the average
	\(\int_G f(x)\: l_x \phi \dd{x}\) of left-translates of \(\phi\).
The set of topological left-invariant means on \(\Linfty\) is
\begin{center}\(\TLIM(G) = \set*{\mu\in\Means(G) : \bigl(\forall f\in P_1(G)\bigr)\ 
\bigl(\forall \phi\in\Linfty\bigr)\ \mu(\phi) = \mu(f*\phi)}\).\end{center}

\begin{prop}
When $G$ is discrete, \(\LIM(G) = \TLIM(G)\).
\end{prop}
\begin{proof}
When $G$ is discrete, every $f\in P_1(G)$ is a sum of point-masses.
Convolution by a point-mass is equivalent to left-translation.
\end{proof}

A net in \(P_1(G)\) is said to converge strongly to topological invariance if it is eventually
	\((K,\epsilon)\)-invariant for each compact \(K\subset G\) and \(\epsilon > 0\).

\begin{prop}[{\cref{limit of TI-net}}]
	If \(\net*{f_\alpha}\) is a net in \(P_1(G)\) converging strongly to topological invariance
	and \(f_\alpha \to \mu\), then \(\mu\in\TLIM(G)\).
\end{prop}

\begin{prop}[{\cref{exists topological folner net}}]
If $G$ is amenable, \(\Means_{\Cp}(G)\) contains a net \(\net{\mu\sub{A_\alpha}}\)
	converging strongly to topological invariance.
In this case, both \(\net{\mu\sub{A_\alpha}}\) and \(\net{A_\alpha}\)
	are called topological F\o{}lner nets.
\end{prop}

For example, \(\seq*{[-n, n]}\) is a topological F\o{}lner sequence for \(\Rbb\).
By \cite[Theorem 3]{Emerson1968a}, it is impossible for a F\o{}lner sequence to be non-topological.

\begin{que}\label{Question non topological Folner net}
If $G$ is amenable and nondiscrete, does it admit a non-topological F\o{}lner net?

When $G$ is nondiscrete but amenable-as-discrete, there exists \(\mu \in \LIM(G) \setminus \TLIM(G)\),
see \cite{Rosenblatt76MathAnn}.
In this case, \cref{Theorem LIM0} yields
a F\o{}lner net \(\net{\mu\sub{X_\alpha S_\alpha}}\) converging to \(\mu\),
which is clearly non-topological.
Otherwise the question remains open.
\end{que}

\begin{que}
Let \(\TLIM_0(G)\) be the limit points of topological F\o{}lner nets.
Does \(\TLIM_0(G) = \TLIM(G)\)?

A partial answer was given by Chou, who showed in \cite{Chou70} that
	\(\TLIM(G)\) is the closed convex hull of \(\TLIM_0(G)\).
The main result of the present paper is to answer this question completely.
If $G$ is $\sigma$-compact non-unimodular,
	the answer is negative by \cref{Theorem Non-Unimodular}.
In all other cases the answer is affirmative:
by Proposition~\cref{Prop When G is compact} when \(G\) is compact,
by \cref{Theorem Unimodular Case} when \(G\) is unimodular,
and by \cref{Theorem when G is large} when \(G\) is larger than \(\sigma\)-compact.
\end{que}

\section*{Converging to invariance weakly but not strongly}
A neighborhood basis about $\mu\in\Means(G)$ is given by sets of the form
\begin{center}
\(\Ncal(\mu, \Fcal, \epsilon) =
\set*{ \nu\in\Means(G) :
	\paren*{\forall f\in\Fcal}\
	\abs*{\mu(f) - \nu(f)} < \epsilon
	}\)
\end{center}
where $\Fcal$ ranges over finite subsets of $\Linfty$ and $\epsilon$ ranges over $(0, 1)$.

\begin{lem}\label{Lemma Partition Basis}
Regard each $\mu\in\Means(G)$ as a finitely additive measure via \(\mu(E) = \mu(\Onebb_E)\).
Then a neighborhood basis about $\mu\in\Means(G)$ is given by sets of the form
\begin{center}
\(\Ncal(\mu, \Pcal, \epsilon) = \set*{
	\nu\in\Means(G) :
	\paren*{\forall E\in\Pcal}\
	|\mu(E) - \nu(E)| < \epsilon}\),
\end{center}
where $\epsilon$ ranges over $(0,1)$ and \(\Pcal\) ranges over finite measurable partitions of $G$.
\end{lem}
\begin{proof}
Pick $\Ncal(\mu, \Fcal, \epsilon)$.
For simplicity, suppose $\Fcal$ consists of simple functions.
Let $M = \max\{\|f\|_\infty : f\in \Fcal\}$.
Let $\Pcal$ be the atoms of the measure algebra generated by $\Fcal$.
Then $\Ncal\big( \mu, \Pcal, \frac{\epsilon}{\card{\Pcal}\cdot M} \big)
	\subset \Ncal(\mu, \Fcal, \epsilon)$.
\end{proof}

\begin{lem}\label{Lemma if E is infinite}
Let $E\subset G$ be any infinite subset, $n\in\Nbb$, and $x\in G\setminus \{e\}$.
Then there exists $S\subset E$ with $\card{S} = n$ and $S \cap xS = \varnothing$.
\end{lem}
\begin{proof}
If $n=0$, take \(S = \varnothing\).
Inductively, suppose there exists $R \subset E$ with $\# R = n-1$ and $R \cap xR = \varnothing$.
Pick any $y\in E \setminus \set{x, x^{-1}}R$, and let $S = R \cup \{y\}$.
\end{proof}

\begin{thm}\label{Theorem weak but not strong discrete}
Suppose $G$ is discrete, $\mu\in\Means(G)$ vanishes on finite sets, and $x\in G\setminus \{e\}$.
Then there exists a net \(\net*{S_{\Pcal}}\) in \(\Cp\) so \(\mu\sub{S_{\Pcal}} \to \mu\),
	but \(S_{\Pcal} \cap x S_{\Pcal} = \varnothing\) for all \(\Pcal\).
\end{thm}
\begin{proof}
Let $\Pscr$ be the directed set of all finite partitions of $G$, ordered by refinement.
Pick \(\Pcal = \set*{E_1, \hdots, E_p} \in \Pscr\).
For $1\leq i\leq p$, we will choose $S_i\subset E_i$ such that
\(\abs*{\frac{\card{S_i} }{\card{S_1} + \cdots + \card{S_p} } - \mu(E_i)} < \frac1p\),
	and take \(S_{\Pcal} = S_1 \cup \hdots \cup S_p\).
Then \(\abs*{\mu\sub{S_{\Pcal}}(E_i) - \mu(E_i)} < \frac1p\),
	and \(\mu\sub{S_{\Pcal}}\to\mu\) by \cref{Lemma Partition Basis}.

We begin by establishing the values \(n_i = \card{S_i}\).
If $\mu(E_i) = 0$, let \(n_i = 0\).
Otherwise $\mu(E_i) > 0$, hence $E_i$ is infinite.
In this case, let $n_i\geq 0$ be an integer such that
	\(\abs*{\frac{n_i}{2p^2} - \mu(E_i)} < \frac{1}{2p^2}\).
Let \(N = \sum_{i=1}^p n_i\).
Now \(\abs*{\frac{n_i}{2p^2} - \frac{n_i}{N}}
	= \frac{n_i}{N} \cdot \abs*{\frac{N}{2p^2} - 1}
	\leq \abs*{\frac{N}{2p^2} - 1}
	= \abs*{\sum_i \frac{n_i}{2p^2} - \mu(E_i)}
	< \frac{1}{2p}\),
so \(\abs*{\mu(E_i) - \frac{n_i}{N}} < \frac1p\).

Apply \cref{Lemma if E is infinite} to choose \(S_1 \subset E_1\)
	with \(S_1 \cap x S_1 = \varnothing\).
For \(k <p\), inductively choose
	\(S_{k+1} \subset E_{k+1} \setminus \set{x, x^{-1}} (S_1 \cup \hdots \cup S_k)\)
	with \(\card{S_{k+1}} = n_k\) and \(S_{k+1} \cap x S_{k+1} = \varnothing\).
Let $S_{\Pcal} = S_1 \cup \hdots \cup S_p$.
Now $\mu\sub{S_{\Pcal}}(E_i) = \frac{n_i}{N}$ and $S_{\Pcal} \cap xS_{\Pcal} = \varnothing$, as desired.
\end{proof}

\begin{point}
The hypothesis ``\(\mu\) vanishes on finite sets'' is necessary in \cref{Theorem weak but not strong discrete}.
For example, pick \(x, y\in G\) and define \(\mu\in\Means(G)\) by
	\(\mu(\{x\}) = \frac23\) and \(\mu(\{y\}) = \frac13\).
For each \(F\in\Cp\), \(\mu\sub{F}(\{x\}) \in \set*{1, \frac12, \frac13, \hdots}\),
hence \(\abs*{\mu(\{x\}) - \mu\sub{F}(\{x\})} \geq \frac16\).
This foreshadows \cref{Theorem Non-Unimodular}.
\end{point}

\begin{lem}\label{Lemma Very Small S}
Suppose $G$ is not discrete.
Pick $E\subset G$ with positive measure, and $X = \{x_1, \hdots, x_n\} \subset G$.
For any \(c > 0\), there exists $S \subset E$ such that $0 < \Haar{S} \leq c$
	and $\{x_1 S, \hdots, x_n S\}$ are mutually disjoint.
\end{lem}
\begin{proof}
Let $K\subset E$ be any compact set with positive measure.
Let $U$ be a small neighborhood of $e$, so $U U^{-1} \cap X^{-1} X = \{e\}$
	and $\max_{k\in K} \Haar{Uk} \leq c$.
Pick \(k_1, \hdots, k_n\in K\) such that \(K\subset U k_1 \cup \hdots \cup U k_n\).
Now $0 < \Haar{K} \leq \sum_{i=1}^n \Haar{U k_i \cap K}$, hence $0 < \Haar{U k_i \cap K}$ for some $i$.
Take $S = U k_i \cap K$.
\end{proof}

\begin{thm}\label{Theorem weak but not strong nondiscrete}
Suppose $G$ is not discrete.
Given any \(\mu\in\Means(G)\) and \(x\in G\setminus\{e\}\),
	there exists a net \(\net*{S_{\Pcal}}\) in \(\Cp\)
	so that \(\mu\sub{S_{\Pcal}}\to \mu\), but
	\(S_{\Pcal} \cap x S_{\Pcal} = \varnothing\) for all \(\Pcal\).
\end{thm}
\begin{proof}
The following construction yields sets \(\net*{S_{\Pcal}}\) that may not be compact.
This suffices to prove the theorem, since each \(S_{\Pcal}\)
can be approximated from within by a compact set.

Let \(\Pscr\) be the directed set of all finite measurable partitions of \(G\),
	ordered by refinement.
Pick \(\Pcal=\{E_1, \hdots, E_p\}\in \Pscr\).
For \(1\leq i\leq p\), we will choose \(S_i'\subset E_i\) such that
$\Haar{S_i'} / \big(\Haar{S_1} + \hdots + \Haar{S_k'}\big) = \mu(E_i)$,
	then take $S_{\Pcal} = S_1' \cup \hdots \cup S_p'$.
Thus $\mu\sub{S_{\Pcal}}(E_i) = \mu(E_i)$,
	and $\mu\sub{S_{\Pcal}}\to\mu$ by \cref{Lemma Partition Basis}.

If $\mu(E_i) = 0$ we can take $S_i = \varnothing$, so assume
$0 < m=\min\{1, \Haar{E_1},\hdots,\Haar{E_p}\}$ and let $c = m / 2p$.
By \cref{Lemma Very Small S},
choose \(S_1 \subset E_1\) with \(0 <|S_1| \leq c\) and \(S_1 \cap xS_1 = \varnothing\).
For \(k < p\), inductively choose
\(S_{k+1} \subset E_{k+1} \setminus \set{x, x^{-1}}(S_1 \cup \hdots \cup S_k)\)
with \(0<\Haar{S_{k+1}} \leq c\) and \(S_{k+1} \cap xS_{k+1} = \varnothing\).
This is possible, since
\[\Haar*{E_{k+1} \setminus \set{x, x^{-1}}(S_1 \cup \hdots \cup S_k)} \geq m - 2kc > 0.\]
Finally, let \(m' = \min\set*{\Haar{S_1}, \Haar{S_2}, \hdots, \Haar{S_p} }\).
For each $i$, choose \(S_i' \subset S_i\) with \(\Haar{S_i'} = m'\cdot \mu(E_i)\).
Now $\mu\sub{S_{\Pcal}}(E_i) = \mu(E_i)$ and \(S_{\Pcal} \cap xS_{\Pcal} = \varnothing\), as desired.
\end{proof}

\section*{\texorpdfstring{$\kappa$-Compactness}{kappa-Compactness}}

\begin{defn}
For the rest of the paper, triples of the form \((\Pcal, K, \epsilon)\)
	are always understood to range over
	finite measurable partitions \(\Pcal\) of \(G\),
	compact sets \(K\subset G\), and \(\epsilon \in (0,1)\).
Recall that \(\Cp(K,\epsilon)\) is the set of all compact
	\((K,\epsilon)\)-invariant sets with positive measure.
In these terms, we can give the formal definition:
\begin{center}
\(\TLIM_0(G) = \big\{
	\mu\in\TLIM(G) : \paren*{\forall \paren*{\Pcal, K, \epsilon}}\ 
	\paren*{\exists A\in\Cp(K,\epsilon)}\ 
	\mu\sub{A} \in \Ncal(\mu, \Pcal, \epsilon)\big\}\).
\end{center}
\end{defn}

For $S\subset G$, let $\kappa(S)$ denote the smallest cardinal such that there exists $\Kcal$,
a collection of compact subsets of $G$ with $\card{\Kcal} = \kappa(S)$ and $S\subset \bigcup \Kcal$.
Notice that either \(\kappa(G) = 1\), \(\kappa(G) = \Nbb\), or \(\kappa(G) > \Nbb\).

\begin{prop}
When $\kappa(G) = 1$, $\TLIM_0(G) = \TLIM(G)$.
\end{prop}
\begin{proof}
By \cref{Prop When G is compact}, both sets are just $\{\mu_G\}$.
\end{proof}

\begin{lem}\label{Lemma if Kappa S is small then measure 0}
If $\mu\in\LIM(G)$ and $\kappa(S) < \kappa(G)$, then $\mu(S) = 0$.
\end{lem}
\begin{proof}
Since \(\Nbb \cdot \kappa(S) \leq \kappa(G)\), we can find disjoint translates
\(\set*{x_1 S, x_2 S, \hdots}\).
If $\mu(S) > 0$, then $\mu(x_1 S \cup \hdots \cup x_n S) = n \cdot \mu(S)$
is eventually greater than 1, contradicting \(\norm{\mu} = 1\).
\end{proof}

\begin{point}
Throughout the paper, we make free use of the following formulas:
\\For \(A,B \in\Cp\),
\(\norm{\mu\sub{A} - \mu\sub{B}}_1
	= \frac{\Haar{A\setminus B}}{\Haar{A}} + \frac{\Haar{B \setminus A}}{\Haar{B}}
	+ \Haar{A\cap B} \cdot \abs*{\frac{1}{\Haar{A}} - \frac{1}{\Haar{B}} }\).
\\If \(A\subset B\),
this becomes
\(\norm{\mu\sub{A} - \mu\sub{B}}_1
	= 0
	+ \frac{\Haar{B \setminus A}}{\Haar{B}}
	+ \Haar{A} \cdot \paren*{\frac{1}{|A|} - \frac{1}{|B|}}
	= 2 \frac{\Haar{B \setminus A}}{\Haar{B}}.\)
\end{point}

\begin{lem}\label{Lemma many disjoint sets}
Pick \((\Pcal, K, \epsilon)\) and $\mu\in\TLIM_0(G)$.
Let \(\kappa = \kappa(G)\).
There exists a family of mutually disjoint sets
\(\set*{A_\alpha : \alpha < \kappa} \subset \Cp(K,\epsilon)\)
such that \(\set*{\mu\sub{A_\alpha} : \alpha < \kappa} \subset \Ncal(\mu, \Pcal, \epsilon)\).
\end{lem}
\begin{proof}
Choose any precompact open set \(U\).
For \(\beta < \kappa(G)\),
suppose \(\set*{A_\alpha : \alpha < \beta} \subset \Cp(K,\epsilon)\)
have been chosen so that \(\set*{A_\alpha U : \alpha < \beta}\) are mutually disjoint
and \(\big\{\mu\sub{A_\alpha} : \alpha < \beta\big\} \subset \Ncal(\mu, \Pcal, \epsilon)\).
Once \(A_\beta\) is constructed, the result follows by transfinite induction.

Define \(B = \bigcup_{\alpha < \beta} A_\alpha UU^{-1}\), which is open and thus measurable.
Notice \(\kappa(B) = \beta < \kappa(G)\).
If \(\Pcal = \set*{E_1, \hdots, E_p}\),
let $\Pcal_B = (E_1 \setminus B, \hdots, E_p \setminus B, B)$.
Let \(\delta = \epsilon/4 \),
and pick \(A\in\Cp(K,\delta)\) with \(\mu\sub{A} \in \Ncal(\mu, \Pcal_B, \delta)\).
Define \(A_\beta = A \setminus B\),
which ensures \(\set*{A_\alpha U : \alpha\leq \beta}\) are mutually disjoint.
\cref{Lemma if Kappa S is small then measure 0} tells us $\mu(B) = 0$,
hence \(\mu\sub{A}(B) < \delta\).
Thus \(\big\|\mu\sub{A} - \mu\sub{A_\beta}\big\|_1
	= 2 \frac{\Haar{A \setminus A_\beta} }{\Haar{A}}
	= 2\frac{\Haar{A \cap B}}{\Haar{A}}
	= 2 \mu\sub{A}(B)
	< 2 \delta\).
By the triangle inequality, \(\mu\sub{A_\beta} \in \Ncal(\mu, \Pcal, 3\delta)\).
If \(x\in K\),
\begin{align*}
\norm{xA_\beta \Sdiff A_\beta}_1
&	\leq \Haar{xA_\beta \Sdiff xA} + \Haar{xA \Sdiff A} + \Haar{A\Sdiff A_\beta}
\\&	< 3\delta \Haar{A} < 4\delta \Haar{A_\beta},\ \ \text{which shows } A_\beta\in\Cp(K,4\delta).
\qedhere
\end{align*}
\end{proof}

\section*{The method of Hindman and Strauss}

\begin{lem}\label{Lemma TLIM0 is total in TLIM}
The closed convex hull of $\TLIM_0(G)$ is all of $\TLIM(G)$.
\end{lem}
\begin{proof}
Chou originally proved this for $\sigma$-compact groups, see \cite[Theorem 3.2]{Chou70}.
In \cite{milnes}, Milnes points out that the result is valid even when $G$ is not $\sigma$-compact,
although his construction of a F\o{}lner-net \(\net*{U_\alpha}\) has a small problem:
For each index \(\alpha\), he asks us to choose a compact set \(U_\alpha\),
such that \(\bigcup_{\beta \prec \alpha}U_\beta \subset U_\alpha\).
However, \(\alpha\) may have infinitely many predecessors,
in which case \(\bigcup_{\beta \prec \alpha} U_\beta\) has no reason to be precompact!
For a proof in full generality, see \cref{Lemma_cl_conv_Xp}.
\end{proof}

The following deceptively simple lemma is due to Hindman and Strauss,
see \cite[Proof of Theorem 4.5]{Hindman2009}.
\begin{lem}\label{Main Idea}
Suppose
$\tfrac12(\mu + \nu)$ is in $\TLIM_0(G)$ whenever $\mu$ and $\nu$ are.
Then $\TLIM(G)_0 = \TLIM(G)$.
\end{lem}
\begin{proof}
Since the dyadic rationals are dense in \([0,1]\) and \(\TLIM_0(G)\) is closed,
the hypothesis implies \(\TLIM_0(G)\) is convex.
The result follows by \cref{Lemma TLIM0 is total in TLIM}.
\end{proof}

\cref{Main Idea} is useful because
\(\frac12\paren*{\mu\sub{A} + \mu\sub{B}} = \mu\sub{A\cup B}\)
when \(A,B \in \Cp\) are disjoint and equal in measure.
\cref{Lemma Mean Approximation} tells us what happens when \(A\) and \(B\)
are \textit{approximately} disjoint and equal in measure.

\begin{lem}\label{Lemma Mean Approximation}
Pick \(A, B \in\Cp\) and \(\delta \in \bigl(0, \frac12\bigr)\).
Suppose \(\mu\sub{B}(A) < \delta\) and
\(\Haar{A}/\Haar{B} \in \bigl( (1-\delta)^2, (1+\delta)^2\bigr)\).
Then \(\norm{\mu\sub{A\cup B} - \frac12(\mu\sub{A} + \mu\sub{B})}_1 < 3\delta\).
\end{lem}
\begin{proof}
Let $B' = B\setminus A$ and $r = \frac{\Haar{A}}{\Haar{B'}} = \frac{\Haar{A}}{\Haar{B}} \cdot \frac{\Haar{B}}{\Haar{B'}}$.
Since
\[\textstyle 1\geq \frac{\Haar{B'}}{\Haar{B}}
= \frac{\Haar{B} - \Haar{B\cap A} }{\Haar{B}} = 1 - \mu\sub{B}(A) > 1 - \delta,\]
we see \( r\in \bigl((1-\delta)^2, \frac{(1 + \delta)^2}{1 - \delta}\bigr)\).
In these terms,
\[\textstyle \mu\sub{A\cup B} = \mu\sub{A \cup B'}
	= \frac{\Haar{A}}{\Haar{A}+\Haar{B'}} \mu\sub{A} + \frac{\Haar{B'}}{\Haar{A}+\Haar{B'}}\mu\sub{B'}
	= \big(\frac{1}{1+r^{-1}}\big)\mu\sub{A} + \big(\frac{1}{1 + r}\big) \mu\sub{B'}.\]
Now we can compute
\[\textstyle \big\| \mu\sub{A\cup B} - \frac12 \big(\mu\sub{A} + \mu\sub{B'}\big) \big\|_1
	\leq \abs*{ \frac12 - \frac{1}{1 + r^{-1}} } \cdot \norm{\mu\sub{A}}_1
		+ \abs*{\frac12 - \frac{1}{1 + r} } \cdot \norm{\mu\sub{B'}}
	= \abs*{\frac{1-r}{1+r}} < 2\delta.\]
On the other hand,
\(\big\|\mu\sub{B} - \mu\sub{B'}\big\|_1
	= 2\frac{|B \setminus B'|}{|B|}
	< 2\delta\),
so \[\textstyle \big\|\frac12 \paren*{\mu\sub{A} + \mu\sub{B'}}
	- \frac12 \paren*{\mu\sub{A} + \mu\sub{B}}\big\|_1 < \delta.\]
The result follows by the triangle inequality.
\end{proof}

\begin{lem}\label{Lemma Invariance of Union}
Suppose $A,B \in \Cp(K, \delta)$. Then $A\cup B\in \Cp(K, 2\delta)$.
If $A\cap B = \varnothing$, then $A\cup B\in \Cp(K, \delta)$.
\end{lem}
\begin{proof}
Pick \(x\in K\).
Notice $x(A\cup B) \Sdiff (A \cup B) \subset (xA \Sdiff A) \cup (xB \Sdiff B)$.
\\It follows that \(\frac{\Haar{x(A\cup B) \Sdiff (A\cup B)} }{\Haar{A\cup B}}
	\leq \frac{\Haar{xA\Sdiff A} }{\Haar{A\cup B}} + \frac{\Haar{xB\Sdiff B} }{\Haar{A\cup B}}
	< \frac{\delta\Haar{A}}{\Haar{A\cup B}} + \frac{\delta\Haar{B}}{\Haar{A\cup B}}
	< 2\delta\).
\\If \(A\cap B = \varnothing\),
then \(\Haar{A\cup B} = \Haar{A} + \Haar{B}\),
hence $\frac{\delta\Haar{A}}{\Haar{A\cup B}} + \frac{\delta\Haar{B}}{\Haar{A\cup B}} = \delta$.
\end{proof}

\begin{thm}\label{Theorem Hindman and Strauss Technique}
Let \(G\) be noncompact.
Then the following statement implies \(\TLIM_0(G) = \TLIM(G)\):
\\[-5mm]\begin{align*}
&	\paren*{\forall (p, K, \epsilon)}\ 
	\paren*{\exists M > 0}\ 
	\paren*{\forall \mu\in\TLIM_0(G)}\ 
	\paren*{\forall \Pcal\ \text{with}\ \#\Pcal=p}
\\&	\exists A\in\Cp(K,\epsilon)
	\text{ with } \Haar{A}/M \in (1-\epsilon, 1+\epsilon)
	\text{ and } \mu\sub{A}\in\Ncal(\mu, \Pcal, \epsilon)
\end{align*}
The crux of the statement is that \(M\) is allowed to depend on
\((\#\Pcal, K, \epsilon)\), but not on \(\Pcal\) itself.
\end{thm}
\begin{proof}
Pick $(\Pcal, K, \epsilon)$ and $\mu,\nu \in \TLIM_0(G)$.
Say \(\Pcal = \paren*{E_1, \hdots, E_p}\) and \(\delta = \epsilon/4\).
By hypothesis, obtain \(M>0\) for \( (p+1, K, \delta)\).
Let \(\Pcal_{\varnothing} = \paren*{E_1, \hdots, E_p, \varnothing}\),
	so \(\#\Pcal_\varnothing = p+1\).
Pick \(A\in\Cp(K,\delta)\) with \(\mu\sub{A}\in\Ncal(\mu, \Pcal_\varnothing, \delta)\)
	and \(\Haar{A}/M \in (1-\delta, 1+\delta)\).
Let \(\Pcal_A = (E_1 \setminus A, \hdots, E_p \setminus A, A)\).
Pick \(B\in\Cp(K,\delta)\)
	with \(\mu\sub{B} \in \Ncal\paren{\nu, \Pcal_A, \delta}\)
	and \(\Haar{B}/M\in (1-\delta, 1+\delta)\).
Now \(\Haar{A} / \Haar{B} \in \bigl((1-\delta)^2, (1+\delta)^2\bigr)\).
Since \(A\) is compact but \(G\) is not,
\cref{Lemma if Kappa S is small then measure 0} tells us
\(\nu(A) = 0\), hence \(\mu\sub{B}(A) < \delta\).
Applying \cref{Lemma Mean Approximation} to the previous two statements, we conclude
\(\norm{\frac12\paren*{\mu\sub{A} + \mu\sub{B}} - \mu\sub{A\cup B}}_1 < 3\delta\).
By the triangle inequality,
\(\mu\sub{A\cup B} \in \Ncal\big(\frac12(\mu+\nu), \Pcal, 4\delta\big)\).
By \cref{Lemma Invariance of Union}, \(A\cup B \in \Cp(K,2\delta)\).
We conclude \(\frac12(\mu + \nu)\in\TLIM_0(G)\).
By \cref{Main Idea}, we are done.
\end{proof}

\section*{Ornstein-Weiss quasi-tiling}

\begin{point}\label{Ci definitions}
We require the following notions of \((K,\epsilon)\)-invariance.
\\\(\Cp_0(K,\epsilon) = \set*{A\subset G :
	A\text{ is compact and }\ (\forall x\in K)\ \Haar{xA\Sdiff A} / \Haar{A} < \epsilon}\).
\\This is just \(\Cp(K,\epsilon)\) above.
\\
\(\Cp_1(K,\epsilon) = \{A\subset G : A \text{ is compact and }
	\Haar{KA \Sdiff A} / \Haar{A} < \epsilon\}\).
\\
	\(\Cp_2(K,\epsilon) = \set*{A\subset G : A \text{ is compact and }
	\Haar{\partial_K(A)} / \Haar{A} < \epsilon}\),
\\In this definition, \(\partial_K(A) = K A \setminus \bigcap_{x\in K} x A\).
This differs slightly from \cite{Ornstein1987},
where ``the \(K\)-boundary of \(A\)''
is defined as \(K^{-1} A \setminus \bigcap_{x\in K^{-1}} x A\).\medskip

\noindent For \(j\in\set{0,1,2}\), define
\begin{center}
\(\TLIM_j(G) = \big\{
	\mu\in\TLIM(G) : \paren*{\forall \paren*{\Pcal, K, \epsilon}}\ 
	\big(\exists A\in\Cp_j(K,\epsilon) \big)\ 
	\mu\sub{A} \in \Ncal(\mu, \Pcal, \epsilon)\big\}\).
\end{center}
We shall say ``\(\Cp_i\) is asymptotically contained in \(\Cp_j\)'' to mean the following:
\begin{center}
\(\paren*{\forall(K,\epsilon)}\
	\paren*{\exists(K',\epsilon')}\
	\paren*{\forall A\in\Cp_i(K', \epsilon')}\
	\exists B\in\Cp_j(K,\epsilon) \text{ with } \norm{\mu\sub{A}-\mu\sub{B}}_1 < \epsilon\).
\end{center}
\end{point}

\begin{lem}
To prove \(\TLIM_i(G) \subset \TLIM_j(G)\), it suffices to show
\(\Cp_i\) is asymptotically contained in \(\Cp_j\).
\end{lem}
\begin{proof}
Pick \((\Pcal, K, \epsilon)\) and \(\mu\in\TLIM_i(G)\).
By hypothesis, obtain \((K', \epsilon')\).
Obtain \(A\in\Cp_i(K', \epsilon')\) with \(\mu\sub{A}\in\Ncal(\mu, \Pcal, \epsilon)\).
Obtain \(B\in\Cp_j(K, \epsilon)\) with \(\norm{\mu\sub{A} - \mu\sub{B}}_1 < \epsilon\).
By the triangle inequality, \(\mu\sub{B}\in\Ncal(\mu, \Pcal, 2\epsilon)\).
Since \(\epsilon\) was arbitrary, we conclude \(\mu\in\TLIM_j(G)\).
\end{proof}

\begin{lem}\label{Lemma no TLIMi left behind}
\(\TLIM_0(G) = \TLIM_1(G) = \TLIM_2(G).\)
\end{lem}
\begin{proof}
Trivially, \(\Cp_2(K,\epsilon)\subset \Cp_1(K,\epsilon)\).

To prove \(\Cp_1(K, \epsilon)\subset \Cp_0(K, 2\epsilon)\), suppose \(x\in K\) and \(A\in\Cp_1(K,\epsilon)\).
Then $\Haar{xA \Sdiff A}  = 2\Haar{xA \setminus A} \leq 2\Haar{KA \setminus A} < 2\epsilon\Haar{A}.$

\cref{Emerson Greenleaf Lemma} shows that \(\Cp_0\) is asymptotically contained in \(\Cp_1\).

To prove \(\Cp_1\) is asymptotically contained in \(\Cp_2\),
pick \((K,\epsilon)\) and let \(J = K \cup K^{-1} \cup \set{e}\).
Suppose \(A\in\Cp_1\paren{J^2, \epsilon}\), and let \(B = JA\).
Then \(A \subset \partial_{J}(B)\), and \(JB \setminus \partial_J(B) \subset J^2 A \setminus A\),
which shows \(B\in\Cp_2(J, \epsilon) \subset \Cp_2(K, \epsilon)\).
Also, since \(B \subset J^2 A\), we see \(|B| < (1+\epsilon)|A|\), hence
\(\norm{\mu\sub{A} - \mu\sub{B}}_1
	= 2\frac{|A\setminus B|}{|A|}
	< 2\epsilon\).
\end{proof}

\begin{lem}\label{Lemma many disjoint sets 2}
Suppose \(G\) is noncompact, and pick \((\Pcal, K, \epsilon)\) and $\mu\in\TLIM_0(G)$.
There exists a family of mutually disjoint sets \(\set*{B_n: n\in\Nbb} \subset \Cp_1(K,\epsilon)\)
with \(\set*{\mu\sub{B_n} : n\in\Nbb} \subset \Ncal(\mu, \Pcal, \epsilon)\).
\end{lem}
\begin{proof}
By \cref{Emerson Greenleaf Lemma}, obtain $(K', \epsilon')$ so
for each $A\in\Cp_0(K',\epsilon')$,
there exists \(B\in\Cp_1(K, \epsilon)\)
with \(\norm{\mu\sub{A} - \mu\sub{B}}_1 < \epsilon/2\).
By \cref{Lemma many disjoint sets},
choose a family of mutually disjoint sets \(\set*{A_n : n\in\Nbb}\subset\Cp_0(K', \epsilon')\)
with \(\set{\mu\sub{A_n} : n\in\Nbb}\subset\Ncal\paren*{\mu,\Pcal, \epsilon/2}\).
For each \(n\), choose \(B_n \in \Cp_1(K, \epsilon)\)
with \(\norm{\mu\sub{A_n} - \mu\sub{B_n}}_1 < \epsilon/2\).
By the triangle inequality, \(\mu\sub{B_n} \in \Ncal(\mu, \Pcal, \epsilon)\).
\end{proof}

\begin{lem}\label{Lemma A can be large}
Suppose $G$ is noncompact unimodular.
Pick $(\Pcal, K, \epsilon)$, $\mu\in\TLIM_0(G)$, and $M>0$.
Then there exists $B\in\Cp_2(K,\epsilon)$ with $|B| \geq M$
	and $\mu\sub{B}\in\Ncal(\mu,\Pcal,\epsilon)$.
\end{lem}
\begin{proof}
Without loss of generality, assume \(|K| > 0\).
Let \(\delta = \epsilon/3\).
By \cref{Lemma many disjoint sets 2}, choose disjoint sets
$\set{A_n : n\in\Nbb} \subset \Cp_1\paren{K^2,\epsilon}$
with \(\set{\mu\sub{A_n} : n\in\Nbb}\subset\Ncal\paren{\mu,\Pcal,\delta}\).
For each \(n\), $\Haar{A_n} > (1-\epsilon)\Haar{K^2 A_n}> \Haar{K^2}$,
where the second inequality holds because \(G\) is unimodular.
For sufficiently large \(m\), \(\Haar{A_1 \cup \hdots \cup A_m}_1 \geq M\).
Define \(A = A_1 \cup \hdots \cup A_m\).
Since this union is disjoint, it is easy to check \(A\in \Cp_1(K^2, \epsilon)\).
Notice \(\mu\sub{A}= \frac{\Haar{A_1}}{\Haar{A}} \mu\sub{A_1}
	+  \hdots + \frac{\Haar{A_m} }{\Haar{A}}\mu\sub{A_m} \in \Ncal\paren*{\mu,\Pcal,\delta}\),
as the latter set is convex.
Finally, let $B = KA$.
It follows, as in the proof of \cref{Lemma no TLIMi left behind},
that \(B\in\Cp_2(K, \epsilon)\) and \(\norm{\mu\sub{A} - \mu\sub{B}}_1 < 2\delta\),
hence \(\mu\sub{B}\in\Ncal(\mu,\Pcal, 3\delta)\) by the triangle inequality.
\end{proof}

Recall that \(\sqcup\) denotes a disjoint union.
\begin{defn}\label{Definition Quasitiling}
Let \(B\subset G\) and \(T_1 \subset \hdots \subset T_n \subset G\) have finite positive measure.
\(\Tcal = \set*{T_1, \hdots, T_n}\) is said to
\(\epsilon\)-quasi-tile \(B\) if there exists
\(R = S_1 \cup \hdots \cup S_N = R_1 \sqcup \hdots \sqcup R_N \subset B\) satisfying the following:
\\\phantom{ }
(1) For \(i\in\set{1, \hdots, N}\), \(S_i = Tx\) for some $T\in\Tcal$ and $x\in G$.
\\\phantom{ }
(2) For \(i\in\set{1, \hdots, N}\), \(R_i \subset S_i\)
	with \(\norm{\mu\sub{R_i} - \mu\sub{S_i}}_1 < \epsilon\).
\\\phantom{ }
(3) \(\norm{\mu\sub{B} - \mu\sub{R}}_1 < \epsilon\).
\\Ornstein and Weiss \cite{Ornstein1987} give the following weaker conditions in place of (2) and (3):
\\\phantom{ }
\((2')\) For \(i\in\set{1, \hdots, N}\), \(R_i \subset S_i\)
	with \(\Haar{R_i} > (1-\epsilon) \Haar{S_i}\).
\\\phantom{ }
\((3')\) \(|R| > (1-\epsilon) |B|\).
\\Of course \((2')\) implies \(\norm{ \mu\sub{S_i} - \mu\sub{R_i} }_1
	= 2\frac{\Haar{S_i \setminus R_i}}{\Haar{S_i}}
	< 2 \epsilon\),
and \((3')\) implies \(\norm{ \mu\sub{A} - \mu\sub{R} }_1
	= 2 \frac{\Haar{B\setminus R}}{\Haar{B}}
	< 2\epsilon\).
Hence the definitions become equivalent
when we quantify over all \(\epsilon \in (0,1)\), as in \cref{Ornstein Weiss Lemma}.
\end{defn}

\begin{lem}\label{Ornstein Weiss Lemma}
Let \(G\) be unimodular. Pick \((K,\epsilon)\).
Then there exists \(\epsilon' >0\) and \(\Tcal = \set*{T_1, \hdots, T_n} \subset\Cp_2(K,\epsilon)\)
such that \(\Tcal\) \(\epsilon\)-quasi-tiles any
\(B\in\Cp_2(T_{\mathrlap{n}}{\mathstrut}^{-1} T_n, \epsilon')\).
\end{lem}
\begin{proof}
This is \cite[p.\ 30, Theorem 3 and Remark]{Ornstein1987},
although Ornstein and Weiss write \(\set*{F_1, \hdots, F_K}\)
instead of \(\set*{T_1, \hdots, T_n}\),
and invert the definition of \(\partial_K(A)\) as described in \cref{Ci definitions}.
\end{proof}

\begin{lem}\label{Lemma disjointified are still invariant}
If \(S_i \in \Cp_1(K,\epsilon)\)
and \(R_i \subset S_i\) with \(\norm{\mu\sub{R_i} - \mu\sub{S_i}}_1 < \epsilon\),
then \(R_i \in \Cp_1(K,3\epsilon)\).
\end{lem}
\begin{proof}
\(\frac{\Haar{K R_i \setminus R_i}}{\Haar{R_i}}
	\leq \frac{\Haar{S_i}}{\Haar{R_i}}\cdot \frac{\Haar{KS_i \setminus S_i}}{\Haar{S_i}}
		+ \frac{\Haar{S_i \setminus R_i}}{\Haar{R_i}}
	< (1+\epsilon)\cdot\epsilon + \norm{ \mu\sub{R_i}  - \mu\sub{S_i}}_1
	< 3\epsilon\).
\end{proof}

\begin{thm}\label{Theorem Unimodular Case}
If $G$ is noncompact unimodular, $\TLIM_0(G) = \TLIM(G)$.
\end{thm}
\begin{proof}
Pick \((\Pcal, K, \epsilon)\) and \(\mu\in\TLIM_0(G)\),
say \(\Pcal = \set*{E_1, \hdots, E_p}\).
Let \(V = \{x\in\Rbb^p : \|x\|_1 = 1\}\).
For \(m\in\Means(G)\), let \(v(m) = \begin{bmatrix} m(E_1) & \hdots & m(E_p) \end{bmatrix} \in V\).
Thus \(m\in\Ncal(\mu, \Pcal,\epsilon)\) if and only if \(\|v(m) - v(\mu)\|_\infty < \epsilon\).
Let \(D\subset V\) be a finite \(\epsilon\)-dense subset, \ie\ 
for each \(v\in V\) there exists \(d\in D\) with \(\norm{v-d}_\infty < \epsilon\).
By \cref{Ornstein Weiss Lemma}, obtain
\(\Tcal = \set*{T_1, \hdots, T_n} \subset \Cp_2(K,\epsilon)\) and \(\epsilon' > 0\).
Let \(M = \tfrac{|D|}{\epsilon}\cdot |T_n|\).
Notice that \(M\) depends only on \((p, K, \epsilon)\).
By \cref{Lemma A can be large},
there exists \(B\in\Cp_2(T_n^{-1} T_n,\epsilon')\)
with \(\Haar{B} \geq M\) and \(\mu\sub{B}\in\Ncal(\mu,\Pcal,\epsilon)\).
Our goal is to construct \(A\in\Cp_1(K,3\epsilon)\) such that 
\(\Haar{A}/M \in (1-\epsilon, 1]\) and
\(\norm{v\paren*{\mu\sub{A}} - v\paren*{\mu\sub{B}}}_\infty < 5\epsilon\).
This implies \(\mu\sub{A}\in\Ncal(\mu,\Pcal,6\epsilon)\).
Since \(\epsilon\) was arbitrary, \(\TLIM_0(G) = \TLIM(G)\) follows
by \cref{Theorem Hindman and Strauss Technique}.

As in \cref{Definition Quasitiling},
let \(R = \bigsqcup_{i=1}^N R_i\) with \(\norm{\mu\sub{B} - \mu\sub{R}}_1 < \epsilon\).
Clearly \[\|v\paren*{\mu\sub{B}} - v\paren*{\mu\sub{R}}\|_\infty < \epsilon.\]
By \cref{Lemma disjointified are still invariant}, each \(R_i\) is in \(\Cp_1(K,3\epsilon)\).
Since the $R_i$'s are disjoint, any union of them is in \(\Cp_1(K,3\epsilon)\) as well.
Notice that each $R_i$ has measure at most $|T_n|$.
Let \(r_i = \Haar{R_i}/\Haar{R}\), so \(\sum r_i = 1\) and \(\mu\sub{R} = \sum r_i \cdot \mu\sub{R_i}\).

For each \(R_i\), pick \(d_i \in D\) with
	\(\norm{d_i - v\paren*{\mu\sub{R_i}}}_\infty < \epsilon\).
For each \(d\in D\), let \(C_d = \set*{i : d_i = d}\) and \(c_d = \sum \set*{r_i : d_i = d}\).
Thus
\[\big\| v\paren*{\mu\sub{R}} - \sum_{d\in D} c_d \cdot d\big\|_\infty < \epsilon.\]
Let \(S_d \subset C_d\) be a maximal subset satisfying
$\textstyle s_d = \sum \set*{r_i : i \in S_d} \leq \tfrac{M}{\Haar{R}} c_d.$
This guarantees
\[0 \leq \big( c_d - \tfrac{\Haar{R}}{M} s_d\big) < \tfrac{\epsilon}{|D|},\ \ \ 
\text{hence}\ \ \ 
\Big| \sum_{d\in D}\big(c_d - \tfrac{\Haar{R}}{M} s_d \big) \Big| < \epsilon.\]
Let \(A = \bigcup_{d\in D} \bigcup_{i \in S_d} R_i\).
Notice the following:
\[ \Haar{A} = \sum_{d\in D} \sum_{i\in S_d} \Haar{R_i} = \sum_{d\in D} \Haar{R} s_d,
\ \ \ \text{hence}\ \ \ 
\tfrac{\Haar{A}}{M} = \sum_{d\in D} \tfrac{\Haar{R}}{M}s_d \in (1-\epsilon, 1].\]
\[\mu\sub{A} = \sum_{d\in D} \sum_{i\in S_d} \tfrac{\Haar{R}}{\Haar{A}} r_i \cdot \mu\sub{R_i},\ \ \ 
\text{hence}\ \ \ 
v(A) = \sum_{d\in D} \sum_{i\in S_d} \tfrac{\Haar{R}}{\Haar{A}}r_i \cdot v(\mu\sub{R_i}).\]
Finally, we apply the triangle inequality:
\begin{align*}
\bigl\|v\paren*{\mu\sub{B}} - v\paren*{\mu\sub{A}} \bigr\|_\infty
&	\leq \bigl\|v\paren*{\mu\sub{B}} - v\paren*{\mu\sub{R}} \bigr\|_\infty
	+ \bigl\|v\paren*{\mu\sub{R}} - \sum_{d\in D} c_d \cdot d \bigr\|_\infty
\\&	+ \big\|\sum_{d\in D} \big(c_d - \tfrac{\Haar{R}}{M} s_d\big) \cdot d\big\|_\infty
	+ \big\|\sum_{d\in D} \big(\tfrac{\Haar{R}}{M} - \tfrac{\Haar{R}}{\Haar{A}}\big) s_d \cdot d\big\|_\infty
\\&	+ \big\|\sum_{d\in D}\sum_{i\in S_d} \tfrac{\Haar{R}}{\Haar{A}} r_i \cdot
		\paren*{d - v\paren*{\mu\sub{R_i}}}\big\|_\infty
	< 5\epsilon.\
\qedhere
\end{align*}
\end{proof}
\clearpage

\section*{Easy cases}
\begin{thm}\label{Theorem when G is large}
If $\kappa(G)>\Nbb$, $\TLIM_0(G) = \TLIM(G)$.
\end{thm}
\begin{proof}
Pick \(\mu,\nu\in\TLIM_0(G)\) and \((\Pcal, K, \epsilon)\).
By the same technique as \cref{Lemma many disjoint sets}, obtain mutually disjoint sets
\(\set*{A_\alpha, B_\alpha : \alpha < \kappa(G)} \subset \Cp(K,\epsilon)\),
such that \(\mu\sub{A_\alpha}\in \Ncal(\mu,\Pcal, \epsilon)\),
and \(\mu\sub{B_\alpha} \in \Ncal(\nu,\Pcal,\epsilon)\) for each \(\alpha\).

Let $r = (1+\epsilon)^{1/2}$, and let $I_n = [r^n, r^{n+1})$ for $n\in\Zbb$.
By the pigeonhole principle, there exist \(m,n\in\Zbb\) such that
\(\set*{\alpha : |A_\alpha| \in I_m}\) and \(\set*{\alpha : |B_\alpha| \in I_m}\) are infinite.
Pick \(M,N\in\Nbb\) so that \(M/N \in I_{m-n}\).
Pick \(A_1, \hdots, A_N\) from \(\set*{A_\alpha : |A_\alpha|\in I_m}\)
and \(B_1, \hdots, B_M\) from \(\{B_\alpha : |B_\alpha| \in I_n\}\).
Let \(A = A_1 \cup \hdots \cup A_N\) and \(B = B_1 \cup \hdots \cup B_M\).
Thus $|A| / |B| \in (1-\epsilon, 1+\epsilon)$.
Since the $A_i$'s and $B_i$'s are mutually disjoint, we have $A,B, A\cup B\in\Cp(K,\epsilon)$.
Clearly \(\mu\sub{A}\in\Ncal(\mu,\Pcal,\epsilon)\) and
\(\mu\sub{B}\in\Ncal(\nu,\Pcal, \epsilon)\).
By \cref{Lemma Mean Approximation} and the triangle inequality,
\(\mu\sub{A\cup B} \in \Ncal\bigl(\frac12(\mu + \nu), \Pcal, 4\epsilon\bigr)\).
Since \(\epsilon\) was arbitrary, we conclude \(\frac12(\mu + \nu) \in \TLIM_0(G)\).
\end{proof}

\begin{thm}\label{Theorem Non-Unimodular}
If $G$ is $\sigma$-compact non-unimodular, $\TLIM_0(G) \neq \TLIM(G)$.
\end{thm}
\begin{proof}
Recall that the modular function \(\Delta\) is defined by \(|Ex| = |E|\Delta(x)\).

Let \(K_1 \subset K_2 \subset \hdots\) be a sequence of compacta with \(\bigcup_n K_n = G\).
For each $n$, pick \(F_n \in \Cp\big(K_n, \frac1n\big)\),
then pick $x_n, y_n\in G$ such that $F_n x_n \subset \{x : \Delta(x) \geq 1\}$
and \(F_n y_n \subset \big\{x : \Delta(x) \leq \min\paren*{1, |F_n|^{-1} 2^{-n}}\big\}\).
Let $X = \bigcup_n F_n x_n$ and $Y = \bigcup_n F_n y_n$.
By construction, \(X\cap Y = \varnothing\) and  \(|Y| \leq 1\).
Let $\mu$ be an accumulation point of $\seq*{\mu\sub{F_n x_n}}$
and $\nu$ an accumulation point of $\seq*{\mu\sub{F_n y_n}}$.
Thus $\mu,\nu\in \TLIM_0(G)$.
Let $m = \frac12(\mu + \nu)$.
Notice $\mu(X) = \nu(Y) = 1$ and $\mu(Y) = \nu(X) = 0$, hence $m(X) = m(Y) = \frac12$.

Suppose \(m\in \TLIM_0(G)\).
Let $K$ be any compact set with $|K| \geq \frac72$.
Define \(\epsilon = \frac16\) and \(\Pcal = \{X, Y\}\).
By \cref{Lemma no TLIMi left behind}, there exists $A\in\Cp_1(K, \epsilon)$
with $\mu\sub{A} \in \Ncal(m, \Pcal, \epsilon)$.
Now \(\frac{1}{|A|}
	\geq \frac{|Y|}{|A|}
	\geq \frac{|Y\cap A|}{|A|}
	= \mu\sub{A}(Y) > m(Y) - \epsilon = \frac13\), so \(|A| < 3\).
On the other hand, $\mu\sub{A}(X) > \frac13$,
so $|A \cap X| \neq 0$. Say $a \in A \cap X$.
Now $\frac72 \leq |K| \leq |Ka| \leq |KA| \leq |A| + |KA \setminus A| < (1+\epsilon)|A|$,
	so $3 < |A|$, a contradiction.
Hence $m\in\TLIM(G) \setminus \TLIM_0(G)$.
\end{proof}
\clearpage

\section*{Non-topological invariant means}
In this section, triples of the form \((\Pcal, K, \epsilon)\) range over
	finite measurable partitions \(\Pcal\) of \(G\),
	finite sets \(K\subset G\),
	and \(\epsilon \in (0,1)\).
In these terms, we define
\begin{center}
\(\LIM_0(G) = \set*{
	\mu\in\LIM(G) : \big(\forall \paren*{\Pcal, K, \epsilon}\big)\ 
	\big(\exists A\in\Cp(K,\epsilon)\big)\ 
	\mu\sub{A} \in \Ncal(\mu, \Pcal, \epsilon)}\).
\end{center}

\begin{lem}[{\cite[Theorem 2.12]{Hindman2009}}]\label{Lemma LIM0 is total in LIM}
The closed convex hull of $\LIM_0(G)$ is all of $\LIM(G)$.
\end{lem}

\begin{thm}\label{Theorem LIM when G is large}
If $\kappa(G) > \Nbb$, then $\LIM_0(G) = \LIM(G)$.
\end{thm}
\begin{proof}
The proof is identical to that of \cref{Theorem when G is large},
	but with \cref{Lemma LIM0 is total in LIM}
	in place of \cref{Lemma TLIM0 is total in TLIM}.
\end{proof}

\begin{point}\label{Refinement}
Fix a finite measurable partition \(\Pcal = \set*{E_1, \hdots, E_p}\),
	and a finite set \(X = \{x_1, \hdots, x_n\}\), with \(x_1 = e\).
For \(C = (c_1, \hdots, c_n) \in \{1, \hdots, p\}^n\), define
	\(E(C) = \big\{y \in G : x_1 y \in E_{c_1}, \ldots,x_n y \in E_{c_n}\big\}
		= \bigcap_{k=1}^n x_k^{-1} E_{c_k}\).
Thus \(\Qcal = \big\{E(C) : C\in \{1,\hdots,p\}^n\big\}\) is a refinement of \(\Pcal\).
Notice \(E_i = \bigcup\{E(C) : c_1 = i\}\)
and \(x_k^{-1} E_i = \bigcup\{E(C) : c_k = i\}\).

The idea of refining $\Pcal$ this way is due to \cite{WeakNotStrong}.
\end{point}

\begin{thm}\label{Theorem LIM0}
If $G$ is nondiscrete but amenable-as-discrete, then $\LIM_0(G) = \LIM(G)$.
\end{thm}
\begin{proof}
Pick \((\Pcal, K, \epsilon)\) and \(\mu\in\LIM(G)\).
Let $X = \{x_1, \hdots, x_n\}$ be $(K, \epsilon)$-invariant, with $x_1 = e$.
Let $\Qcal$ refine $\Pcal$ as in \cref{Refinement},
and let \(\set{F_1, \hdots, F_q} = \set{F\in \Qcal : |F| > 0}\).
Let $m=\min\{1, |F_1|, \hdots, |F_q|\}$ and $c = m / (n^2 q)$.
By \cref{Lemma Very Small S}, pick $S_1 \subset F_1$
such that $0 <|S_1| \leq c$ and $S_1 S_1^{-1} \cap X^{-1} X = \{e\}$.
For \(k < q\), inductively choose
\(S_{k+1} \subset F_{k+1} \setminus X^{-1} X (S_1 \cup \hdots \cup S_k)\)
with \(0<|S_{k+1}| \leq c\) and \(S_{k+1} S_{k+1}^{-1} \cap X^{-1} X = \{e\}\).
This is possible, since
\(|F_{k+1} \setminus X^{-1} X (S_1 \cup \hdots \cup S_k)| \geq m - n^2kc > 0\).

Now let \(m' = \min\set*{|S_1|, |S_2|, \hdots, |S_q|}\).
For each $i$, choose $S_i' \subset S_i$ with \(|S_i'| = m'\cdot \mu(E_i)\),
then let $S = S_1' \cup \hdots \cup S_q'$.
By construction, $XS = x_1 S \sqcup \hdots \sqcup x_n S$.
For $y\in K$, $|yXS \Sdiff XS| / |XS| = \#(yX \Sdiff X) / \# (X) < \epsilon$,
	hence \(XS \in \Cp(K, \epsilon)\).
Notice $\mu\sub{XS} = \frac{1}{n} \sum_{k=1}^n \mu\sub{x_k S}$.
For each $i$ and $k$, $\mu\sub{x_k S}(E_i)
	= \mu\sub{S}(x_k^{-1} E_i)
	= \sum \{\mu\sub{S}(E(C)) : c_k = i\}
	= \sum \{\mu(E(C)) : c_k = i\}
	= \mu(x_k^{-1} E_i) = \mu (E_i)$.
Hence $\mu\sub{XS}(E_i) = \mu(E_i)$.
\end{proof}
\chapter{A bijection of invariant means on an amenable group with those on a lattice subgroup}\label{lattice paper}
The following chapter is based on my paper \cite{hopfenspergerLattice}.
\setcounter{equation}{0}

\medskip

\noindent\textbf{\large Abstract}

\medskip\noindent
Suppose $G$ is an amenable locally compact group, with lattice subgroup $\Gamma$.
Then there is a natural affine injection $\iota: \LIM(\Gamma)\to \TLIM(G)$, as demonstrated by Grosvenor.
He was able to prove $\iota$ is a surjection essentially in the case $G=\Rbb^d$, $\Gamma=\Zbb^d$.
In the present chapter, I prove $\iota$ is a surjection if and only if $G/\Gamma$ is compact.

\section*{Introduction}
Let $H$ be a closed subgroup of the locally compact group $G$.
$H$ is called \textit{cofinite} if $G/H$ carries a $G$-invariant probability measure,
and \textit{cocompact} if $G/H$ is compact.
After the recent paper of Bader \textit{et al}.\ \cite{Bader2019},
$G$ is said to have property $(M)$ if every cofinite subgroup is cocompact.
The letter $M$ is for Mostow, who proved that solvable Lie groups have property $(M)$.

A discrete subgroup $\Gamma$ is called a lattice if it is cofinite.
Cocompact lattices are better known as \textit{uniform} lattices.
Non-uniform lattices are not too hard to come by, but the most famous examples are non-amenable.
For example, $\text{SL}_2(\Zbb)$ is a non-uniform lattice in $\text{SL}_2(\Rbb)$
which is not amenable, because
$\begin{psmallmatrix}1&2\\0&1\end{psmallmatrix}$ and
$\begin{psmallmatrix}1&0\\2&1\end{psmallmatrix}$ generate a free subgroup.
In fact, \cite[Theorem 1.7]{Bader2019} states that
every finite direct product of amenable linear groups has property $(M)$.
Hence, every lattice in such a group is uniform.
(A locally compact group $G$ is said to be linear if there exists a continuous monomorphism of $G$ into
$GL_d(k)$ for some $d> 0$ and some locally compact field $k$.)

A paper of Grosvenor \cite{Grosvenor1985} constructs a natural affine injection $\iota$
from the invariant means on a lattice subgroup $\Gamma$ into the topological invariant means on the whole group $G$.
Grosvenor proved that $\iota$ is surjective when $G$ is an abelian Lie group
with finitely many connected components.
In light of the results of \cite{Bader2019}, it seemed natural to suppose
$\iota$ might be surjective when $\Gamma$ is uniform.
This is precisely what we shall prove.

\section*{Definitions}
Let $G$ be a locally compact amenable group.
Associate $\TLIM(G)$ with $\LIM(\LUC(G))$ via \cref{TlimG is LimLucG}.
Since $G$ is amenable, so are all its closed subgroups.\footnote{
	To prove this in general is a bit technical,
	see \cite[Theorem 2.3.2]{greenleaf} or \cite[Proposition 1.12]{paterson-Book}.\linebreak
	On the other hand, suppose $\Gamma < G$ is a discrete subgroup.
	In this case, it is easy to construct a Borel transversal $T$ for the right-coset space $\Gamma\backslash G$.
	Define $\Phi: \Linfty(\Gamma) \to \Linfty(G)$ by $\Phi f(\gamma t) = f(\gamma)$,
	where $t\in T$ and $\gamma\in\Gamma$. If $\mu\in\LIM(G)$, then $\mu\circ \Phi \in \LIM(\Gamma)$.
}

\begin{point}
Let $\Cscr(G)$ be the compact symmetric neighborhoods of $e$ in $G$.
For $F\in\Cscr(G)$, define the mean $\mu\sub{F}$ on $\Linfty(G)$ by $\mu\sub{F}(f) = \frac{1}{|F|}\int_F f(x) \dd{x}$,
where $|F|$ is the Haar measure of $F$.
For $E,F \in \Cscr(G)$
\begin{equation}\label{norm mu F}\textstyle
\norm{\mu\sub{E} - \mu\sub{F}}
	= \norm*{ \frac{\Onebb\sub{E}}{|E|} - \frac{\Onebb\sub{F}}{|F|}}_1
	= \frac{|E\setminus F|}{|E|} + \frac{|F \setminus E|}{|F|} + \abs*{ \frac1{|E|} - \frac1{|F|}}.
\end{equation}
The pair $(K,\epsilon)$ always signifies $K\in\Cscr(G)$ and $\epsilon > 0$.
$F\in\Cscr(G)$ is said to be $(K,\epsilon)$-invariant if $|KF\setminus F| < \epsilon |F|$.
A net $\set*{F_\alpha}$ in $\Cscr(G)$ is called a F\o{}lner net if it is eventually $(K,\epsilon)$-invariant
for any $(K,\epsilon)$.
By \cref{Emerson Greenleaf Lemma}, every amenable group admits a F\o{}lner net satisfying this definition.
\end{point}

\begin{lem}[{\cref{Theorem Unimodular Case}}]\label{Every TLIM is limit of Folner net}
Suppose $G$ is unimodular.
For each $\mu\in\TLIM(G)$, there exists a F\o{}lner net $\set*{F_\alpha}$
with $\lim_\alpha \mu\sub{F_\alpha} = \mu$.
\end{lem}

\begin{point}
The following well-known result is due to Weil \cite[Section 9]{Weil1953}.
Suppose $H < G$ is any closed subgroup.
The coset space $G/H$ admits a $G$-invariant Radon measure $\lambda$ if and only if $\Delta_H = \Delta_G|_H$.
In this case, $\lambda$ is determined up to normalization by
\begin{equation}\label{Standard Normalization}\textstyle
\int_G f(g) \dd{g} = \int_{G/H} \int_{H} f(xh) \dd{h} \dd\lambda(xH),
\ \ \ f\in C_0(G).
\end{equation}
If \cref{Standard Normalization} holds,
$\lambda$ is said to satisfy the \textit{standard normalization}.
\end{point}

\begin{point}
Henceforth, let $\Gamma < G$ be a lattice subgroup,
and let $\pi: G \to G/\Gamma$ be the quotient map.
By definition of lattice, $\Gamma$ is discrete, and the coset space $G/\Gamma$ 
admits a $G$-invariant Radon probability measure $\lambda$.
Assume that $\ddd\gamma$ is the counting measure on $\Gamma$,
and that $\lambda$ satisfies the standard normalization.
The following lemma is elementary in the study of lattices.
\end{point}

\begin{lem}
Because $G$ admits a lattice subgroup $\Gamma$, it must be unimodular.
\end{lem}
\begin{proof}
Let $K$ be the kernel of the continuous homomorphism $\Delta_G: G \to \Rbb^\times$.
By Weil's theorem, $\Delta_G|_\Gamma = \Delta_\Gamma = 1$, hence $\Gamma \subset K$.
The probability measure $\lambda$ on $G / \Gamma$ pushes forward to a Haar measure on $G / K$.
Since $\Rbb^\times$ has no nontrivial subgroups with finite Haar measure, $K = G$.
\end{proof}

\section*{Results}

\begin{point}
Given $m\in\LIM(\Gamma)$, define $P_m: \LUC(G) \to C(G/\Gamma)$ by
\begin{equation}\textstyle\label{Equation define Pm}
P_mf(x\Gamma) = \inner*{m, \paren*{l_x f}|_\Gamma},
\ \ \ f\in\LUC(G).
\end{equation}
$P_mf$ is well-defined because $m$ is left-invariant,
and it is continuous because $f\in\LUC(G)$.
It satisfies $P_m(l_y f) = l_y (P_m(f))$ because $l_x l_y = l_{yx}$, and it is onto because
\begin{equation}\textstyle\label{Pm onto}
P_m (h\circ \pi) = h,
\ \ \ h\in C(G/\Gamma).
\end{equation}
\end{point}

\begin{point}
Define $\iota: \LIM(\Gamma) \to \TLIM(G)$ by
\begin{equation}\label{Equation Iota}\textstyle
\iota m(f)
= \int_{G/\Gamma} P_m f \dd\lambda,
\ \ \ f\in\LUC(G).
\end{equation}
The map $\iota m$ is linear, positive, unital, and left-invariant because $m$ and $\lambda$ are.
Grosvenor \cite[Theorem 3.2]{Grosvenor1985} shows that $\iota$ is injective,
although the following proof is somewhat simpler.
\end{point}


\begin{lem} \label{Iota is Injective}
$\iota$ is injective.
\end{lem}
\begin{proof}
Since $\Gamma$ is discrete, let $U\subset G$ be a sufficiently small neighborhood of $e$
such that $U^{-1}U\cap \Gamma = \{e\}$.
By Urysohn's lemma, there exists a continuous $h\geq 0$ with support contained in $U$, such that $\|h\|_1=1$.
Define an embedding $\ell_\infty(\Gamma)\into\LUC(G)$ by $\phi\mapsto\Phi$, where
$\Phi(x) = \begin{Bmatrix*}[l]
	0								&	x\not\in U\Gamma\\
	h(x\gamma^{-1}) \phi(\gamma)	&	x \in U\gamma\end{Bmatrix*}$.
For $t\in U$, $\|l_t \Phi - \Phi\| \leq \|l_t h - h\| \cdot \|\phi\|$, which shows $\Phi\in\LUC(G)$.
A direct computation shows $\iota m(\Phi) = m(\phi)$.
\end{proof}

\begin{thm}
If $\Gamma$ is not uniform, $\iota$ is not onto.
\end{thm}
\begin{proof}
Pick compact sets $K, U\subset G/\Gamma$ with $\lambda(K) > 0$
and $K \subset U^\circ$.
Let $\set*{F_\alpha} \subset \Cscr(G)$ be a F\o{}lner net for $G$.
For each $\alpha$, pick $y_\alpha \in G/\Gamma$
outside the compact set $F_\alpha U$.
Thus  $F_\alpha y_\alpha \cap U = \varnothing$.

Let $\delta(y)\in C(G/\Gamma)^*$ denote the point mass at $y$.
Define $\mu_\alpha$ by the weak integral
\begin{equation}\textstyle
\mu_\alpha = \frac{1}{|F_\alpha|}\int_{F_\alpha} \delta(ty_\alpha) \dd{t}.\end{equation}
Thus $\{\mu_\alpha\}$ is a net of means on $C(G/\Gamma)$ converging to $G$-invariance.
Fix an accumulation point $\lambda_2$ of this net, so $\lambda_2$ is a $G$-invariant mean on $C(G/\Gamma)$.
Choose $h\in C(G/\Gamma)$ with $\Onebb_K \leq h \leq \Onebb_{U}$.
By construction, $\lambda_2(h) = 0$.

Pick any $m\in\LIM(\Gamma)$.
With $P_m$ as in \cref{Equation define Pm},
define $\nu\in\TLIM(G)$ by
\begin{equation}\textstyle
\nu(f) = \langle \lambda_2, P_m f \rangle,
\ \ \ f\in\LUC(G).
\end{equation}
By \cref{Pm onto}, $\nu(h \circ \pi) = 0$,
whereas $\iota \mu(h\circ \pi)  = \lambda(h) > 0$ for each $\mu\in\LIM(\Gamma)$.
\end{proof}

\begin{point}
Although amenable non-uniform lattices are hard to come by, they do exist.
Suppose $\set*{q_n}_{n\geq 0}$ is a sequence of prime powers with $\sum_n 1/q_n < \infty$.
Let $F\sub{q_n}$ be the finite field of order $q_n$, with multiplicative group $F^\times\sub{q_n}$.
Form the groups $\Lambda = \bigoplus_n F\sub{q_n}$, $S = \prod_n F\sub{q_n}^\times$,
and $G = \Lambda \rtimes S$.
$G$ is abelian-by-abelian, hence amenable.
\cite[Theorem 1.11]{Bader2019} states that $G$ contains uncountably many pairwise non-commensurable non-uniform lattices.
\end{point}

\begin{point}
Henceforth, let $\Gamma$ be a uniform lattice.
\end{point}

\begin{lem}
$G/\Gamma$ has a transversal $T$ that is a precompact neighborhood of $e$.
\end{lem}
\begin{proof}
Pick any compact neighborhood $U\subset G$ about $e$ small enough that $U^{-1} U \cap \Gamma = \{e\}$.
Thus $\pi|_{tU}$ is injective for any $t\in G$.
$G$ is covered by $\{t U : t\in G\}$, hence $G/\Gamma$ is covered by $\{\pi(tU) : t\in G\}$.
Let $\{\pi(t_1U), \hdots, \pi(t_nU)\}$ be a finite subcover.
Without loss of generality, suppose $t_1 = e$.
Let $T_1 = t_1 U$.
Inductively let $T_{k+1} = T_k \cup [t_{k+1} U \setminus T_k \Gamma]$.
The desired transversal is $T_n$.
\end{proof}

\begin{point}
Suppose $S\subseteq T$.
We have assumed $\lambda$ satisfies the standard normalization, hence
\begin{equation}\label{lambda comes from Haar}\textstyle
|S| = \int_G \Onebb_S
	= \int_{G/\Gamma} \int_\Gamma \Onebb_S(t\gamma) \dd{\gamma} \dd\lambda(t\Gamma)
	= \int_{G/\Gamma} \Onebb_{\pi(S)} \dd\lambda
	= \lambda(\pi(S)).
\end{equation}
In particular, $|T|= 1$.
Let $\#(D)$ denote the cardinality of $D\subset \Gamma$.
Since $G$ is unimodular,
\begin{equation}\label{F vs TF}\textstyle
\#(D) = \sum_{\gamma\in D} |T| =\sum_{\gamma\in D}|T \gamma| = |TD|.
\end{equation}
\cref{lambda comes from Haar} also implies
\begin{equation}\label{Integral over T}\textstyle
\int_{G/\Gamma} F \dd\lambda = \int_T F(\pi(t)) \dd{t}, \ \ \ F\in \Linfty(G/\Gamma).
\end{equation}
When $m\in\LIM(\Gamma)$, we can plug the function
$P_m f$ into \cref{Integral over T} 
to conclude that the following definition of $\iota: \linftyS(\Gamma)\to\LUC(G)^*$ extends
the definition given in \cref{Equation Iota}:
\begin{equation}\label{extension of iota}\textstyle
\iota m (f) = \int_T \inner*{m, \paren*{l_t f}|_\Gamma} \dd{t},
\ \ \ f\in\LUC(G).
\end{equation}
For $D\subset \Gamma$, we observe that
\begin{equation}\label{iota of muF}\textstyle
\iota \mu\sub{D}(f)
= \frac1{\#(D)}\sum_{\gamma\in D} \int_T f(t\gamma) \dd{t}
= \frac{1}{|TD|} \int_{TD} f(t) \dd{t}
= \mu_{TD}(f).
\end{equation}
\end{point}

\begin{lem}\label{Lemma iota is continuous}
$\iota: \linftyS(\Gamma) \to \LUC(G)^*$ as defined in \cref{extension of iota} is $w^*$-to-$w^*$ continuous.
\end{lem}
\begin{proof}
Suppose $\lim_\alpha m_\alpha = 0 \in \linftyS(\Gamma)$.
Pick $f\in \LUC(G)$ and $\epsilon > 0$.
Partition $T$ into precompact sets $\set*{T_k}_{k=1}^n$
such that $\sup_{t,s\in T_k}|f(t) - f(s)| < \epsilon$.
For each $k$, pick $t_k \in T_k$.
Thus
\begin{equation}\textstyle
\big|\iota m_\alpha(f)
- \sum_{k=1}^n |T_k| \: \langle m_\alpha, (l_{t_k} f)|_\Gamma\rangle\big|
< \epsilon.
\end{equation}
This implies  $\lim_\alpha |\iota m_\alpha(f)| \leq \epsilon$.
Since $\epsilon$ was arbitrary, $\lim_\alpha m_\alpha (f) = 0$.
\end{proof}

\begin{lem}\label{Falpha to Dalpha}
Suppose $\{F_\alpha\}$ is a F\o{}lner net for $G$,
and let
\[D_\alpha = \{\gamma\in\Gamma : F_\alpha \cap T\gamma \neq \varnothing\}.\]
Then $\{D_\alpha\}$ is a F\o{}lner net for $\Gamma$,
and $\lim_\alpha \norm{ \mu\sub{F_\alpha} - \mu\sub{T D_\alpha} } = 0$.
\end{lem}
\begin{proof}
Pick $K\in \Cscr(\Gamma)$ and $\epsilon > 0$.
If $\alpha$ is large enough that $F_\alpha$ is $(TKT^{-1}, \epsilon)$-invariant, then
\begin{equation}\textstyle
\#(KD_\alpha \setminus D_\alpha)
= \abs{TKD_\alpha \setminus TD_\alpha}
\leq \abs{TKT^{-1} F_\alpha \setminus F_\alpha}
< \epsilon \abs{F_\alpha}
\leq \epsilon \#(D_\alpha).
\end{equation}
In other words, $D_\alpha$ is $(K, \epsilon)$-invariant.
This shows $\set*{D_\alpha}$ is a F\o{}lner net for $\Gamma$.

If $\alpha$ is large enough that $F_\alpha$ is $(TT^{-1}, \epsilon)$-invariant, then
\begin{equation}\textstyle
	|TD_\alpha \setminus F_\alpha|
	\leq |T T^{-1} F_\alpha \setminus F_\alpha| < \epsilon|F_\alpha|.
\end{equation}
Since $F_\alpha \subseteq TD_\alpha$, we apply \cref{norm mu F} to see
\begin{equation}\textstyle
\norm{\mu\sub{TD_\alpha} - \mu\sub{F_\alpha}}
	= \frac{|TD_\alpha\setminus F_\alpha|}{|TD_\alpha|}
	+ \paren*{\frac1{|F_\alpha|} - \frac1{|TD_\alpha|}}
	< \epsilon + \paren*{1 - \frac1{1+\epsilon}}
	< 2 \epsilon.
\end{equation}
This proves $\lim_\alpha \norm{\mu\sub{TD_\alpha} - \mu\sub{F_\alpha}} = 0$.
\end{proof}

\begin{thm}
If $\Gamma$ is a uniform lattice, $\iota: \LIM(\Gamma) \to \TLIM(G)$ is surjective.
\end{thm}
\begin{proof}
Pick $\nu\in\TLIM(G)$.
By \cref{Every TLIM is limit of Folner net},
$G$ admits a F\o{}lner net $\set*{F_\alpha}$ such that $\nu = \lim_\alpha \mu\sub{F_\alpha}$.
Define $\set*{D_\alpha}$ as in \cref{Falpha to Dalpha}.
Since the unit ball of $\linftyS(\Gamma)$ is compact,
$\{\mu\sub{D_\alpha}\}$ admits a convergent subnet $\{\mu\sub{D_\beta}\}$.
Evidently $m = \lim_\beta \mu\sub{D_\beta} \in \LIM(\Gamma)$.
Applying \cref{iota of muF}, \cref{Lemma iota is continuous}, and \cref{Falpha to Dalpha}
\begin{equation}\textstyle
\iota m = \lim_\beta \iota \mu\sub{D_\beta} = \lim_\beta \mu_{TD_\beta} = \lim_\beta \mu_{F_\beta} = \nu.
\qedhere
\end{equation}
\end{proof}
\chapter{Some remarks on non-topological invariant means}\label{Chapter Open Problems}
In \cref{Banach Rudin Granirer}, we showed that, if the infinite compact group $G$ is amenable-as-discrete,
the Haar integral is not the unique element of $\LIM(G)$.
By \cref{Prop When G is compact}, this is equivalent to saying $G$ admits a non-topological invariant mean.
This is generalized by the following result:

\begin{prop}[{\cite[Theorem 4.1]{Rudin72}}]\label{basic rudin 72 result}
If the nondiscrete group $G$ is amenable-as-discrete, then $\LIM(G) \setminus \TLIM(G)$ is nonempty.
\end{prop}

\begin{rem}
The assumption that $G$ be amenable-as-discrete is not superfluous.
For example, the compact group $\text{SO}(n)$ has a unique left-invariant mean for $n\geq3$.\footnote{
	The famous Banach-Ruziewicz problem asks whether Lebesgue measure
	is the unique finitely-additive $\text{O}(n+1)$-invariant measure on $\Sbb^n$ for $n\geq 2$.
	\textit{(For $n=1$, Banach showed Lebesgue measure is not unique
	by an argument similar to \cref{Banachs theorem}.)}
	A fortiori, it has been shown for $n\geq 2$ that Lebesgue measure is
	the unique finitely-additive $\text{SO}(n+1)$-invariant measure on $\Sbb^n$.
	The same tool used to prove this theorem also proves Haar integral is the unique
	left-invariant mean on $\text{SO}(n+1)$, see \cite[Theorem 3.4.5]{Lubotzky} and the remarks thereafter.}
It is a classical result that $\text{SO}(3)$
contains a subgroup isomorphic to the free group $\Fbb_2$, hence it is not amenable-as-discrete.
If the converse of \cref{basic rudin 72 result} were true,
it would not be so difficult to prove $\text{SO}(3)$ has a unique left-invariant mean!
\end{rem}

\begin{point}
We would like to determine the exact cardinality of the set of non-topological left-invariant means,
under the assumption that $G$ is amenable-as-discrete.
To keep the discussion manageable, we will only try to determine the cardinality of the entire set $\LIM(G)$.
This suffices to determine the cardinality of the subset of non-topological left-invariant means
in cases when $\#(\LIM(G)) > \#(\TLIM(G))$, in particular when $G$ is compact and nondiscrete.
\end{point}

\begin{point}\label{What is Acal}
Let $\Acal \subset \LinftyS(G)$ be the spectrum of $\Linfty(G)$.
Then $G$ acts on $\LinftyS(G)$ by
\[\langle x\mu,\, \phi\rangle = \langle \mu,\, L_x\phi\rangle\ \ 
\text{for}\ \ x\in G,\ \mu\in\Linfty(G),\ \phi\in\Linfty(G).\]
The Gelfand isomorphism\ \ $\widehat{\ }: \Linfty(G) \to C(\Acal)$ is $G$-equivariant:
\[l_x \widehat{\phi}(h)
=\widehat{\phi}(x^{-1}h)
= \langle x^{-1}h,\, \phi\rangle
= \langle h,\, l_x\phi\rangle
= \widehat{l_x \phi}(h).\]
Associate $C(\Acal)^*$ with the set $M(\Acal)$ of (regular Borel) measures on $\Acal$ via \cref{Riesz Kakutani}.
The dual-isomorphism\ \ $\widehat{\ }: \LinftyS(G) \to M(\Acal)$ is defined by
\[\langle \widehat{\mu},\, \widehat{\phi}\rangle = \langle \mu,\, \phi\rangle\ \ \text{for}\ \ 
\mu\in\LinftyS(G),\ \phi\in\Linfty(G).\]
Thus the action of $G$ on $M(\Acal)$ is given by $x\widehat{\mu} = \widehat{x\mu}$.
\end{point}

\begin{point}
$K\subset\Acal$ is said to be \textit{invariant}
if it is closed, nonempty, and invariant under the action of $G$.
Likewise, an ideal $I\normal \Linfty(G)$ is sad to be \textit{invariant}
if it is closed, proper, and invariant under the action of $G$.
For $K\subset \Acal$, define
\[K^\perp = \{f\in\Linfty(G) : (\forall h\in K)\ h(f)=0\}.\]
For $I\normal \Linfty(G)$, define
\[I^\perp = \{h\in \Acal : (\forall f\in I)\ h(f)=0\}.\]
We see $I\mapsto I^\perp$ is a bijection of invariant ideals with invariant subsets of $\Acal$,
with inverse map $K \mapsto K^\perp$.
Define $\LIM(G, K) = \{\mu\in\LIM(G) : \supp \widehat{\mu} \subset K\}$.
Equivalently, $\LIM(G,K)$ is the set of left-invariant means that vanish on the ideal $K^\perp.$
\end{point}

\begin{prop}
If $G$ is amenable-as-discrete and $K\subset \Acal$ is invariant, then $\LIM(G, K)$ is nonempty.
\end{prop}
\begin{proof}
This is an instance of a general result about amenable group actions,
\cite[Theorem 1.14]{paterson-Book}:
If $G$ acts continuously by affine maps a compact convex subset $C$ of a locally convex vector space,
then $C$ contains a $G$-fixed point.
Simply take $C$ to be the set $P(K) \subset M(\Acal)$ of all probability measures supported on $K$.

What we wish to emphasize is that $G$ does \textit{not} act continuously on $\Acal$,
hence it does not act continuously on $M(\Acal)$.
This is why we are forced to consider the action of $G$ as a discrete group.
A classic paper about the unruliness of the action of $G$ on $\Acal$ is \cite{Rudin75}.
\end{proof}

\begin{point}
If $K, J \subset \Acal$ are distinct minimal invariant subsets, they are clearly disjoint.
With this observation, we have two ways to approach the cardinality of $\LIM(G)$:
\\(1) Study the cardinality of the family of minimal invariant subsets of $\Acal$.
\\(2) Study the cardinality of $\LIM(G,K)$ for a minimal invariant subset $K$.
\end{point}

\section*{\texorpdfstring{The cardinality of the minimal invariant subsets of $\Acal$}
{The cardinality of the minimal invariant subsets of A}}

Recall the definitions of $\kappa = \kappa(G)$ and $\mu = \mu(G)$ given in \cref{Defn Cardinal Invariants}.

\begin{prop}[Rosenblatt, {\cite[Theorem 3.5]{Rosenblatt76MathAnn}}]\label{Rosenblatt MathAnn}
If $\mu \geq \Nbb$ and $\kappa \leq \Nbb$, $\Acal$ has at least $2^{2^{\Nbb}}$ minimal invariant subsets.

Unfortunately, the proof of this proposition uses a Baire category-type theorem
which cannot be generalized to yield a better result when $\mu > \Nbb$.
Luckily, in another paper published the same year, Rosenblatt proved \cref{disjoint PP sets}.
\end{prop}

\begin{defn}
Say that a measurable set $A\subset G$ is PP (permanently positive) if
$\big| \bigcap_{x\in F} xA \big| > 0$ for any $F\Subset G$.
For example, open dense subsets of $G$ are PP,
a fact we exploited in the proof of \cref{Banach Rudin Granirer}.

The following definition of generalizes the previous one, allowing for PP sets that are locally measurable,
as in \cref{Defn Linfty}.
\end{defn}

\begin{defn}
For a locally measurable set $B\subset G$, let $\langle B\rangle_G$ denote
the smallest closed invariant ideal in $\Linfty(G)$ containing $\Onebb_B$.
Now $A\subset G$ is said to be PP if $\langle G\setminus A\rangle_G$ is a proper ideal.
\end{defn}

\begin{prop}[{\cite[Theorem 2.2]{Klawe77}}]
If $G$ is infinite discrete, it has a family $\{U_\alpha\}_{\alpha < |G|}$ of disjoint PP sets.
\end{prop}

\begin{prop}[{\cite[Proposition 3.4]{Rosenblatt76JFA}}]\label{disjoint PP sets}
Suppose $\mu\geq \Nbb$ and $\kappa \leq \Nbb$.
If $M\subset G$ is PP, there exists $E\subset M$ such that both $E$ and $M\setminus E$ are PP.

We can make use of this proposition when $\kappa > \Nbb$, as follows.
\end{prop}

\begin{cor}\label{disjoint PP sets in general}
If $\mu\geq \Nbb$, then $G$ has a sequence $\{U_n\}_{n\in\Nbb}$ of disjoint PP sets.
\end{cor}
\begin{proof}
Let $H < G$ be a compactly generated open subgroup.
Apply the previous proposition to produce a sequence $\{V_n\}_{n\in\Nbb}$ of disjoint sets
which are PP as subsets of $H$.
Suppose $X$ is a transversal for $G/H$.
For each $n$, let $U_n = X^{-1}V_n$.
It is easy to see $\{U_n\}_{n\in\Nbb}$ are disjoint and locally measurable.

If $A\subset H$ is PP as a subset of $H$, then $B = X^{-1}A$ is PP as a subset of $G$.
To see this, choose $\{t_1, \ldots, t_n\}\Subset G$.
Say $t_i \in x_i H$, equivalently $t_i x_i^{-1} \in H$.
We see $\bigcap_{i=1}^n t_i B \supset \bigcap_{i=1}^n t_i x_i^{-1} A$,
which has positive measure in $H$.
\end{proof}

\begin{lem}\label{lemma downward directed}
If $\Scal$ is a downward-directed family of PP sets, then
\linebreak $\bigcap_{S\in \Scal}\langle G\setminus S\rangle_G^\perp$
is an invariant subset of $\Acal$.
\end{lem}
\begin{proof}
Pick $\{S_1, \ldots, S_n\}\subset \Scal$.
Since $\Acal$ is compact, it suffices to show
$\bigcap_{i=1}^n\langle G\setminus S_i\rangle_G^\perp$ is nonempty.
Since $\Scal$ is downward-directed, there exists $T\in \Scal$ with $T\subset S_i$ for each $i$.
Thus $G\setminus T \supset G\setminus S_i$,
and $\langle G\setminus T\rangle_G^\perp \subset \langle G\setminus S_i\rangle_G^\perp$.
Since $T$ is PP, $\langle G\setminus T\rangle_G^\perp$ is nonempty.
\end{proof}

The following result seems to be the first demonstration of more than $2^{2^{\Nbb}}$
non-topological left-invariant means on a group.

\begin{thm}\label{theorem product of compact}
Suppose $G = \prod_{\alpha \in \beta} G_\alpha$ is an infinite product of infinite compact groups.
Then $\Acal$ contains $2^\beta$ disjoint invariant subsets.
\end{thm}
\begin{proof}
For each $\alpha\in \beta$, let $\pi_\alpha: G \to G_\alpha$ denote the canonical surjection.
Choose disjoint PP sets $A^\alpha_0, A^\alpha_1 \subset G_\alpha$.
Given $F\Subset \beta$ and $\epsilon\in\{0,1\}^\beta$,
define $S^F_\epsilon = \bigcap_{\alpha\in F} \pi_\alpha^{-1}\big(A^\alpha_{\epsilon(\alpha)}\big)$.
I claim $\Scal_\epsilon = \big\{S^F_\epsilon : F\Subset \beta\big\}$ is a downward-directed family of PP sets.
$\Scal_\epsilon$ is obviously downward-directed, since $S^F_\epsilon \cap S^E_\epsilon = S^{F\cup E}_\epsilon$,
so it suffices to show that each $S^F_\epsilon$ is PP.
It is easy enough to see
\[\big|S^F_\epsilon\big|
= \big|\bigcap_{\alpha \in F} \pi_\alpha^{-1}\big(A^\alpha_{\epsilon(\alpha)}\big) \big|
= \prod_{\alpha \in F}\big| A^\alpha_{\epsilon(\alpha)} \big|_\alpha,\]
where $|\cdot|_\alpha$ denotes the Haar measure in $G_\alpha$.
Likewise, for $X\Subset G$,
\[\big| \bigcap_{x\in X} x S_\epsilon^F \big|
= \prod_{\alpha\in F}\big| \bigcap_{x\in X} x A^\alpha_{\epsilon(\alpha)} \big|_\alpha > 0.\]
For each $\epsilon \in \{0,1\}^\beta$, \cref{lemma downward directed} guarantees that
$\Acal_\epsilon = \bigcap_{S\in \Scal_\epsilon} \langle G\setminus S\rangle_G^\perp$
is an invariant subset of $\Acal$.
If $\delta\neq \epsilon$, choose $S\in \Scal_{\epsilon}$ and $T\in \Scal_\delta$
with $S\cap T = \varnothing$.
Thus $\Acal_\delta \cap \Acal_\epsilon
\subset \langle G\setminus S\rangle_G^\perp \cap \langle G \setminus T\rangle_G^\perp
= \langle G\rangle_G^\perp = \varnothing$.
We conclude that $\big\{\Acal_\epsilon : \epsilon \in \{0,1\}^\beta\big\}$
is a family of disjoint invariant subsets.
\end{proof}

Suppose $\beta = \beta^{\Nbb}$,
and $G = \prod_{\alpha\in\beta} G_\alpha$ is a product of infinite compact metrizable groups.
Then the above theorem yields $2^\beta$ closed disjoint $G$-orbits in $\Acal$.
Since $\mu(G) = \beta$, \cref{Cardinality of B in general}
tells us the cardinality of $\Acal$ itself is $2^\beta$.
On the other hand, if $G = (\Zbb / 2\Zbb)^{\Nbb}$,
the above theorem yields only $2^{\Nbb}$ subsets,
whereas \cref{Rosenblatt MathAnn} guarantees $2^{2^{\Nbb}}$.

When $G$ is discrete, the results of \cite{ChouExactCard}
show that $\Acal$ has $\#(\Acal) = 2^{2^{|G|}}$ disjoint invariant subsets.
In summary, in every case where the exact cardinality of the minimal invariant subsets of $\Acal$ has been determined,
it has turned out to be $\#(\Acal)$.

\section*{\texorpdfstring
{The size of $\LIM(G,K)$}
{The size of LIM(G,K)}}

\begin{prop}[{\cite{Talagrand80}}]\label{Talagrands Theorem}
If $G$ is nondiscrete and amenable-as-discrete,
and $K\subset \Acal$ is invariant,
then $\LIM(G,K)$ is infinite-dimensional.\footnote{
	To quote the Talagrand of 1980:
\textit{Nous n'\'{e}crirons la preuve que dans le cas o\`{u} $G$ est compact;
les id\'{e}es essentielles sont les m\^{e}mes quand $G$ n'est pas compact,
et les quelques complications qui s'introduisent sont purement techniques.}

To quote the Talagrand of 2021:
\textit{The claim of the paper that the non-compact case differs from the compact one
only by ``technical details'' does not seem likely to me.
I have no idea what I had in mind at the time.
Certainly I would not have written that without precise views on how to do this,
but I don't know what they could have been.}

The proof almost certainly relies on the same structure theorem for locally compact groups as
Talagrand's monumentally difficult \cite{Talagrand79}.
This paper arguably represents the turning point for invariant means, from a subject of mainstream research
to a niche for highly specialized topological results.
Quoting Talagrand once more:
\textit{The only impact on mathematics it had was a review in the Math reviews
saying it had the distinction of being the most complicated paper ever written on the subject!
I simply had no mathematical taste at the time, and I now bitterly regret not having used my energy more wisely.}
}
\end{prop}

\begin{point}
When $G$ is countable, it was shown in \cite{ChouGeometric} that $\LIM(G,K)$ has no exposed points.
This result was extended to $\sigma$-compact groups in \cite{Miao1998}.
When $G$ is uncountable discrete,
it was shown in \cite{Yang1986} that $\LIM(G,\beta G)$ has no exposed points.
The following theorem combined with \cref{Talagrands Theorem} shows that
$\LIM(G,K)$ has no exposed points when $G$ is nondiscrete and amenable-as-discrete.
\end{point}

\begin{thm}
Suppose the dimension of $\LIM(G, K)$ is greater than one.
Then it has no exposed points.
\end{thm}
\begin{proof}
Let $I = K^\perp$.
Let $H$ be the closed span of $\{f - l_x f : x\in G,\ f\in\Linfty(G)\}$.
A positive functional $\mu\in\LinftyS(G)$ is in $\LIM(G,K)$ if and only if it satisfies the following:
\begin{enumerate}
	\item $\mu = 0$ on $H$. (Thus $\mu$ is left-invariant.)
	\item $\mu = 0$ on $I$. (Thus $\supp \mu \subset K$.)
	\item $\mu(1) = 1$.
\end{enumerate}
Let $J$ be the closure of $H + I$.
Since $\LIM(G,K)$ is nonempty, we know $J \cap \Cbb = \{0\}$.
If $J + \Cbb = \Linfty(G)$, then there is a unique element $\mu\in \LIM(G, K)$.
Likewise, if $J + \Cbb + \Cbb f = \Linfty(G)$ for some $f$, then $\LIM(G, K)$ is one-dimensional.
By hypothesis, this is not the case.

Suppose $\mu\in \LIM(G,K)$ is exposed by $f\in \Linfty(G)$.
Since $J + \Cbb + \Cbb f$ is a proper closed subspace of $\Linfty(G)$, there exists $g\geq 0$ outside of it.
Define $\nu$ on $J + \Cbb + \Cbb f + \Cbb g$ by
\begin{enumerate}
	\item $\nu = 0$ on $J$.
	\item $\nu(1) = 1$.
	\item $\nu(f) = \mu(f)$.
	\item $\nu(g)$ is any number in $[0,\, \|g\|_\infty]$ other than $\mu(g)$.
\end{enumerate}
By Hahn-Banach, extend $\nu$ to a norm-one functional on $\Linfty(G)$.
Thus $\nu\in \LIM(G,K)$.
Since $\nu\neq \mu$ but $\nu(f) = \mu(f)$, we see $f$ does not expose $\mu$ after all.
\end{proof}
\bibliographystyle{amsalpha}
\bibliography{myBib}
\end{document}